\documentclass[a4paper, 11pt]{article}

\usepackage{fullpage}

\usepackage[english]{babel}
\usepackage[T1]{fontenc}
\usepackage[utf8]{inputenc}
\usepackage{amsmath, amsfonts, amssymb, amsthm}
\usepackage{mathrsfs}
\usepackage{mathtools}
\usepackage{bm}
\usepackage[mathscr]{eucal}
\usepackage{authblk}
\usepackage{xfrac}
\usepackage[symbol]{footmisc}

\usepackage{booktabs}
\usepackage[colorlinks=true, linkcolor=blue, citecolor=red]{hyperref}
\usepackage{enumitem}
\usepackage{graphicx}
\usepackage{subcaption}
\usepackage[toc,page]{appendix}

\usepackage[linesnumbered,ruled,vlined]{algorithm2e}
\SetAlFnt{\small}
\usepackage[backend=biber, style=nature, sortcites=ynt, terseinits=true, maxbibnames=500, doi=true, sorting=nyt]{biblatex}

\DeclareFieldFormat{pages}{\mkcomprange{#1}}
\DeclareFieldFormat{pmlr}{%
  \mkbibacro{PMLR}\addcolon\space
  \ifhyperref
    {\href{https://proceedings.mlr.press/#1}{\nolinkurl{#1}}}
    {\nolinkurl{#1}}}

\DeclareFieldAlias{eprint:pmlr}{pmlr}
\DeclareFieldAlias{eprint:PMLR}{eprint:pmlr}

\DeclareFieldFormat{hal}{%
  \mkbibacro{HAL}\addcolon\space
  \ifhyperref
    {\href{https://hal.archives-ouvertes.fr/#1}{\nolinkurl{#1}}}
    {\nolinkurl{#1}}}

\DeclareFieldAlias{eprint:hal}{hal}
\DeclareFieldAlias{eprint:HAL}{eprint:hal}

\renewbibmacro*{eprint}{%
  \printfield{pmlr}%
  \printfield{hal}%
  \newunit\newblock
  \iffieldundef{eprinttype}
    {\printfield{eprint}}
    {\printfield[eprint:\strfield{eprinttype}]{eprint}}
}

\newcommand{\Real}{\mathbb{R}}
\newcommand{\N}{\mathbb{N}}
\newcommand{\Z}{\mathbb{Z}}

\newcommand{\Tensor}[1]{{\mathscr{#1}}}
\newcommand{\Matrix}[1]{\mathsf{#1}}

\newcommand{\ball}[1][]{#1{{\mathbb{B}}}}
\newcommand{\ballrx}[3][]{\ball[#1]_{#2} (#3)}

\newcommand{\set}[3][]{#1\{ #2~:~#3 #1\}}

\newcommand{\dist}[3][]{\mathrm{dist} #1( #2, #3 #1)}
\newcommand{\Proj}[3][]{\Pi_{#2} #1( #3 #1)}

\newcommand{\NormalSup}[4][]{N^{#2}_{#3} #1( #4 #1)}

\newcommand{\NormalProx}[3][]{\NormalSup[#1]{\mathrm{prox}}{#2}{#3}}
\newcommand{\NormalLim}[3][]{\NormalSup[#1]{\mathrm{lim}}{#2}{#3}}

\newcommand{\interior}{\mathrm{int}}
\newcommand{\boundary}{\mathrm{bdr}}

\newcommand{\norm}[2][]{#1| #2 #1|}
\newcommand{\Norm}[3][]{#1\| #2 #1\|_{\mathrm{#3}}}
\newcommand{\dotp}[3][]{#1\langle #2, #3 #1\rangle}
\newcommand{\Dotp}[4][]{#1\langle #2, #3 #1\rangle_{\mathrm{#4}}}

\newcommand{\rank}[2][]{\mathrm{rank} #1( #2 #1)}
\newcommand{\ttrank}[2][]{\mathrm{rank_{TT}} #1( #2 #1)}

\newcommand{\Expectation}{\mathbb{E}}

\newtheorem{theorem}{Theorem}[section]
\newtheorem{lemma}[theorem]{Lemma}
\newtheorem{corollary}[theorem]{Corollary}

\newtheorem{remark}[theorem]{Remark}

\renewcommand{\thefootnote}{\fnsymbol{footnote}}

\title{Quasioptimal alternating projections and their use in low-rank approximation of matrices and tensors\let\thefootnote\relax\footnotetext{This research was funded in whole by the Austrian Science Fund (FWF) \href{https://doi.org/10.55776/F65}{10.55776/F65}. For open access purposes, the author has applied a CC BY public copyright license to any author-accepted manuscript version arising from this~submission.}}
\author{Stanislav Budzinskiy\footnote{Faculty of Mathematics, University of Vienna, Kolingasse 14-16, 1090 Vienna, Austria (\href{mailto:stanislav.budzinskiy@univie.ac.at}{stanislav.budzinskiy\allowbreak@univie.ac.at}).}}
\date{}

\begin{document}
\maketitle

\renewcommand{\thefootnote}{\arabic{footnote}}
\setcounter{footnote}{0}

\begin{abstract}
We study the convergence of specific inexact alternating projections for two non-convex sets in a Euclidean space. The $\sigma$-quasioptimal metric projection ($\sigma \geq 1$) of a point $x$ onto a set $A$ consists of points in $A$ the distance to which is at most $\sigma$ times larger than the minimal distance $\dist{x}{A}$. We prove that quasioptimal alternating projections, when one or both projections are quasioptimal, converge locally and linearly for super-regular sets with transversal intersection. The theory is motivated by the successful application of alternating projections to low-rank matrix and tensor approximation. We focus on two problems---nonnegative low-rank approximation and low-rank approximation in the maximum norm---and develop fast alternating-projection algorithms for matrices and tensor trains based on cross approximation and acceleration techniques. The numerical experiments confirm that the proposed methods are efficient and suggest that they can be used to regularise various low-rank computational routines.
\end{abstract}



\begin{refsection}
\section{Introduction}

\subsection{Alternating projections}
A recurring problem in mathematics is to find a point in a subset $A$ of a Euclidean space $H$. Consider a simple and illustrative example: solve a system of linear equations. When the matrix of the system is square and invertible, the solution set is a singleton. For an underdetermined system with a full-rank matrix, the solutions form an affine subspace, and we typically choose the (unique) solution of the smallest norm.

This is a particular case of a so-called \emph{metric projection} of a point $x \in H$ onto a set:
\begin{equation}
\label{eq:mp}
    \Proj{A}{x} = \set[\big]{a \in A}{\norm{x - a} = \inf\nolimits_{\hat{a} \in A} \norm{x - \hat{a}}}.
\end{equation}
Indeed, the least-norm solution to a system of linear equations corresponds to the metric projection of the zero vector onto the affine subspace of all solutions. Moreover, this solution can be expressed in closed form via the QR decomposition of the matrix.

For more complex sets, it is no longer evident how to \emph{represent} the metric projection. Suppose we want to find a point in the intersection of two `simple' sets. Von Neumann proved \cite[Thm.~13.7]{von1950functional} that if $A$ and $B$ are linear subspaces then the sequence of \emph{alternating projections}~(AP)
\begin{equation*}
    b_{n+1} \in \Proj{B}{a_n}, \quad a_{n+1} \in \Proj{A}{b_{n+1}}, \quad n \in \N_0,
\end{equation*}
started with $a_0 \in \Proj{A}{x}$ converges to $\Proj{A \cap B}{x}$ for every $x \in H$. The global convergence to $A \cap B$, but not to $\Proj{A \cap B}{x}$, persists for closed convex sets \cite{cheney1959proximity, bregman1965method}, and there is local convergence to $A \cap B$ for a class of closed non-convex sets \cite{lewis2009local}.

This basic framework finds numerous practical applications in systems of linear inequalities \cite{agmon1954relaxation, motzkin1954relaxation}, overdetermined systems of linear equations \cite{strohmer2009randomized}, design of frames \cite{tropp2005designing}, phase retrieval \cite{bauschke2002phase}, and finance \cite{higham2002computing}, for example.

\subsection{Low-rank approximation with constraints}
Our practical interest is in low-rank approximation of matrices and higher-order tensors (see, e.g., \cite{ballard2025tensor}) with constraints \cite{zhang2003optimal, grubivsic2007efficient, grussler2018low}. The best low-rank approximation of a matrix in the Frobenius norm is given by its truncated singular value decomposition (SVD), and we can impose additional constraints with the AP \cite{chu2003structured}. For tensors, there are no general ways to represent even the best rank-one approximation \cite{de2000best, zhang2001rank}, and the minimisation of the Euclidean distance has a complicated optimisation landscape.

It is still possible to use the AP to compute constrained low-rank approximations for tensors in subspace-based formats such as Tucker and tensor train \cite{ballard2025tensor}. Indeed, the set of order-$d$ tensors with low Tucker ranks is the intersection of $d$ sets of low-rank matrices corresponding to the $d$ unfoldings. Therefore, the projections can be alternated between the $d$ low-rank sets and the constraint set; this approach was taken in \cite{jiang2023nonnegative} to compute low-rank nonnegative approximations. But there is a drawback: we would never actually encounter a low-rank tensor during such iterations, it is reached only in the limit. This is in contrast to the matrix case, where a sequence of low-rank matrices converges to the constraint set and the iterations can be halted once the constraints are satisfied to a `sufficient extent'.

While it is infeasible to compute the metric projection of an order-$d$ tensor onto the set of tensors with low Tucker ranks, the higher-order SVD (HOSVD, \cite{de2000multilinear}) method constructs a \emph{good} approximation that is at most $\sqrt{d}$ times worse than the best approximation. In \cite{sultonov2023low}, the HOSVD was used in place of the metric projection to compute low-rank nonnegative approximations of tensors in the AP fashion. Similar to the matrix case, this approach produces a sequence of low-rank tensors that, according to the numerical experiments, converges to the nonnegative~orthant.

\subsection{Inexact metric projections}
The above example of low-rank tensor approximation highlights the need to generalise the notion of metric projections in order to admit points that might not attain the minimal distance, but are within a constant factor from it.

This generalisation is also inevitable once we begin to consider the practical side of computations: every projection we compute numerically is inherently \emph{inexact}. There are three main sources of inexactness: the floating-point arithmetic, the termination of iterative algorithms, and the algorithms that are not designed to be exact. 

In low-rank tensor approximation, the three aspects act in unison. At the bottom level, the basic operations of linear algebra are subject to round-off errors in the floating-point arithmetic \cite{higham2002accuracy}. They are then used as building blocks in the numerical algorithms for the SVD, which are iterative and run until the tolerable error is reached \cite{demmel1997applied}. Finally, algorithms similar to the HOSVD use the SVD to compute \emph{quasioptimal} low-rank tensor approximations.

\subsection{Contributions: Brief version}
For $\sigma \geq 1$, we define the \emph{$\sigma$-quasioptimal metric projection} of a point $x \in H$ onto $A$ as
\begin{equation}
\label{eq:qmp}
    \Proj{A}{x; \sigma} = \set{a \in A}{\norm{x-a} \leq \sigma \cdot \inf\nolimits_{\hat{a} \in A} \norm{x - \hat{a}} }.
\end{equation}
The HOSVD yields an element of the $\sqrt{d}$-quasioptimal metric projection onto the set of tensors with low Tucker ranks (assuming that the SVD is exact), and the definition coincides with the (optimal) metric projection $\Proj{A}{x}$ when $\sigma = 1$.

We propose the method of \emph{quasioptimal alternating projections} (QAP) with quasioptimality constants $\sigma_A \geq 1$ and $\sigma_B \geq 1$ and the initial condition $a_0 \in A$:
\begin{equation}
\label{eq:qap}
    b_{n + 1} \in \Proj{B}{a_n; \sigma_B}, \quad a_{n + 1} \in \Proj{A}{b_{n + 1}; \sigma_A}, \quad n \in \N_0.
\end{equation}
Following the ideas of \cite{lewis2009local}, we prove local linear convergence of the sequence to $A \cap B$ when $\sigma_A \geq 1$ with $\sigma_B = 1$ (Theorem~\ref{theorem:1qoptimal_convergence} and Corollary~\ref{corollary:1qoptimal_convergence_2sr}) and when $\sigma_A > 1$ with $\sigma_B > 1$ (Theorem~\ref{theorem:2qoptimal_convergence}). In both cases, we estimate the rate of convergence and the admissible values of $\sigma_A$ and $\sigma_B$. Our results also show that the sequence \eqref{eq:qap} converges to the quasioptimal metric projections of $a_0$ onto $B$ and $A \cap B$; see also Theorem~\ref{theorem:ap_as_qmp}.

Consider an example with $A$ and $B$ being the coordinate axes in the plane (Fig.~\ref{fig:2qopt_perp}). The Pythagorean theorem ensures that 
\begin{equation*}
    \frac{\norm{a_1}}{\norm{a_0}} \leq \sqrt{\sigma_A^2 -1} \cdot \sqrt{\sigma_B^2 - 1}
\end{equation*}
for every $b_1 \in \Proj{B}{a_0; \sigma_B}$ and $a_1 \in \Proj{A}{b_1; \sigma_A}$. Hence, the iterations converge if the right-hand side is below one. Theorem~\ref{theorem:2qoptimal_convergence} captures this behaviour and extends it to non-convex sets with a non-orthogonal intersection. Note that the method might still converge even if the right-hand side is greater than one: everything depends on the particular choice of $b_{n + 1} \in \Proj{B}{a_n; \sigma_B}$ and $a_{n + 1} \in \Proj{A}{b_{n + 1}; \sigma_A}$. 

\begin{figure}[!t]
\centering
    \includegraphics[width=0.9\linewidth]{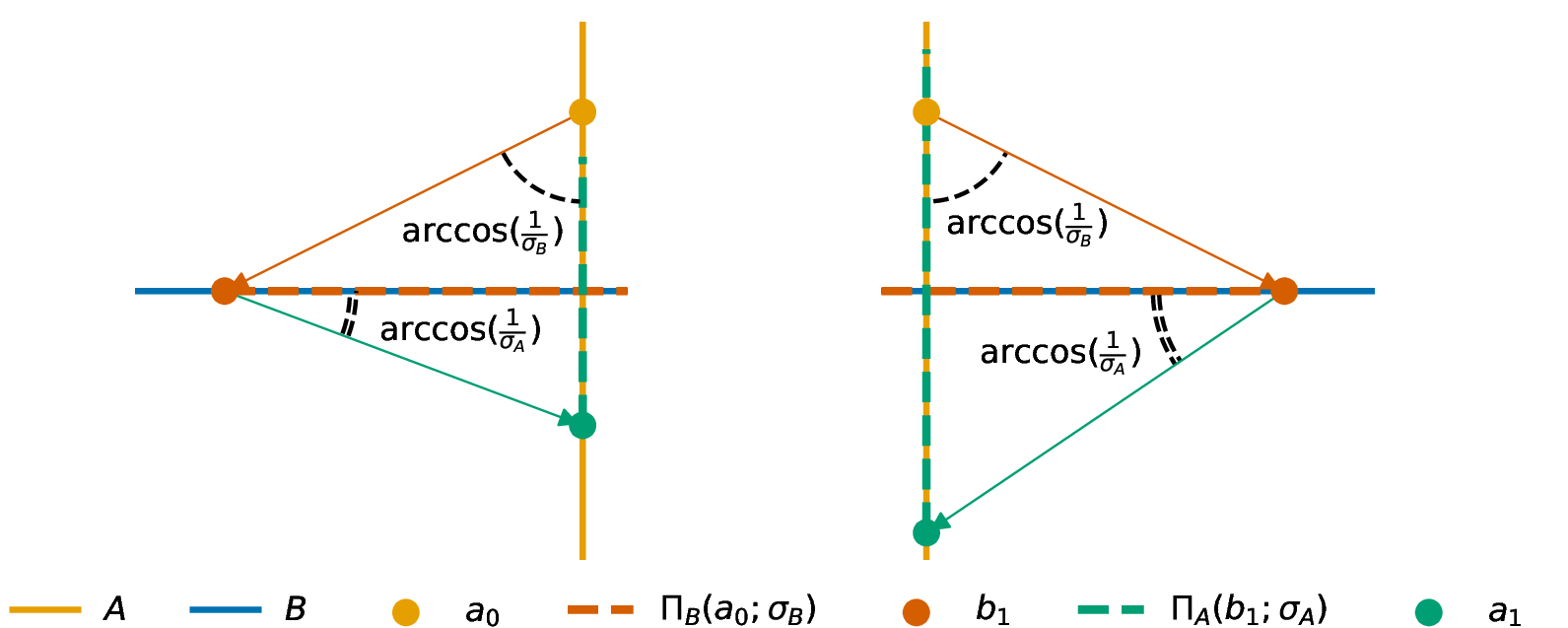}
\caption{Quasioptimal alternating projections for two perpendicular lines in the plane converge when $\sqrt{\sigma_A^2 -1} \cdot \sqrt{\sigma_B^2 - 1} < 1$ and can diverge otherwise.}
\label{fig:2qopt_perp}
\end{figure}

In numerical experiments, we compute constrained low-rank tensor approximations with the QAP. First, we use them to impose nonnegativity on low-rank approximations. Second, we develop a method to obtain low-rank approximations with small \emph{entrywise} errors and employ it to approximate orthogonal matrices.

\subsection{Outline}

Our article could be of interest to two different communities. To reach out to both, we deliberately choose to make the presentation detailed and include the necessary background material on non-smooth optimisation and low-rank approximation.

The first, theoretical part of the paper addresses the convergence of the QAP \eqref{eq:qap}. In Section~\ref{sec:prelim}, we collect some preliminary results that are useful in the convergence analysis of the (Q)AP. We then revisit the main contributions of the paper in more details in Section~\ref{sec:contributions} and compare them with the existing literature in Section~\ref{sec:related}. The properties of quasioptimal metric projections are studied in Section~\ref{sec:qoptimal}. The following Section~\ref{sec:convergence} is devoted to the QAP and their local convergence for non-convex sets.

The second, numerical part considers two applications of the QAP. We present an introduction low-rank matrix and tensor approximation in Subsections~\ref{subsec:lrm}-\ref{subsec:qoptimal_lowrank}. In Subsection~\ref{subsec:nonnegative}, we continue the investigation of \cite{jiang2023nonnegative, sultonov2023low} on using the (Q)AP to compute nonnegative low-rank approximation. We introduce a new application of the (Q)AP to low-rank approximation in entrywise norms in Subsection~\ref{subsec:maximum}.

We formulate possible directions of further research into the QAP in Section~\ref{sec:conclusion}. The appendix contains some auxiliary results related to our theory. Additional numerical experiments and an overview of the optimality properties of low-rank approximation methods can be found in the supplementary material.
\section{Preliminaries and notation}
\label{sec:prelim}

Throughout the text, $H$ is a Euclidean space with inner product $\dotp{\cdot}{\cdot}$ and norm $\norm{\cdot}$. The open unit ball at the origin and the open ball of radius $r > 0$ centred at $x \in H$ are written as $\ball = \set{x \in H}{\norm{x} < 1}$ and $\ballrx{r}{x} = x + r \ball$. We say that a set $\mathcal{V} \subseteq H$ is a neighbourhood of $x \in H$ if $\mathcal{V}$ is open and contains $x$. For a set $A \subseteq H$, let $\overline{A}$ denote its closure, $\interior A$ its interior and $\boundary A$ its boundary. A non-empty set $A$ is called (i)~\emph{locally closed} at $\hat{a} \in A$ if there exists a neighbourhood $\mathcal{V} \subseteq H$ of $\hat{a}$ such that $A \cap \mathcal{V}$ is closed in $\mathcal{V}$; (ii)~\emph{convex} if for every $a_1, a_2 \in A$ and $\alpha \in (0,1)$ their combination $\alpha a_1 + (1 - \alpha) a_2$ is in $A$; (iii)~a \emph{cone} if for every $a \in A$ and $\alpha \geq 0$ their product $\alpha a$ is in $A$.

\subsection{Metric projections}
For a non-empty set $A \subseteq H$ and a point $x \in H$, we define the \textit{distance} from $x$ to $A$,
\begin{equation*}
    \dist{x}{A} = \inf\nolimits_{a \in A} \norm{x - a},
\end{equation*}
and rewrite the definition of the (optimal) metric projection \eqref{eq:mp} as
\begin{equation*}
    \Proj{A}{x} = \set{a \in A}{\norm{x - a} = \dist{x}{A}}.
\end{equation*}
For $x \in A$, the metric projection is a singleton $\Proj{A}{x} = \{ x \}$. In general, however, there can be points $x \in H$ such that $\Proj{A}{x}$ is empty. The following lemma ties the existence of optimal metric projections with the local topological properties of the set.

\begin{lemma}
\label{lemma:locally_closed_local_proj}
Let $A \subseteq H$ be non-empty and $\hat{a} \in A$.
\begin{enumerate}
    \item Let $\mathcal{V} \subseteq H$ be a neighbourhood of $\hat{a}$ such that every $x \in \mathcal{V}$ has a non-empty optimal metric projection $\Proj{A}{x}$. Then $A \cap \mathcal{V}$ is closed in $\mathcal{V}$.
    \item Let $A \cap \ballrx{r}{\hat{a}}$ be closed in $\ballrx{r}{\hat{a}}$. Then for each $x \in \ballrx{r/2}{\hat{a}}$ the optimal metric projection $\Proj{A}{x}$ is non-empty and $\norm{\pi - \hat{a}} \leq 2 \norm{x - \hat{a}} < r$ for every $\pi \in \Proj{A}{x}$.
\end{enumerate}
\end{lemma}
\begin{proof}
This is a folklore result; we prove it in Appendix~\ref{appendix:implicit_function_theorem} for completeness.
\end{proof}

\subsection{Normal cones}
\label{subsec:cones}
The concept of an `outward' normal to a set is one of the central notions of variational analysis, which is used to formulate optimality criteria for constrained optimisation problems. From the perspective of the AP, we can learn a lot about the relative positions of two sets by comparing their specific `outward' normal directions.

Let $A \subseteq H$ be non-empty and $\hat{a} \in A$. We say that $v \in H$ is a \emph{proximal normal} to $A$ at $\hat{a}$ if there exists $\tau > 0$ such that $\hat{a} \in \Proj{A}{\hat{a} + \tau v}$. Let $\NormalProx{A}{\hat{a}}$ denote the collection of all proximal normals to $A$ at $\hat{a}$: 
\begin{equation*}
    \NormalProx{A}{\hat{a}} = \set{v \in H}{\hat{a} \in \Proj{A}{\hat{a} + \tau v} \text{ for some } \tau > 0}.
\end{equation*}
The set $\NormalProx{A}{\hat{a}}$ is a convex cone, known as the \emph{proximal normal cone} to $A$ at $\hat{a}$. If $\Proj{A}{x}$ is non-empty for some $x \in H$, then $x - \pi \in \NormalProx{A}{\pi}$ for each $\pi \in \Proj{A}{x}$.

Assume, in addition, that $A$ is locally closed at $\hat{a}$. Then $v \in H$ is a \emph{limiting normal} to $A$ at $\hat{a}$ if there exists a sequence of points $(a_n) \subset A$ and a corresponding sequence of proximal normals $(v_n) \subset H$ with $v_n \in \NormalProx{A}{a_n}$ such that $a_n \to \hat{a}$ and $v_n \to v$ as $n \to \infty$. The set $\NormalLim{A}{\hat{a}}$ of all limiting normals to $A$ at $\hat{a}$ is a closed cone and is called the \emph{limiting normal cone} to $A$ at $\hat{a}$.

Clearly, $\NormalProx{A}{\hat{a}} \subseteq \NormalLim{A}{\hat{a}}$. It can also be shown that we can replace proximal normals with limiting normals in the definition of the limiting normal cone:
\begin{equation}
\label{eq:limiting_normal_cone_semicontinuity}
    \NormalLim{A}{\hat{a}} = \set{v \in H}{\text{there are } A \ni a_n \to \hat{a}\text{ and }\NormalLim{A}{a_n} \ni v_n \to v \text{ as } n \to \infty}.
\end{equation}
Another useful property of the limiting normal cone is how it distinguishes between interior and boundary points of the set: $\NormalLim{A}{\hat{a}} = \{ 0 \}$ if and only if $\hat{a} \in \interior A$. Find a more in-depth discussion of proximal and limiting normal cones in \cite{mordukhovich2006variational, rockafellar1998variational}.

\subsection{Super-regular sets}
\label{subsec:superregular}
When we think of normal cones, it is customary to imagine them pointing in the `opposite' direction from the set; indeed, this is what happens for smooth manifolds~and convex sets. In general, this picture can be misleading as a simple example shows.

\begin{figure}[!b]
\centering
    \includegraphics[width=0.9\linewidth]{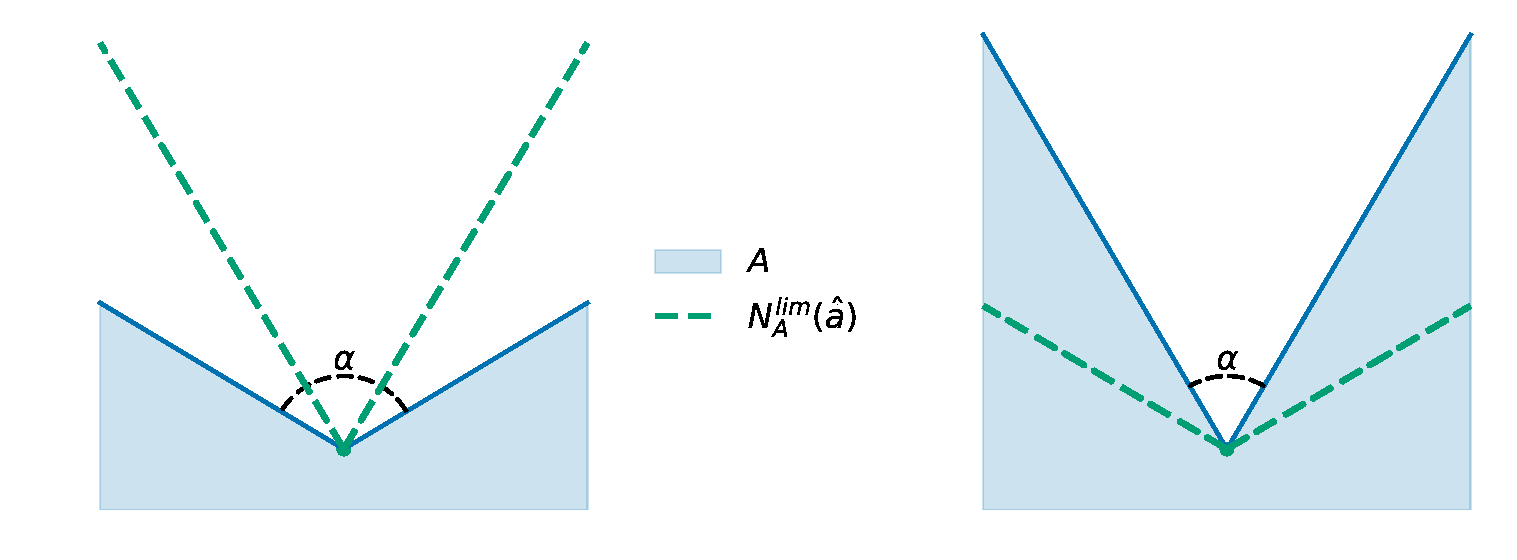}
\caption{Limiting normal cones can point `into' the set.}
\label{fig:non_super_regular}
\end{figure}

Consider a closed concave set $A \subset \Real^2$ with an angle of size $\alpha$ and let $\hat{a}$ be its vertex (see Fig.~\ref{fig:non_super_regular}). When $\alpha > \frac{\pi}{2}$, the limiting normal cone $\NormalLim{A}{\hat{a}}$ does point `away' from $A$, but it still forms acute angles with some of the inward-looking directions. For smaller angles $\alpha \leq \frac{\pi}{2}$, the limiting normals themselves point `into' $A$. 

\begin{remark}
In addition, this example illustrates that the inclusion of normal cones $\NormalProx{A}{\hat{a}} \subset \NormalLim{A}{\hat{a}}$ can be strict, since $\NormalProx{A}{\hat{a}} = \{ 0 \}$.
\end{remark}

A class of sets that do not have inward-looking normals was introduced in \cite{lewis2009local}. Let $A \subseteq H$ be non-empty and locally closed at $\hat{a} \in A$. The set $A$ is called \emph{super-regular} at $\hat{a}$ if for every $\varepsilon > 0$ there exists $\delta > 0$ such that $A \cap \ballrx{\delta}{\hat{a}}$ is closed in $\ballrx{\delta}{\hat{a}}$ and
\begin{equation}
\label{eq:sr_normal_angle}
    \dotp{v}{a' - a} \leq \varepsilon \norm{v} \cdot \norm{a' - a} \quad  \text{for all} \quad a,a' \in A \cap \ballrx{\delta}{\hat{a}} \quad \text{and} \quad v \in \NormalLim{A}{a}.
\end{equation}
Essentially, a set is super-regular at a point if, locally, each limiting normal forms an `almost' obtuse angle with every direction that points into the set. The super-regularity property is shared, among others, by smooth manifolds and convex sets.

In general, the triangle inequality ensures that $\norm{\pi - a} \leq 2 \norm{x - a}$ for $\pi \in \Proj{A}{x}$. When the set is super-regular, the optimal metric projections are nearly \textit{non-expansive}.

\begin{lemma}
\label{lemma:superregular_proj_nonexpansove}
Let $A \subseteq H$ be non-empty and super-regular at $\hat{a} \in A$. Then for every $\varepsilon > 0$ there exists $\delta > 0$ such that for every $x \in \ballrx{\delta/2}{\hat{a}}$ and $\pi \in \Proj{A}{x}$ it holds that
\begin{equation*}
    \norm{\pi - a} \leq (1 + \varepsilon) \cdot \norm{x - a} \quad \text{for every} \quad a \in A \cap \ballrx{\delta}{\hat{a}}.
\end{equation*}
\end{lemma}
\begin{proof}
See \cite[Thm.~2.14]{hesse2013nonconvex}.
\end{proof}

\subsection{Transversal intersection of sets}
\label{subsec:transversal}
Whether the AP can find a common point of two sets, depends on how the sets are \textit{oriented} relative to each other. Consider an example in Fig.~\ref{fig:non_transversal}, where $A$ is a piecewise linear curve and $B$ is the horizontal axis; they intersect at the origin. No matter the starting point and the order of the projections, the AP either get stuck in a loop or converge in one step. The latter takes place for the starting points on the vertical axis: all of them when we begin with $B$, and those in the lower halfspace when we begin with $A$. Therefore, the AP fail to produce the common point of $A$ and $B$ for \emph{almost every} starting point.

\begin{figure}[ht]
\centering
    \includegraphics[width=0.9\linewidth]{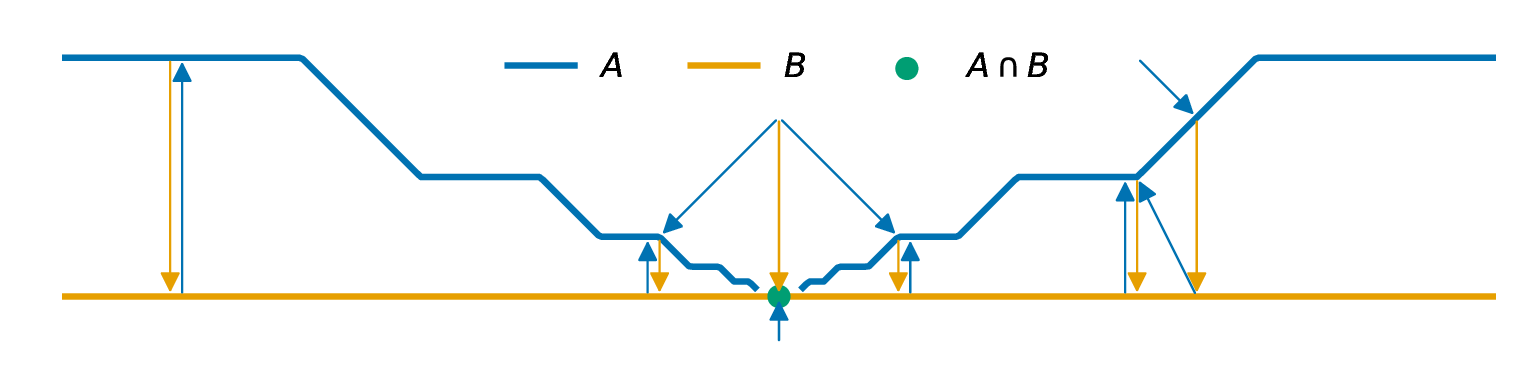}
\caption{For certain pairs of sets $A$ and $B$, the alternating projections do not converge to $A \cap B$ for almost every starting point.}
\label{fig:non_transversal}
\end{figure}

The following definition aims to rule out such pathological cases. Consider a pair of non-empty sets $A, B \subseteq H$ with a common point $\bar{x} \in A \cap B$, where both sets are locally closed. We say that $A$ and $B$ \emph{intersect transversally} at $\bar{x}$ if 
\begin{equation*}
    -\NormalLim{A}{\bar{x}} \cap \NormalLim{B}{\bar{x}} = \{ 0 \}.
\end{equation*}
While we follow the terminology of \cite{drusvyatskiy2015transversality}, this property has other names in the literature, such as \textit{strong regularity} \cite{hesse2013nonconvex} and \textit{linearly regular intersection} \cite{lewis2009local}. See also \cite{kruger2018set}.

The `quality' of the intersection of two sets at a point $\bar{x}$ can also be quantified:
\begin{equation*}
    \cos(A,B,\bar{x}) = \max \set[\big]{\dotp{u}{v}}{u \in -\NormalLim{A}{\bar{x}} \cap \overline{\ball},~v \in \NormalLim{B}{\bar{x}} \cap \overline{\ball}}.
\end{equation*}
This value lies in the interval $[0,1]$ and generalises the cosine of the angle between two linear subspaces. The following lemmas reveal the relationship between $\cos(A,B,\bar{x})$ and transversality; find their proofs in \cite{lewis2009local, phan2016linear}.

\begin{lemma}
Let non-empty $A, B \subseteq H$ intersect at $\bar{x} \in A \cap B$, where both sets are locally closed. They intersect transversally at $\bar{x}$ if and only if $\cos(A,B,\bar{x}) < 1$.
\end{lemma}

\begin{lemma}
\label{lemma:locally_transversal}
Let non-empty $A, B \subseteq H$ intersect at $\bar{x} \in A \cap B$, where both sets are locally closed. Assume that $\cos(A,B,\bar{x}) < c$ for some $c \in (0, 1]$. Then there is a neighbourhood $\mathcal{V} \subseteq H$ of $\bar{x}$ such that for every $a \in A \cap \mathcal{V}$ and $b \in B \cap \mathcal{V}$ it holds that
\begin{equation*}
    \dotp{u}{v} < c \quad \text{for all} \quad u \in -\NormalLim{A}{a} \cap \overline{\ball} \quad \text{and} \quad v \in \NormalLim{B}{b} \cap \overline{\ball}.
\end{equation*}
\end{lemma}
\section{Contributions: Detailed version}
\label{sec:contributions}

Recall the definition \eqref{eq:qmp} of the $\sigma$-quasioptimal metric projection:
\begin{equation*}
    \Proj{A}{x; \sigma} = \set{a \in A}{\norm{x-a} \leq \sigma \cdot \dist{x}{A}}.
\end{equation*}
It serves to describe a wide range of inexact projections that appear in applications. From the theoretical perspective, it is important to understand when the quasioptimal metric projections are non-empty. Clearly, $\Proj{A}{x; \sigma}$ contain the optimal metric projection $\Proj{A}{x}$, so we can refer to Lemma~\ref{lemma:locally_closed_local_proj} to establish the local existence of quasioptimal metric projections. In Lemma~\ref{lemma:qoptimal_projection_exists}, we prove a stronger \emph{global} result: $\Proj{A}{x; \sigma}$ are non-empty almost everywhere in $H$ for $\sigma > 1$.

Our convergence theory for the QAP \eqref{eq:qap} is inspired by \cite{lewis2009local}. Namely, we prove that the iterations converge locally and linearly to some $\hat{x} \in A \cap B$ when the quasioptimality constants are not too large and there is a common point $\bar{x} \in A \cap B$ such that (i)~at least one of the sets is super-regular at $\bar{x}$ and (ii)~the sets intersect transversally at $\bar{x}$.

When one of the projections is optimal, $\sigma_B = 1$, Theorem~\ref{theorem:1qoptimal_convergence} and Corollary~\ref{corollary:1qoptimal_convergence_2sr} provide an upper bound for the values of $\sigma_A$ that guarantee convergence for every instance of the sequence of AP. This upper bound is determined by $\cos(A, B, \bar{x})^{-1}$, hence the projection onto $A$ can be \emph{very} inexact if the sets are almost orthogonal. This is promising for applications to low-rank approximation of order-$d$ tensors as the corresponding quasioptimality~constants~$\sigma_A$~scale~like~$\sqrt{d}$.

If $\sigma_A > 1$ and $\sigma_B > 1$, we prove convergence in Theorem~\ref{theorem:2qoptimal_convergence} provided that both $A$ and $B$ are super-regular, $\sigma_A$ and $\sigma_B$ are individually bounded by $\cos(A, B, \bar{x})^{-1}$ and
\begin{equation*}
    \Big( c + \sqrt{1 - c^2} \sqrt{\sigma_A^2 - 1} \Big) \cdot \Big( c + \sqrt{1 - c^2} \sqrt{\sigma_B^2 - 1} \Big) < 1, \quad c = \cos(A, B, \bar{x}).
\end{equation*}

Our proof of Theorem~\ref{theorem:2qoptimal_convergence} relies on what appears to be a hitherto unknown property of super-regular sets; we call it the \emph{Pythagorean property}. The Pythagorean theorem together with the trigonometric identities completely describes right-angled triangles. If $x$ is a vertex opposite the right angle, its adjacent leg is the unique element of $x - \Proj{A}{x}$ and the hypotenuse is an element of $x - \Proj{A}{x; \sigma}$ for some $\sigma > 1$, where $A$ is the line through the second leg. In Section~\ref{subsec:basp} and Appendix~\ref{appendix:basp}, we show that the following principles of plane geometry generalise to super-regular sets as inequalities:
\begin{itemize}
    \item the cosine of the angle between a `leg' and the `hypotenuse' is determined by the ratio of their lengths (Lemma~\ref{lemma:basp_superregular} and Corollary~\ref{corollary:basp_superregular_qomp}),
    \item the squared length of a `leg' is determined by the difference of the squared lengths of the `hypotenuse'  and the second `leg' (Corollary~\ref{corollary:qmp_close_to_omp}).
\end{itemize}

With $\sigma_A = 1$ in Theorem~\ref{theorem:1qoptimal_convergence} and Corollary~\ref{corollary:1qoptimal_convergence_2sr}, we return to the realm of the exact AP; they refine \cite[Thm.~5.2]{lewis2009local}, where the limit $\hat{x} \in A \cap B$ is shown to satisfy
\begin{equation*}
    \norm{\hat{x} - a_0} \leq \frac{1 + \cos(A, B, \bar{x})}{1 - \cos(A, B, \bar{x})} \cdot \norm{a_0 - \bar{x}}.
\end{equation*}
We tighten this estimate by replacing $\norm{a_0 - \overline{x}}$ with $\norm{a_0 - b_1}$. As a consequence, 
\begin{equation*}
    \norm{\hat{x} - a_0} \leq \frac{1 + \cos(A, B, \bar{x})}{1 - \cos(A, B, \bar{x})} \cdot \dist{a_0}{B} \leq \frac{1 + \cos(A, B, \bar{x})}{1 - \cos(A, B, \bar{x})} \cdot \dist{a_0}{A \cap B},
\end{equation*}
so the limit of the AP started with $a_0 \in A$ lies in the quasioptimal metric projections of $a_0$ onto $B$ and $A \cap B$. This observation holds for the QAP as well and allows us to estimate the distance from $\hat{x}$ to the optimal metric projections of $a_0$ onto $B$ and $A \cap B$ in Theorem~\ref{theorem:ap_as_qmp} and Corollary~\ref{corollary:ap_as_qmp}.

As we have made clear, our interest in the theory behind the QAP stems from their successful applications to constrained low-rank approximation of matrices and tensors. Naturally, we showcase some of them and focus on two problems in particular.\footnote{We leave their theoretical analysis within the framework of the AP for the future.}

The first problem is low-rank approximation of nonnegative matrices and tensors. It was approached numerically with the QAP in \cite{matveev2023sketching, sultonov2023low}, showing linear convergence. In Subsection~\ref{subsec:nonnegative}, we develop fast AP algorithms based on low-rank \emph{cross approximation} (see Subsection~\ref{subsec:qoptimal_lowrank} and the supplementary material) and acceleration techniques:
\begin{itemize}
    \item with cross approximation, the metric projection onto the nonnegative orthant can be computed without multiplying the low-rank factors;
    \item reflections and shifts can improve the rate of convergence for low-rank nonnegative approximation at the expense of a mild increase of the approximation error;
    \item a single step of the cross-approximation-based AP leads to noticeable regularisation in low-rank recovery problems such as matrix completion without changing the asymptotic computational complexity.
\end{itemize}

The second problem is low-rank approximation in the \emph{maximum} norm 
\begin{equation}
\label{eq:max_norm}
    \Norm{\Tensor{X}}{\max} = \max_{i_1, \ldots, i_d} |\Tensor{X}(i_1, \ldots, i_d)|.
\end{equation}
The truncated SVD yields the best low-rank approximation to a matrix in any unitarily invariant norm (e.g., the spectral norm and the Frobenius norm), and its error depends on the decay of the singular values. The maximum norm is not unitarily invariant, hence neither the smallest approximation error nor the low-rank matrix that achieves it can be deduced from the SVD.

One of the results of \cite{srebro2005rank} is a low-rank approximation guarantee in the maximum norm that does not depend on the singular values' decay: any $n \times n$ matrix $\Matrix{X}$ can be approximated as $\Norm{\Matrix{X - Y}}{max} 
\leq \varepsilon \Norm{\Matrix{X}}{2}$ with $\rank{\Matrix{Y}} \leq \lceil 9 \log(3n^2) / \varepsilon^2\rceil$. In Subsection~\ref{subsec:maximum}, we propose to alternate the projections between the set of fixed-rank matrices and the maximum-norm ball to compute such approximations; we use this method to (i)~approximate identity matrices and random orthogonal matrices, and (ii)~reconstruct low-rank matrices and tensors with quantised entries.
\section{Related work}
\label{sec:related}

Inexact AP have been studied before. In \cite[Thm.~6.1]{lewis2009local}, the metric projection $b_{n+1} \in \Proj{B}{a_n}$ was replaced with a point $b_{n+1} \in B$ that satisfies (i)~a monotonicity condition $\norm{b_{n+1} - a_n} \leq \norm{a_n - b_{n-1}}$, and (ii)~a \emph{normal-cone condition}, i.e., the step $a_n - b_{n+1}$ is almost a limiting normal to $B$ at~$b_{n + 1}$:
\begin{equation}
\label{eq:inexact_angle}
    \dist[\Bigg]{\frac{a_n - b_{n+1}}{\norm{a_n - b_{n+1}}}}{\NormalLim{B}{b_{n+1}}} \leq \gamma, \quad \gamma \in  [0, 1).
\end{equation}
Note that there can be points $b_{n+1}$ with $a_n - b_{n+1} \in \NormalLim{B}{b_{n+1}}$ but $b_{n+1} \not\in \Proj{B}{a_n}$. Consider a closed annulus $B$ in the plane. Every $x \not\in B$ has at least two points $b \in B$ that satisfy $x - b \in \NormalLim{B}{b}$: on the inner and outer circles, respectively. This is when the monotonicity condition becomes restrictive and narrows down the candidates. Both metric projections were made inexact in this sense in \cite[Thm.~34]{kruger2015regularity}.

When $B$ is not super-regular, as is the case in \cite[Thm.~6.1]{lewis2009local}, this particular form of inexactness makes it possible to overcome the irregular behaviour of the limiting normal cone to $B$ in the proof of convergence. We, on the contrary, work only with quasioptimal metric projections onto super-regular sets in Theorem~\ref{theorem:1qoptimal_convergence} and Theorem~\ref{theorem:2qoptimal_convergence}. This allows us to invoke the Pythagorean property (Corollary~\ref{corollary:basp_superregular_qomp}): an optimal metric projection $\pi \in \Proj{B}{a_n}$ and a quasioptimal metric projection $\pi_{\sigma_B} \in \Proj{B}{a_n; \sigma_B}$ satisfy
\begin{equation*}
    \dotp[\bigg]{\frac{a_n - \pi}{\norm{a_n - \pi}}}{\frac{a_n - \pi_{\sigma_{B}}}{\norm{a_n - \pi_{\sigma_B}}}} \gtrsim \frac{1}{\sigma_B}.
\end{equation*}
Thus, instead of assuming that $a_n - \pi_{\sigma_B}$ is close to $\NormalLim{B}{\pi_{\sigma_B}}$, we base our reasoning on the observation that $a_n - \pi_{\sigma_B}$ is well-aligned with $\NormalProx{B}{\pi}$. 

The algorithms that compute inexact projections as in \cite[Thm.~6.1]{lewis2009local} and \cite[Thm.~34]{kruger2015regularity}, or at least the analysis of their performance, should be aware of the first-order information about the set. The use of quasioptimal metric projections allows us to avoid involving such information in the statements of our theorems: an algorithm that computes quasioptimal metric projections only needs to meet a zeroth-order condition as long as the corresponding set is super-regular. In addition, an explicit monotonicity-like condition is not required: it is satisfied automatically when the two quasioptimality constants are jointly not too large.

We suspect, but cannot prove now, that the normal-cone condition \eqref{eq:inexact_angle} holds for quasioptimal metric projections onto sufficiently regular sets. For example, $\sigma$-quasioptimal metric projections onto a line satisfy it with $\gamma = \sqrt{1 - 1 / \sigma^2}$. Similar behaviour can be expected from smooth manifolds, but we do not know whether the same can be said about general super-regular sets.

A combination of quasioptimality and the normal-cone condition \eqref{eq:inexact_angle} was considered in \cite[Thm.~31]{kruger2015regularity}, where neither of the two sets is required to be super-regular and a weaker assumption of intrinsic transversality \cite{drusvyatskiy2014alternating} is used. As in \cite[Thm.~6.1]{lewis2009local}, the absence of super-regularity necessitates the normal-cone condition for the proof.

Another point of view is to treat an inexact projection as a sum of a metric projection and an error. Two examples were studied in \cite[Thms.~1, 3]{drusvyatskiy2019local}. In the first one, the metric projection onto $A$ was approximated with a map $f: H \to H$ such that
\begin{equation*}
    \lim_{x \to A} \frac{\dist{f(x)}{\Proj{A}{x}}}{\dist{x}{A}} = 0.
\end{equation*}
Interestingly, $f(x)$ need not lie in $A$, but if it does then $f(x)$ is a quasioptimal metric projection for every $x$ that is sufficiently close to $A$. The inverse is false: in general, quasioptimal metric projections do not belong to this class of inexact projections. In the second example, $f$ is allowed to depend on the previous iteration and satisfies
\begin{equation*}
    \dist[\Big]{f(b_{n+1}, a_{n})}{\Proj{A}{b_{n+1}}} \leq \varepsilon \cdot \norm{b_{n+1} - a_n}, \quad b_{n+1} \in \Proj{B}{a_n}.
\end{equation*}
A necessary condition for convergence is $\varepsilon < 1$. At the same time, Corollary~\ref{corollary:qmp_close_to_omp} says that the following holds for the quasioptimal metric projection $\pi_{\sigma_A} \in \Proj{A}{b_{n+1}; \sigma_A}$:
\begin{equation*}
    \dist{\pi_{\sigma_A}}{\Proj{A}{b_{n+1}}} \lesssim \sqrt{\sigma_A^2 - 1} \cdot \norm{b_{n+1} - a_n}.
\end{equation*}
So the result of \cite{drusvyatskiy2019local} is only applicable for $\sigma_A \lesssim \sqrt{2}$, while Theorem~\ref{theorem:1qoptimal_convergence} can guarantee convergence with arbitrarily large $\sigma_A$ as long as $A$ and $B$ are `sufficiently orthogonal'.
\section{Quasioptimal metric projections}
\label{sec:qoptimal}

Several simple, yet useful, properties follow directly from the definition \eqref{eq:qmp}. The quasioptimal metric projections are nested,
\begin{equation*}
    \Proj{A}{x; \sigma} \subseteq \Proj{A}{x; \sigma'} \quad \text{when} \quad \sigma \leq \sigma',
\end{equation*}
cannot be too spread out when $x$ is close to $A$,
\begin{equation*}
    \norm{\pi - \pi'} \leq (\sigma + \sigma') \cdot \dist{x}{A} \quad \text{for every} \quad \pi \in \Proj{A}{x; \sigma} \quad \text{and} \quad \pi' \in \Proj{A}{x; \sigma'},
\end{equation*}
and remain near points of reference,
\begin{equation*}
    \norm{\pi - \hat{a}} \leq (\sigma + 1) \cdot \norm{x - \hat{a}} \quad \text{for every} \quad \pi \in \Proj{A}{x; \sigma} \quad \text{and} \quad \hat{a} \in A.
\end{equation*}

\subsection{Existence}
When is $\Proj{A}{x; \sigma}$ non-empty? Lemma~\ref{lemma:locally_closed_local_proj} states that every $x \in H$ has a non-empty optimal metric projection $\Proj{A}{x}$ if and only if $A$ is closed. When $A$ is open, $\Proj{A}{x}$ is non-empty only for $x \in A$, in which case the projection is $\Proj{A}{x} = \{ x \}$. In the more general case where $A$ contains only a part of its boundary, the set of points without any metric projections onto $A$ has positive measure (e.g., an open interval in Fig.~\ref{fig:empty_optimal_proj}). The $\sigma$-quasioptimal metric projections are \textit{almost} immune to this problem.

\begin{figure}[!h]
\centering
    \includegraphics[width=0.9\linewidth]{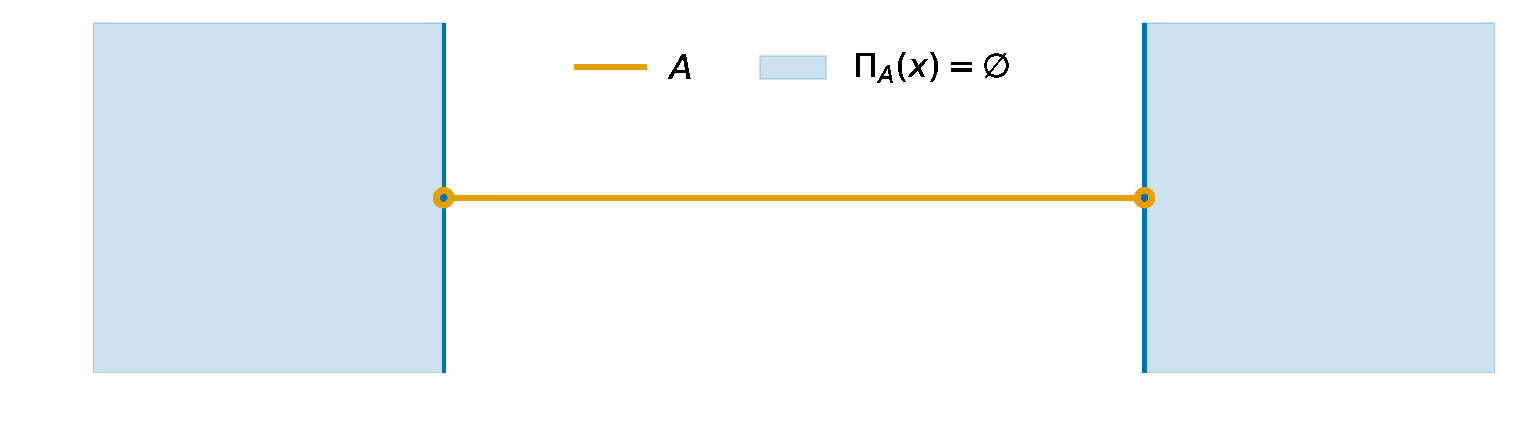}
\caption{When $A$ is not closed, the set $\set{x \in H}{\Proj{A}{x} = \emptyset}$ has positive measure. For $\sigma$-quasioptimal metric projections with $\sigma > 1$, this is a subset of $\boundary A$.}
\label{fig:empty_optimal_proj}
\end{figure}

\begin{lemma}
\label{lemma:qoptimal_projection_exists}
Let $A \subseteq H$ be non-empty and $\sigma > 1$. Then for every $x \in H$ the $\sigma$-quasioptimal metric projection $\Proj{A}{x; \sigma}$ satisfies
\begin{enumerate}
    \item $\Proj{A}{x; \sigma} = \{ x \}$ when $x \in A$;
    \item $\Proj{A}{x; \sigma} \neq \emptyset$ when $x \in A \cup (H \setminus \overline{A})$;
    \item $\Proj{A}{x; \sigma} = \emptyset$ when $x \in \overline{A} \setminus A$.
\end{enumerate}
\end{lemma}
\begin{proof}
$\bm{1.}$ For $x \in A$, we have $\dist{x}{A} = 0$, and $\norm{x - a} = 0$ leads to $a = x$.\\
$\bm{2.}$ When $x \in H \setminus \overline{A}$, $\dist{x}{A} > 0$ and there exists a minimising sequence $(a_n) \subset A$ such that $\norm{x - a_n} \to \dist{x}{A}$ from above. That is, for every $\varepsilon > 0$ we can pick $N$ large enough so that $0 \leq \norm{x - a_n} - \dist{x}{A} < \varepsilon$ for all $n \geq N$. If we choose $\varepsilon = (\sigma - 1) \cdot \dist{x}{A} > 0$ then the tail of the sequence $(a_n)$ belongs to $\Proj{A}{x; \sigma}$.\\
$\bm{3.}$ Lastly, $\dist{x}{A} = 0$ for $x \in \overline{A} \setminus A$. For every $a \in \Proj{A}{x; \sigma}$ this implies that $\norm{x - a} = 0$ and $x = a$, which contradicts the assumption that $x \not\in A$.
\end{proof}

\begin{corollary}
\label{corollary:qoptimal_projection_locally_closed}
Let $A \subseteq H$ be non-empty and $\sigma > 1$. 
\begin{enumerate}
    \item Let $\mathcal{V} \subseteq H$ be a neighbourhood of $\hat{a} \in A$. The $\sigma$-quasioptimal metric projection $\Proj{A}{x; \sigma}$ is non-empty for each $x \in \mathcal{V}$ if and only if $A \cap \mathcal{V}$ is closed in $\mathcal{V}$.
    \item The $\sigma$-quasioptimal metric projection $\Proj{A}{x; \sigma}$ is non-empty for every $x \in H$ if and only if $A$ is closed.
\end{enumerate}
\end{corollary}
\begin{proof}
Let $x \in \mathcal{V}$ be a limit point of $A \cap \mathcal{V}$, so that $\dist{x}{A} = 0$. If $\Proj{A}{x; \sigma} \neq \emptyset$, there exists $a \in A$ such that $\norm{x - a} \leq \sigma \cdot \dist{x}{A} = 0$ and thus $x = a \in A \cap \mathcal{V}$. In the other direction, if $A \cap \mathcal{V}$ is closed in $\mathcal{V}$ then every limit point of $A \cap \mathcal{V}$ in $\mathcal{V}$ belongs to $A$. Therefore, each $x \in \mathcal{V}$ either lies in $A$ or $\dist{x}{A} > 0$, in which case we can repeat the argument from the proof of Lemma~\ref{lemma:qoptimal_projection_exists}. For the second part, choose $\mathcal{V} = H$.
\end{proof}

\begin{remark}
\begin{enumerate}
    \item Lemma~\ref{lemma:qoptimal_projection_exists} and Corollary~\ref{corollary:qoptimal_projection_locally_closed} hold in general Hilbert spaces.
    \item When $\sigma = 1$, the second part of Corollary~\ref{corollary:qoptimal_projection_locally_closed} holds in Euclidean spaces, but not in infinite-dimensional Hilbert spaces: there exist closed sets in Hilbert spaces such that certain points have empty metric projections onto them \cite[Example~3.13]{bauschke2017convex}.
    \item If $A \cap \mathcal{V}$ is closed in $\mathcal{V}$, the first part of Corollary~\ref{corollary:qoptimal_projection_locally_closed} states that the $\sigma$-quasioptimal metric projection exists for every $x \in \mathcal{V}$ when $\sigma > 1$. For optimal metric projections, this can be guaranteed only in a smaller neighbourhood (see Lemma~\ref{lemma:locally_closed_local_proj}).
\end{enumerate}
\end{remark}

\subsection{Pythagorean property of super-regular sets}
\label{subsec:basp}
Until now, our discussion of quasioptimal metric projections was based only on the metric properties of the ambient space. The Euclidean structure of $H$ makes it possible to compare the \emph{directions} of optimal and quasioptimal metric projections.

Consider a hyperplance $A$, a point $x \not\in A$, its orthogonal projection $\pi \in A$, and a quasioptimal metric projection $\pi_\sigma \in \Proj{A}{x; \sigma}$. A basic trigonometric argument gives that $\pi - x$ and $\pi_\sigma - x$ form an acute angle that is upper bounded by $\arccos(\frac{1}{\sigma})$. Thus, the quasioptimal metric projections onto a hyperplance enjoy a geometric property that we call \textit{Pythagorean}: they are aligned with the direction of the optimal metric projection. See Fig.~\ref{fig:angular_spread_line} for a two-dimensional example.

\begin{figure}[!ht]
\centering
    \includegraphics[width=0.9\linewidth]{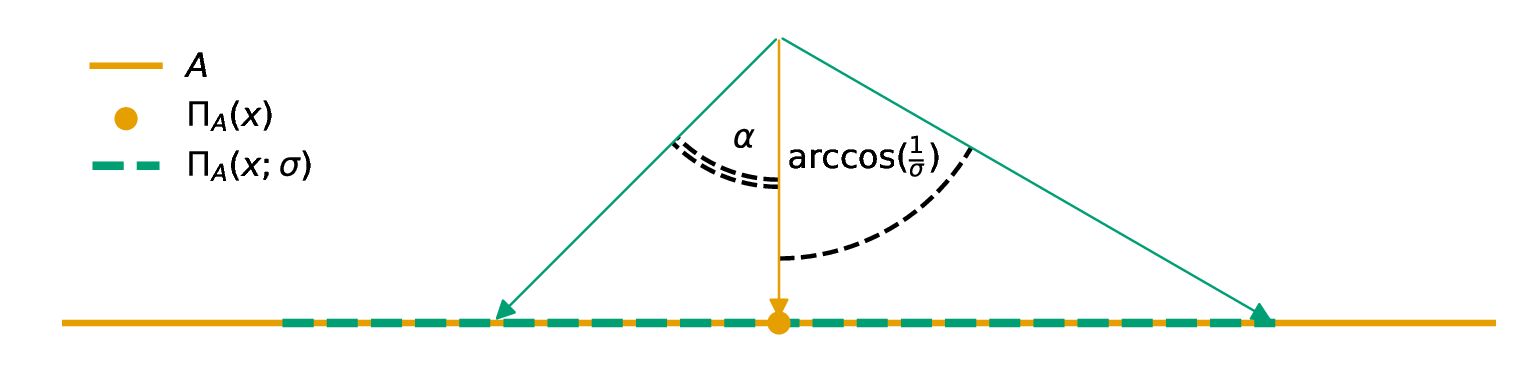}
\caption{Every $\sigma$-quasioptimal projection onto a hyperplane satisfies the Pythagorean property: the angle $\alpha$ it forms with the orthogonal projection is bounded by $\arccos(\frac{1}{\sigma})$.}
\label{fig:angular_spread_line}
\end{figure}

In Lemma~\ref{lemma:basp_superregular} and Corollary~\ref{corollary:basp_superregular_qomp}, we show that the Pythagorean property persists as we substitute a super-regular set for the hyperplane. This intuitive result is crucial for the proof of the convergence Theorem~\ref{theorem:2qoptimal_convergence}. In Appendix~\ref{appendix:basp}, we derive analogues of Lemma~\ref{lemma:basp_superregular} for prox-regular and convex sets.

\begin{lemma}
\label{lemma:basp_superregular}
Let $A \subseteq H$ be non-empty and super-regular at $\hat{a} \in A$. Then for every $\varepsilon > 0$ there exists $\delta > 0$ such that for every $x \in \ballrx{\delta/2}{\hat{a}} \setminus A$ it holds that
\begin{equation*}
        \dotp[\bigg]{\frac{x - \pi}{\norm{x - \pi}}}{\frac{x - a}{\norm{x - a}}} \geq \frac{\norm{x - \pi}}{\norm{x - a}} - \varepsilon (1 + \varepsilon) 
\end{equation*}
for each $\pi \in \Proj{A}{x}$ and $a \in A \cap \ballrx{\delta}{\hat{a}}$.
\end{lemma}
\begin{proof}
Fix $\varepsilon > 0$. Since $A$ is super-regular at $\hat{a}$, we can choose $\delta > 0$ such that $A \cap \ballrx{\delta}{\hat{a}}$ is closed in $\ballrx{\delta}{\hat{a}}$ and property \eqref{eq:sr_normal_angle} holds:
\begin{equation*}
    \dotp{v}{a' - a} \leq \varepsilon \norm{v} \cdot \norm{a' - a} \quad  \text{for all} \quad a,a' \in A \cap \ballrx{\delta}{\hat{a}} \quad \text{and} \quad v \in \NormalLim{A}{a}.
\end{equation*}
For each $x \in \ballrx{\delta/2}{\hat{a}}$, Lemma~\ref{lemma:locally_closed_local_proj} states that the optimal metric projection $\Proj{A}{x}$ is non-empty and every $\pi \in \Proj{A}{x}$ lies in $A \cap \ballrx{\delta}{\hat{a}}$. 

Pick $\pi \in \Proj{A}{x}$, $a \in A \cap \ballrx{\delta}{\hat{a}}$, and recall that $x - \pi \in \NormalLim{A}{\pi}$. By property \eqref{eq:sr_normal_angle},
\begin{equation*}
     \varepsilon \norm{x - \pi} \cdot \norm{a - \pi} \geq \dotp{x - \pi}{a - \pi} = \norm{x - \pi}^2 - \dotp{x - \pi}{x - a}.
\end{equation*}
Lemma~\ref{lemma:superregular_proj_nonexpansove} gives $\norm{\pi - a} \leq (1 + \varepsilon) \norm{x - a}$, and we conclude that 
\begin{equation*}
     \dotp{x - \pi}{x - a} \geq \norm{x - \pi} \big(\norm{x - \pi} - \varepsilon \norm{a - \pi} \big) \geq \norm{x - \pi} \big(\norm{x - \pi} - \varepsilon(1+\varepsilon) \norm{x - a} \big). \qedhere
\end{equation*}
\end{proof}

\begin{corollary}
\label{corollary:basp_superregular_qomp}
Let $A \subseteq H$ be non-empty and super-regular at $\hat{a} \in A$. Then for every $\varepsilon > 0$ there exists $\delta > 0$ such that for every $x \in \ballrx{\delta/2}{\hat{a}} \setminus A$ it holds that
\begin{equation*}
        \dotp[\bigg]{\frac{x - \pi}{\norm{x - \pi}}}{\frac{x - \pi_\sigma}{\norm{x - \pi_\sigma}}} \geq \frac{1}{\sigma} - \varepsilon (1 + \varepsilon) 
\end{equation*}
for each $\pi \in \Proj{A}{x}$ and $\pi_\sigma \in \Proj{A}{x; \sigma} \cap \ballrx{\delta}{\hat{a}}$.
\end{corollary}

\begin{remark}
The Pythagorean property can be violated when the set is not super-regular (see Fig.~\ref{fig:bad_angular_spread_spiral}). Consider a closed piecewise linear curve $A \subset \Real^2$ with vertices $a_n$ defined for every $n \in \Z$ as $a_0 = (0,0)$ and 
\begin{equation*}
    a_{n+1} = a_n + 2^{-n} \cdot
    \begin{cases}
        (0,-1), & n \pmod 4 \equiv 0, \\
        (1,0),  & n \pmod 4 \equiv 1, \\
        (0,1),  & n \pmod 4 \equiv 2, \\
        (-1,0), & n \pmod 4 \equiv 3.
    \end{cases}
\end{equation*}
Simple series summation gives $a_n \to \hat{a} = \frac{2}{5}(1, -2)$ as $n \to \infty$. For every $k \in \Z$,
\begin{equation*}
    a_{4k} = \frac{2}{5}\left[1 - 2^{-4k}\right] (1, -2), \quad a_{4k - 1} = a_{4k} + 2^{-4k} (2, 0).
\end{equation*}
Therefore, every linear segment $[a_{4k}, a_{4k-1}]$ is horizontal, and the vertical line through $\hat{a}$ intersects each of them. Consider a sequence of points $(x_k)$ defined by
\begin{equation*}
    x_k = \frac{2}{5} \left( 1, -2\left[ 1 - \frac{17}{32} \cdot 2^{-4k} \right] \right), \quad k \in \Z.
\end{equation*}
Each $x_k$ has the same first coordinate as $\hat{a}$ and is equidistant from the horizontal segments $[a_{4k}, a_{4k-1}]$ and $[a_{4k+4}, a_{4k+3}]$ with $\dist{x_k}{A} = \frac{3}{8} \cdot 2^{-4k}$. The optimal metric projection of $x_k$ onto $A$ consists of two points $\Proj{A}{x_k} = \{ \pi_{k,1}, \pi_{k,2} \}$, where
\begin{equation*}
    \pi_{k,1} = \frac{2}{5} \left(1, -2\left[ 1 - 2^{-4k} \right] \right), \quad \pi_{k,2} = \frac{2}{5} \left(1, -2\left[ 1 - \frac{1}{16} \cdot 2^{-4k} \right] \right),
\end{equation*}
which are oriented in opposite directions
\begin{equation*}
    \dotp[\bigg]{\frac{x_k - \pi_{k,1}}{\norm{x_k - \pi_{k,1}}}}{\frac{x_k - \pi_{k,2}}{\norm{x_k - \pi_{k,2}}}} = -1.
\end{equation*}
Thus, since $x_k \to \hat{a}$ as $k \to \infty$, every neighbourhood of $\hat{a}$ contains a point whose optimal metric projections onto $A$ violate the Pythagorean property.
\end{remark}

\begin{figure}[!t]
\centering
    \includegraphics[width=0.8\linewidth]{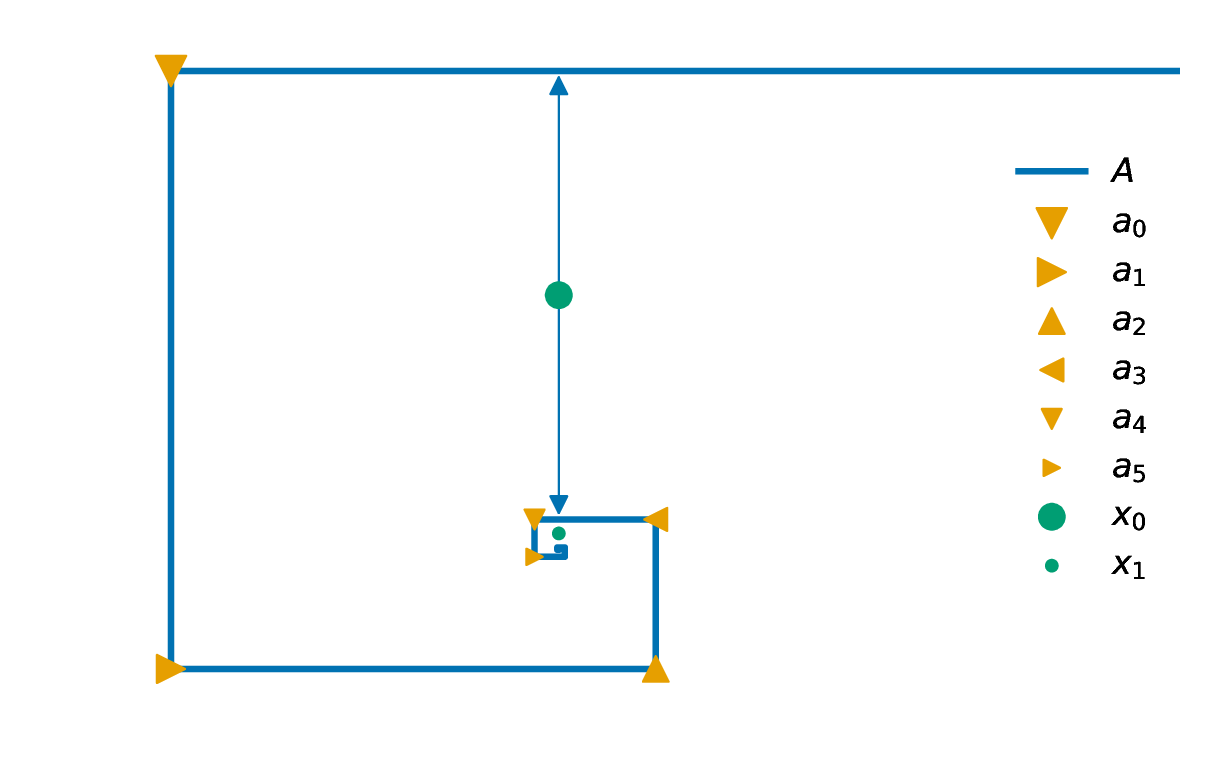}
\caption{The Pythagorean property can be violated when $A$ is not super-regular.}
\label{fig:bad_angular_spread_spiral}
\end{figure}

The Pythagorean property also gives us better control over the location of $\Proj{A}{x; \sigma}$ relative to $\Proj{A}{x}$. Recall a simple bound $\norm{\pi_\sigma - \pi} \leq (\sigma + 1) \cdot \dist{x}{A}$ that holds for every $\pi_\sigma \in \Proj{A}{x; \sigma}$ and $\pi \in \Proj{A}{x}$. This estimate turns out to be rather pessimistic for super-regular sets, especially when $\sigma$ is close to 1.

\begin{corollary}
\label{corollary:qmp_close_to_omp}
Let $A \subseteq H$ be non-empty and super-regular at $\hat{a} \in A$. Then for every $\varepsilon > 0$ there exists $\delta > 0$ such that for every $x \in \ballrx{\delta/2}{\hat{a}}$ and $\sigma \geq 1$ it holds that
\begin{equation*}
    \norm{\pi_\sigma - \pi} \leq \left( \varepsilon + \sqrt{\sigma^2 - 1 + \varepsilon^2}\right) \cdot \dist{x}{A}
\end{equation*}
for each $\pi \in \Proj{A}{x}$ and $\pi_\sigma \in \Proj{A}{x; \sigma} \cap \ballrx{\delta}{\hat{a}}$.
\end{corollary}
\begin{proof}
From the proof of Lemma~\ref{lemma:basp_superregular}, we get 
\begin{equation*}
    \dotp{x - \pi}{x - \pi_\sigma} \geq \norm{x - \pi} \big(\norm{x - \pi} - \varepsilon \norm{\pi_\sigma - \pi} \big).
\end{equation*}
We put this estimate into the following expression of the squared norm
\begin{align*}
    \norm{\pi_\sigma - \pi}^2 &= \norm{\pi - x}^2 + \norm{\pi_\sigma - x}^2 - 2 \dotp{x - \pi}{x - \pi_\sigma} \\
    &\leq (\sigma^2 - 1) \norm{\pi - x}^2 + 2\varepsilon \norm{\pi - x} \norm{\pi_\sigma - \pi}.
\end{align*}
It remains to solve the inequality for $\norm{\pi_\sigma - \pi}$ and note that $\norm{\pi - x} = \dist{x}{A}$.
\end{proof}
\section{Convergence analysis}
\label{sec:convergence}

\subsection{Technical convergence lemma}
The techniques we use to study the convergence of the QAP \eqref{eq:qap} are inspired by \cite{lewis2009local}. Broadly speaking, we will approach the proof by showing that the steps taken by the method, $|b_{n+1} - a_{n}|$ and $|a_{n+1} - b_{n+1}|$, shrink with iterations. 

We will consider three separate scenarios, which differ in the values taken by the quasioptimality constants $\sigma_A$ and $\sigma_B$ and the properties that we require from the sets $A$ and $B$. These nuances lead to different estimates on the step sizes; the proofs of convergence, however, follow the same pattern, which we summarise in Lemma~\ref{lemma:convergence_under_the_hood}.

\begin{lemma}
\label{lemma:convergence_under_the_hood}
Let $A, B \subseteq H$ be non-empty and assume that there is an open set $\mathcal{V} \subseteq H$ such that $A \cap \mathcal{V}$ and $B \cap \mathcal{V}$ are closed in $\mathcal{V}$. Let $\mathcal{U} \subset \mathcal{V}$ be closed and $\kappa_A, \kappa_B > 0$. Suppose that, for $n \in \N$, sequences $(a_n) \subset A \cap \mathcal{U}$ and $(b_n) \subset B \cap \mathcal{U}$ satisfy
\begin{equation*}
    \norm{a_n - b_n} \leq \kappa_A \norm{b_n - a_{n-1}}, \quad \norm{b_{n+1} - a_n} \leq \kappa_B \norm{a_n - b_n}.
\end{equation*}
If $\kappa_A \cdot \kappa_B < 1$ then both $(a_n)$ and $(b_n)$ converge to $\hat{x} \in A \cap B \cap \mathcal{U}$ and, for $n \in \N_0$,
\begin{equation*}
    \norm{\hat{x} - a_n} \leq \tfrac{1 + \kappa_{A}}{1 - \kappa_{A} \kappa_{B}} (\kappa_{A} \kappa_{B})^n \cdot \norm{b_1 - a_0}, \quad \norm{\hat{x} - b_{n+1}} \leq \tfrac{(1 + \kappa_{B}) \kappa_{A}}{1 - \kappa_{A} \kappa_{B}} (\kappa_{A} \kappa_{B})^n \cdot \norm{b_1 - a_0}. 
\end{equation*}
\end{lemma}
\begin{proof}
By induction, we get that for all $n \in \N_0$ the sequences $(a_n)$ and $(b_n)$ satisfy
\begin{equation*}
    \norm{b_{n+1} - a_n} \leq (\kappa_A \kappa_B)^n \cdot \norm{b_1 - a_0}, \quad \norm{a_{n+1} - b_{n+1}} \leq \kappa_A (\kappa_A \kappa_B)^n \cdot \norm{b_1 - a_0}.
\end{equation*}
Next, we show that the sequence $( a_0, b_1, a_1, b_2, a_2, \ldots )$ is Cauchy. To this end, fix an arbitrary $n \in \N$; then for every $k > n$ we have
\begin{align*}
    \max\{ \norm{a_k - a_n}, \norm{b_k - a_n} \} &\leq \norm{b_{n+1} - a_{n}} + \norm{a_{n+1} - b_{n+1}} + \norm{b_{n+2} - a_{n+1}} + \ldots \\
    &\leq \norm{b_1 - a_0} \cdot \left[ (\kappa_A \kappa_B)^{n} + \kappa_A(\kappa_A \kappa_B)^{n} + (\kappa_A \kappa_B)^{n + 1} + \ldots \right] \\
    &= \tfrac{1 + \kappa_A}{1 - \kappa_A \kappa_B} (\kappa_A \kappa_B)^n \cdot \norm{b_1 - a_0},
\end{align*}
and for every $k \geq n$ it holds that 
\begin{align*}
    \max\{ \norm{a_k - b_n}, \norm{b_k - b_n} \} &\leq \norm{a_{n} - b_{n}} + \norm{b_{n+1} - a_{n}} + \norm{a_{n+1} - b_{n+1}} + \ldots \\
    &\leq \norm{b_1 - a_0} \cdot \left[ \kappa_A (\kappa_A \kappa_B)^{n-1} + (\kappa_A \kappa_B)^{n} + \kappa_A(\kappa_A \kappa_B)^{n} + \ldots \right] \\
    &= \tfrac{(1 + \kappa_{B}) \kappa_{A}}{1 - \kappa_{A} \kappa_{B}} (\kappa_{A} \kappa_{B})^{n-1} \cdot \norm{b_1 - a_0}.
\end{align*}
Therefore, the sequence is indeed Cauchy and has a limit $\hat{x}$, which satisfies 
\begin{equation*}
    \norm{\hat{x} - a_n} \leq \tfrac{1 + \kappa_{A}}{1 - \kappa_{A} \kappa_{B}} (\kappa_{A} \kappa_{B})^n \cdot \norm{b_1 - a_0}, \quad \norm{\hat{x} - b_{n+1}} \leq \tfrac{(1 + \kappa_{B}) \kappa_{A}}{1 - \kappa_{A} \kappa_{B}} (\kappa_{A} \kappa_{B})^n \cdot \norm{b_1 - a_0}. 
\end{equation*}
This $\hat{x}$ is a limit point of $A \cap \mathcal{U}$ and $B \cap \mathcal{U}$, which are closed in $\mathcal{V}$. So $\hat{x} \in A \cap B \cap \mathcal{U}$.
\end{proof}

\subsection{One projection is quasioptimal}
In this section we assume that the projection onto $B$ is optimal. The projection onto $A$ is allowed to be quasioptimal, and we specify that $A$ is super-regular. The following lemma adapts a sub-result of \cite[Thm.~5.2]{lewis2009local} to the case of quasioptimal projections. It will allow us to estimate the step $\norm{a_{n+1} - b_{n+1}}$ in terms of its predecessor $\norm{b_{n+1} - a_n}$ in the proof of Theorem~\ref{theorem:1qoptimal_convergence}. In Corollary~\ref{corollary:1qoptimal_convergence_2sr}, we obtain a faster convergence rate when $B$ is super-regular too. We illustrate the settings of Lemma~\ref{lemma:1step_1qoptimal} in Fig~\ref{fig:1qomp}.

\begin{figure}[ht]
\centering
    \includegraphics[width=0.9\linewidth]{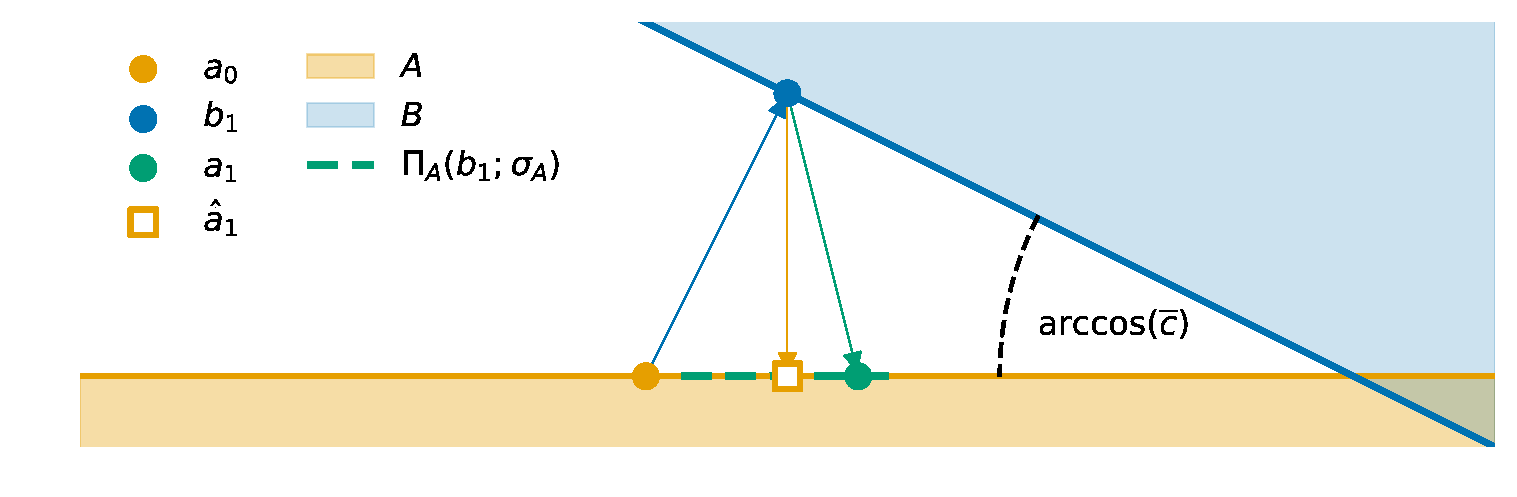}
\caption{The illustration of the settings of Lemma~\ref{lemma:1step_1qoptimal}.}
\label{fig:1qomp}
\end{figure}

\begin{lemma}
\label{lemma:1step_1qoptimal}
Let $A, B \subseteq H$ be non-empty and intersect at $\bar{x} \in A \cap B$, where
\begin{itemize}
    \item $A$ and $B$ are locally closed and intersect transversally with $\cos(A,B,\bar{x}) = \bar{c} < 1$, and
    \item $A$ is super-regular.
\end{itemize}
Pick $\varepsilon > 0$ and $c \in (\bar{c}, 1)$. There exists $\delta > 0$ such that for every $\sigma_A \geq 1$, each triplet of points $a_0 \in A \cap \ballrx{\delta}{\bar{x}}$, $b_1 \in \Proj{B}{a_0} \cap \ballrx{\delta}{\bar{x}}$ with $\Proj{A}{b_1} \cap \ballrx{\delta}{\bar{x}} \neq \emptyset$, and $a_1 \in \Proj{A}{b_1; \sigma_A}$ satisfies
\begin{equation*}
    \norm{a_1 - b_1} \leq \sigma_{A} (c + 2 \varepsilon) \cdot \norm{b_1 - a_0}.
\end{equation*}
If $a_0 \in A$ is such that $\norm{a_0 - \bar{x}} < \tfrac{\delta}{4}$ then $b_1$ and $a_1$ are well-defined, and the estimate holds for every $b_1 \in \Proj{B}{a_0}$.
\end{lemma}
\begin{proof}
Fix $\varepsilon$ and $c$. First, we set $\delta = \min\{ \delta_1(\varepsilon), \delta_2, \delta_3(c) \}$, where $\delta_1(\varepsilon)$ comes from the definition of super-regularity for $A$; $\delta_2$ ensures that $B \cap \ballrx{\delta_2}{\bar{x}}$ is closed in $\ballrx{\delta_2}{\bar{x}}$; and $\delta_3(c)$ ensures that the assertion of Lemma~\ref{lemma:locally_transversal} holds in $\ballrx{\delta_3(c)}{\bar{x}}$.

Consider $\hat{a}_1 \in \Proj{A}{b_1} \cap \ballrx{\delta}{\bar{x}}$. Then $\norm{a_1 - b_1} \leq \sigma_A \cdot \norm{\hat{a}_1 - b_1}$ and it remains to estimate $\norm{\hat{a}_1 - b_1}$ in terms of $\norm{b_1 - a_0}$. If either of them is zero, the proof is complete. Otherwise, we expand the squared norm
\begin{equation*}
    \norm{\hat{a}_1 - b_1}^2 = \dotp{b_1 - \hat{a}_1}{a_0 - \hat{a}_1} + \dotp{\hat{a}_1 - b_1}{a_0 - b_1}.
\end{equation*}
Recall that $A$ is super-regular at $\bar{x}$. Since $a_0, \hat{a}_1 \in A \cap \ballrx{\delta}{\bar{x}}$ and $b_1 - \hat{a}_1 \in \NormalLim{A}{\hat{a}_1}$,
\begin{equation*}
    \dotp{b_1 - \hat{a}_1}{a_0 - \hat{a}_1} \leq \varepsilon \norm{b_1 - \hat{a}_1} \cdot \norm{a_0 - \hat{a}_1} \leq 2\varepsilon \norm{b_1 - \hat{a}_1} \cdot \norm{a_0 - b_1}.
\end{equation*}
Since $a_0 - b_1 \in \NormalLim{B}{b_1}$ and $b_1 \in B \cap \ballrx{\delta}{\bar{x}}$, the transversality and Lemma~\ref{lemma:locally_transversal} give
\begin{equation*}
    \dotp{\hat{a}_1 - b_1}{a_0 - b_1} \leq c\norm{\hat{a}_1 - b_1} \cdot \norm{a_0 - b_1}.
\end{equation*}
Adding the two estimates, we get $\norm{\hat{a}_1 - b_1} \leq (c + 2\varepsilon) \cdot \norm{a_0 - b_1}$.

Finally, we show that the points $b_1$ and $a_1$ are well-defined when $a_0 \in A \cap \ballrx{\delta / 4}{\overline{x}}$. By Lemma~\ref{lemma:locally_closed_local_proj}, $\Proj{B}{a_0}$ is non-empty and $\norm{b_1 - \bar{x}} \leq 2 \norm{a_0 - \bar{x}} < \delta / 2$. We use Lemma~\ref{lemma:locally_closed_local_proj} again to show that $\emptyset \neq \Proj{A}{b_1} \subseteq \Proj{A}{b_1; \sigma_A}$ and $\Proj{A}{b_1} \subseteq \ballrx{\delta}{\overline{x}}$.
\end{proof}

\begin{theorem}
\label{theorem:1qoptimal_convergence}
Let $A, B \subseteq H$ be non-empty and intersect at $\bar{x} \in A \cap B$, where
\begin{itemize}
    \item $A$ and $B$ are locally closed and intersect transversally with $\cos(A,B,\bar{x}) = \bar{c} < 1$, and
    \item $A$ is super-regular.
\end{itemize}
For every $c \in (\bar{c}, 1)$ there exists $\delta > 0$ such that for each $\sigma_A \geq 1$ satisfying $\kappa = \sigma_A \cdot c < 1$ and each $a_0 \in A$ satisfying $\norm{a_0 - \bar{x}} < \frac{1 - \kappa}{8} \delta$, the QAP \eqref{eq:qap} with $\sigma_B = 1$ converge to $\hat{x} \in A \cap B$ such that $\norm{\hat{x} - \bar{x}} < \tfrac{3}{8} \delta$ and, for $n \in \N_0$,
\begin{equation*}
    \norm{\hat{x} - a_n} \leq \tfrac{1 + \kappa}{1 - \kappa} \kappa^n \cdot \dist{a_0}{B}, \quad \norm{\hat{x} - b_{n+1}} \leq \tfrac{2 \kappa}{1 - \kappa} \kappa^{n} \cdot \dist{a_0}{B}.
\end{equation*}
\end{theorem}
\begin{proof}
Fix $c$ and choose $\delta$ as follows: pick some $c' \in (\bar{c}, c)$, set $\varepsilon = \frac{1}{2} (c' - c)$ and use the value of $\delta$ from Lemma~\ref{lemma:1step_1qoptimal} that corresponds to $c'$ and $\varepsilon$. This gives $\sigma_A (c' + 2\varepsilon) = \kappa < 1$.

Let us assume for a moment that the sequences $( a_n )$ and $( b_n )$ are well-defined and $\norm{a_n - \bar{x}} < \delta / 4$ for all $n \in \N_0$. Then Lemma~\ref{lemma:1step_1qoptimal} and the definition of $\Proj{B}{a_n}$ guarantee
\begin{equation*}
    \norm{a_n - b_n} \leq \kappa \norm{b_n - a_{n-1}}, \quad \norm{b_{n+1} - a_n} \leq \norm{a_n - b_n}, \quad n \in \N.
\end{equation*}
In addition, for every $n \in \N$ we have
\begin{equation*}
    \norm{b_{n} - \bar{x}} \leq \norm{b_n - a_{n-1}} + \norm{a_{n-1} - \bar{x}} < \norm{b_1 - a_0} + \norm{a_{n-1} - \bar{x}} < \tfrac{3 - \kappa}{8} \delta = \tilde{\delta}.
\end{equation*}
This means that we can apply Lemma~\ref{lemma:convergence_under_the_hood}. Setting $\mathcal{V} = \ballrx{\delta}{\bar{x}}$, $\mathcal{U} = \overline{\ballrx{\tilde{\delta}}{\bar{x}}}$, $\kappa_A = \kappa$ and $\kappa_B = 1$, we get the convergence to a point $\hat{x} \in A \cap B \cap \mathcal{U}$ at a rate
\begin{equation*}
    \norm{\hat{x} - a_n} \leq \tfrac{1 + \kappa}{1 - \kappa} \kappa^n \cdot \norm{b_1 - a_0}, \quad \norm{\hat{x} - b_{n+1}} \leq \tfrac{2 \kappa}{1 - \kappa} \kappa^{n} \cdot \norm{b_1 - a_0}, \quad n \in \N_0.
\end{equation*}
It remains to note that $\norm{b_1 - a_0} = \dist{a_0}{B}$.

Now we verify our assumptions. Let $n = 0$ and $d_0 =\norm{a_0 - \bar{x}} < \frac{1 - \kappa}{8} \delta$. As $d_0 < \delta/4$, Lemma~\ref{lemma:1step_1qoptimal} states that $b_1$ and $a_1$ exist and $\norm{a_1 - b_1} \leq \kappa \norm{b_1 - a_0}$. Moreover, 
\begin{align*}
    \norm{a_1 - \bar{x}} &\leq \norm{a_1 - b_1} + \norm{b_1 - a_0} + \norm{a_0 - \bar{x}} \\
    &\leq (2 + \kappa) d_0 = \tfrac{d_0}{1 - \kappa} \left[ 2 - \kappa(1 + \kappa)\right] < \delta / 4.
\end{align*}
Suppose that the sequences are well-defined up until $b_n$ and $a_n$, and that 
\begin{equation*}
    \norm{a_{n} - \bar{x}} \leq \tfrac{d_0}{1 - \kappa} \left[ 2 - \kappa^{n} (1 + \kappa) \right].
\end{equation*}
Then $\norm{a_{n} - \bar{x}} < \delta / 4$, and we can use Lemma~\ref{lemma:1step_1qoptimal} again to show that $b_{n+1}$ and $a_{n+1}$ exist and satisfy $\norm{a_{n+1} - b_{n+1}} \leq \kappa \norm{b_{n+1} - a_n}$. This leads to
\begin{align*}
    \norm{a_{n+1} - \bar{x}} &\leq \norm{a_{n+1} - b_{n+1}} + \norm{b_{n+1} - a_{n}} + \norm{a_{n} - \bar{x}} \\
    &\leq (1+\kappa) \kappa^n d_0 + \tfrac{d_0}{1 - \kappa} \left[ 2 - \kappa^{n} (1 + \kappa) \right] \\
    &= \tfrac{d_0}{1 - \kappa} \left[ 2 - \kappa^{n+1} (1 + \kappa) \right] < \delta / 4. 
\end{align*}
It follows that $(a_n)$ and $(b_n)$ are well-defined and $(a_n) \subset \ballrx{\delta / 4}{\overline{x}}$.
\end{proof}

\begin{corollary}
\label{corollary:1qoptimal_convergence_2sr}
Let $A, B \subseteq H$ be non-empty and intersect at $\bar{x} \in A \cap B$, where
\begin{itemize}
    \item $A$ and $B$ are locally closed and intersect transversally with $\cos(A,B,\bar{x}) = \bar{c} < 1$, and
    \item $A$ and $B$ are super-regular.
\end{itemize}
For every $c \in (\bar{c}, 1)$ there exists $\delta > 0$ such that for each $\sigma_A \geq 1$ satisfying $c\kappa < 1$ with $\kappa = \sigma_A \cdot c$ and each $a_0 \in A$ satisfying $\norm{a_0 - \bar{x}} < \frac{1 - c \kappa}{4[2 + \kappa(1-c)]} \delta$, the QAP \eqref{eq:qap} with $\sigma_B = 1$ converge to $\hat{x} \in A \cap B$ such that $\norm{\hat{x} - \bar{x}} < \tfrac{3}{8} \delta$ and, for $n \in \N_0$,
\begin{equation*}
    \norm{\hat{x} - a_n} \leq \tfrac{1 + \kappa}{1 - c \kappa} (c \kappa)^n \cdot \dist{a_0}{B}, \quad \norm{\hat{x} - b_{n+1}} \leq \tfrac{(1 + c) \kappa}{1 - c \kappa} (c\kappa)^{n} \cdot \dist{a_0}{B}.
\end{equation*}
\end{corollary}
\begin{proof}
The proof mimics that of Theorem~\ref{theorem:1qoptimal_convergence}, so we only highlight the differences. We need to make sure that the results of Lemma~\ref{lemma:1step_1qoptimal} can be applied to both $A$ and $B$. Therefore, we choose $c'$ and $\varepsilon$ as in Theorem~\ref{theorem:1qoptimal_convergence} and set $\delta = \min\{ \delta_A, \delta_B \}$, where $\delta_A$ and $\delta_B$ come from Lemma~\ref{lemma:1step_1qoptimal} for $A$ and $B$, respectively. It follows that
\begin{equation*}
    \norm{a_n - b_n} \leq \kappa \norm{b_n - a_{n-1}}, \quad \norm{b_{n+1} - a_n} \leq c\norm{a_n - b_n}, \quad n \in \N,
\end{equation*}
provided that the sequences are well-defined and $(a_n) \subset \ballrx{\delta / 4}{\bar{x}}$. This can be shown by a modified induction argument, where we prove that 
\begin{equation*}
    \norm{a_{n} - \bar{x}} \leq \tfrac{\norm{a_0 - \bar{x}}}{1 - c \kappa} \left[ 2 + \kappa(1-c)- (c\kappa)^{n} (1 + \kappa) \right] < \delta / 4. \qedhere
\end{equation*}
\end{proof}

\subsection{Both projections are quasioptimal}
Now, we even out the assumptions related to $A$ and $B$. We suppose that both sets are super-regular and consider quasioptimal metric projections with $\sigma_A > 1$ and $\sigma_B > 1$. In Lemma~\ref{lemma:1step_2qoptimal} we estimate the length of the steps taken by the QAP (see Fig.~\ref{fig:2qomp}). This result relies on the Pythagorean property for $B$ (unlike the previous Lemma~\ref{lemma:1step_1qoptimal}) and leads to our second convergence Theorem~\ref{theorem:2qoptimal_convergence}. 

\begin{figure}[ht!]
\centering
    \includegraphics[width=0.9\linewidth]{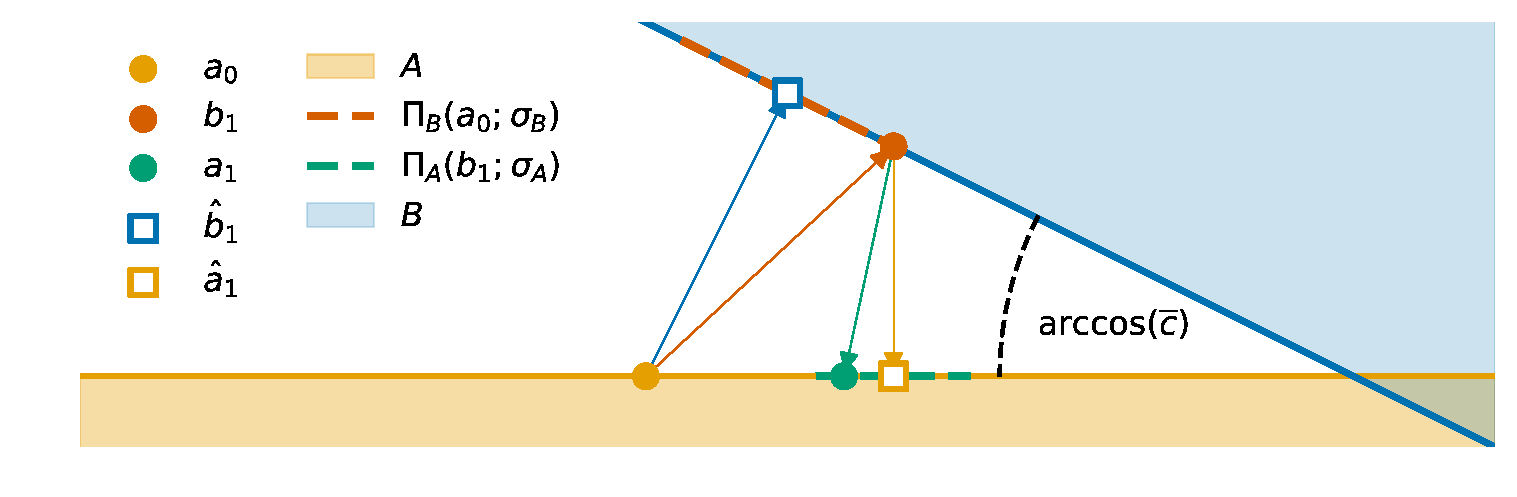}
\caption{The illustration of the settings of Lemma~\ref{lemma:1step_2qoptimal}.}
\label{fig:2qomp}
\end{figure}

\begin{lemma}
\label{lemma:1step_2qoptimal}
Let $A, B \subseteq H$ be non-empty and intersect at $\bar{x} \in A \cap B$, where
\begin{itemize}
    \item $A$ and $B$ are locally closed and intersect transversally with $\cos(A,B,\bar{x}) = \bar{c} < 1$, and
    \item $A$ and $B$ are super-regular.
\end{itemize}
Pick $\varepsilon > 0$ and $c \in (\bar{c}, 1)$. There exists $\delta > 0$ such that for every $\sigma_A \geq 1$ and $\sigma_B \geq 1$, each triplet of points $a_0 \in A \cap \ballrx{\delta}{\bar{x}}$ with $\Proj{B}{a_0} \cap \ballrx{\delta}{\bar{x}} \neq \emptyset$, $b_1 \in \Proj{B}{a_0; \sigma_B} \cap \ballrx{\delta}{\bar{x}}$ with $\Proj{A}{b_1} \cap \ballrx{\delta}{\bar{x}} \neq \emptyset$, and $a_1 \in \Proj{A}{b_1; \sigma_A}$ satisfies
\begin{gather*}
    \norm{a_1 - b_1} \leq \sigma_{A} f(\varepsilon, c, \sigma_B)  \cdot \norm{b_1 - a_0},\\
    f(\varepsilon, c, \sigma_B) = 2\varepsilon + c \left(\tfrac{1}{\sigma_B} - \varepsilon(1 + \varepsilon) \right) + \sqrt{1 - c^2} \sqrt{1 - \left(\tfrac{1}{\sigma_B} - \varepsilon(1 + \varepsilon) \right)^2}.    
\end{gather*}
If $a_0 \in A$ is such that $\norm{a_0 - \bar{x}} < \tfrac{\delta}{2 (\sigma_B + 1)}$ then $b_1$ and $a_1$ are well-defined, and the estimate holds for every $b_1 \in \Proj{B}{a_0}$.
\end{lemma}
\begin{proof}
For fixed $\varepsilon$ and $c$, we select $\delta$ just as in Corollary~\ref{corollary:1qoptimal_convergence_2sr} to ensure that super-regular properties of $A$ and $B$ together with Lemma~\ref{lemma:locally_transversal} hold in $\ballrx{\delta}{\bar{x}}$. We then follow the proof of Lemma~\ref{lemma:1step_1qoptimal} with some modifications. If $a_1 = b_1$ or $b_1 = a_0$ then the proof is complete. Otherwise, take an optimal metric projection $\hat{a}_1 \in \Proj{A}{b_1} \cap \ballrx{\delta}{\bar{x}}$, get the inequality $\norm{a_1 - b_1} \leq \sigma_A \cdot \norm{\hat{a}_1 - b_1}$ and expand the squared norm
\begin{equation*}
    \norm{\hat{a}_1 - b_1}^2 = \dotp{b_1 - \hat{a}_1}{a_0 - \hat{a}_1} + \dotp{\hat{a}_1 - b_1}{a_0 - b_1}.
\end{equation*}
The first term is bounded like in Lemma~\ref{lemma:1step_1qoptimal} based on the super-regularity of $A$:
\begin{equation*}
    \dotp{b_1 - \hat{a}_1}{a_0 - \hat{a}_1} \leq 2\varepsilon \norm{b_1 - \hat{a}_1} \cdot \norm{a_0 - b_1}.
\end{equation*}
In the second term, however, we can no longer guarantee that $a_0 - b_1 \in \NormalLim{B}{b_1}$ since $b_1$ is a quasioptimal projection. Let $\hat{b}_1 \in \Proj{B}{a_0} \cap \ballrx{\delta}{\overline{x}}$; then by Lemma~\ref{lemma:locally_transversal}, 
\begin{equation*}
    \dotp{\hat{a}_1 - b_1}{a_0 - \hat{b}_1} \leq c\norm{\hat{a}_1 - b_1} \cdot \norm{a_0 - \hat{b}_1},
\end{equation*}
since $\hat{a}_1 \neq \hat{b}_1$ and $a_0 \neq \hat{b}_1$. By virtue of Lemma~\ref{lemma:basp_superregular}, we have
\begin{equation*}
    \dotp{a_0 - b_1}{a_0 - \hat{b}_1} \geq \left( \tfrac{1}{\sigma_B} - \varepsilon(1 + \varepsilon) \right) \norm{a_0 - b_1} \norm{a_0 - \hat{b}_1}.
\end{equation*}
Consider three angles: $\alpha$ between $\hat{a}_1 - b_1$ and $a_0 - \hat{b}_1$, $\beta$ between $a_0 - b_1$ and $a_0 - \hat{b}_1$, and $\gamma$ between $\hat{a}_1 - b_1$ and $a_0 - b_1$.
The triangle inequality on the sphere gives $\gamma \geq \alpha - \beta$, so
\begin{equation*}
    \cos(\gamma) \leq \cos(\alpha - \beta) \leq \cos\left( \arccos(c) - \arccos\left( \tfrac{1}{\sigma_B} - \varepsilon(1 + \varepsilon) \right) \right)
\end{equation*}
and
\begin{equation*}
    \dotp{\hat{a}_1 - b_1}{a_0 - b_1} \leq (f(c, \varepsilon, \sigma_B) - 2\varepsilon) \cdot \norm{b_1 - \hat{a}_1} \norm{a_0 - b_1}.
\end{equation*}

Let us show that the points $b_1$ and $a_1$ are well-defined when $\norm{a_0 - \bar{x}} < \tfrac{\delta}{2 (\sigma_B + 1)}$. By Lemma~\ref{lemma:locally_closed_local_proj}, $\Proj{B}{a_0; \sigma_B} \neq \emptyset$ for all $\sigma_B \geq 1$ and $\norm{b_1 - \bar{x}} < \delta/2$ for every $b_1 \in \Proj{B}{a_0; \sigma_B}$. Then Lemma~\ref{lemma:locally_closed_local_proj} gives that $\Proj{A}{b_1; \sigma_A} \neq \emptyset$ for all $\sigma_A \geq 1$ and $\Proj{A}{b_1} \subseteq \ballrx{\delta}{\bar{x}}$.
\end{proof}

\begin{theorem}
\label{theorem:2qoptimal_convergence}
Let $A, B \subseteq H$ be non-empty and intersect at $\bar{x} \in A \cap B$, where
\begin{itemize}
    \item $A$ and $B$ are locally closed and intersect transversally with $\cos(A,B,\bar{x}) = \bar{c} < 1$, and
    \item $A$ and $B$ are super-regular.
\end{itemize}
For every $c \in (\bar{c}, 1)$ and every $\sigma_A, \sigma_B \in (1, \frac{1}{c})$ define $\kappa_{A}$ and $\kappa_{B}$ as
\begin{equation*}
    \kappa_{A} = \frac{\sigma_A}{\sigma_B} \left[ c + \sqrt{1 - c^2} \sqrt{\sigma_B^2 - 1} \right], \quad \kappa_{B} = \frac{\sigma_B}{\sigma_A} \left[ c + \sqrt{1 - c^2} \sqrt{\sigma_A^2 - 1} \right].
\end{equation*}
If $\kappa_{A} \cdot \kappa_{B} < 1$ then there exists $\delta > 0$ such that for every $a_0 \in A$ satisfying 
\begin{equation*}
    \norm{a_0 - \bar{x}} < \tfrac{1 - \kappa_A \kappa_B}{2(\sigma_B + 1)[\sigma_B + 1 + \kappa_A(\sigma_B-\kappa_B)]} \delta,
\end{equation*}
the QAP \eqref{eq:qap} converge to $\hat{x} \in A \cap B$ such that $\norm{\hat{x} - \bar{x}} < \tfrac{3}{8} \delta$ and, for $n \in \N_0$,
\begin{gather*}
    \norm{\hat{x} - a_n} \leq \tfrac{\sigma_B (1 + \kappa_{A})}{1 - \kappa_{A} \kappa_{B}} (\kappa_{A} \kappa_{B})^n \cdot \dist{a_0}{B},\\
    \norm{\hat{x} - b_{n+1}} \leq \tfrac{\sigma_B(1 + \kappa_{B}) \kappa_{A}}{1 - \kappa_{A} \kappa_{B}} (\kappa_{A} \kappa_{B})^n \cdot \dist{a_0}{B}.
\end{gather*}
\end{theorem}
\begin{proof}
Similar to Theorem~\ref{theorem:1qoptimal_convergence}, we want to choose $c' \in (\bar{c}, c)$, pick some $\varepsilon > 0$ and infer the corresponding value of $\delta$. Since the scaling factor in Lemma~\ref{lemma:1step_2qoptimal} is more complicated than in Lemma~\ref{lemma:1step_1qoptimal}, the implicit function theorem is required. We only outline the argument here, and the details can be found in Appendix~\ref{appendix:implicit_function_theorem}. 

The function $f$ from Lemma~\ref{lemma:1step_2qoptimal} satisfies $f(0, c, \sigma_A) = \frac{\kappa_{B}}{\sigma_B}$ and $f(0, c, \sigma_B) = \frac{\kappa_{A}}{\sigma_A}$. Provided that $c'$ is sufficiently close to $c$, the implicit function theorem guarantees that we can find $\varepsilon_{A}$ and $\varepsilon_{B}$ such that $f(\varepsilon_{B}, c', \sigma_A) = \frac{\kappa_{B}}{\sigma_B}$ and $f(\varepsilon_{A}, c', \sigma_B) = \frac{\kappa_{A}}{\sigma_A}$. In addition, $\varepsilon_A$ and $\varepsilon_B$ are positive when $\sigma_A < 1/c$ and $\sigma_B < 1/c$, which is assumed. 

We can then take $\varepsilon_A$ and choose $\delta_A > 0$ based on Lemma~\ref{lemma:1step_2qoptimal}. On reversing the roles of $A$ and $B$ in Lemma~\ref{lemma:1step_2qoptimal}, we also get $\delta_B > 0$ that corresponds to $\varepsilon_B$. Finally, we set $\delta = \min\{ \delta_{A}, \delta_{B} \}$. The rest is similar to the proof of Theorem~\ref{theorem:1qoptimal_convergence}. Thanks to Lemma~\ref{lemma:1step_2qoptimal}, the steps satisfy
\begin{equation*}
    \norm{a_n - b_n} \leq \kappa_A \norm{b_n - a_{n-1}}, \quad \norm{b_{n+1} - a_n} \leq \kappa_B \norm{a_n - b_n}, \quad n \in \N,
\end{equation*}
when $\norm{a_{n} - \bar{x}} < \tfrac{\delta}{2 (\sigma_B + 1)}$. This, in turn, follows from an inductive proof of
\begin{equation*}
    \norm{a_{n} - \bar{x}} \leq \tfrac{\norm{a_0 - \bar{x}}}{1 - \kappa_{A} \kappa_{B}} \left[ 1 + \sigma_B + \kappa_{A}(\sigma_B - \kappa_{B}) - \sigma_B (\kappa_{A} \kappa_{B})^{n} (1 + \kappa_{A}) \right].
\end{equation*}
We should also note that $\norm{b_1 - a_0} \leq \sigma_B \cdot \dist{a_0}{B}$.
\end{proof}

\subsection{Alternating projections as a quasioptimal metric projection}
\label{subsec:ap_as_qomp}
Let us consider the general implications of the theory we have developed by focusing on the similarities of our theorems. We have shown that when (i)~the two sets $A$ and $B$ are super-regular at some $\bar{x} \in A \cap B$, (ii)~we can compute sufficiently accurate metric projections, and (iii)~the starting point $a_0 \in A$ is sufficiently close to $\bar{x}$, the QAP converge to a point $\hat{x} \in A \cap B$, which is also close to $\bar{x}$, at a rate
\begin{equation*}
    \norm{\hat{x} - a_n} \leq \sigma_{AB} q^n \cdot \dist{a_0}{B}
\end{equation*}
for some $\sigma_{AB} \geq 1$ and $q \in (0, 1)$. For instance, $\norm{\hat{x} - a_0} \leq \sigma_{AB} \cdot \dist{a_0}{B}$. The trivial bound $\dist{a_0}{B} \leq \dist{a_0}{A \cap B}$ leads us to the observation that the AP \emph{act as a quasioptimal metric projection} of $a_0$ onto $B$ and $A \cap B$:
\begin{equation*}
    \hat{x} \in \Proj{B}{a_0; \sigma_{AB}} \cap \Proj{A \cap B}{a_0; \sigma_{AB}}.
\end{equation*}

Since both $A$ and $B$ are locally closed at $\bar{x}$, so is their intersection $A \cap B$. Then, by Lemma~\ref{lemma:locally_closed_local_proj}, there is a neighbourhood $\mathcal{V}$ of $\bar{x}$ such that every $x \in \mathcal{V}$ has $\Proj{A \cap B}{x} \neq \emptyset$. Even though it might be impossible to compute the optimal metric projection, our convergence theorems guarantee that there is a smaller neighbourhood $\mathcal{U} \subseteq \mathcal{V}$ such that we can construct a quasioptimal metric projection for every $x \in \mathcal{U}$. It suffices to apply the QAP to $a_0 \in \Proj{A}{x; \sigma_A}$ to obtain $\hat{x} \in A \cap B$ such that 
\begin{align*}
    \norm{\hat{x} - x} &\leq \norm{\hat{x} - a_0} + \norm{a_0 - x} \leq \sigma_{AB} \cdot \dist{a_0}{B} + \norm{a_0 - x} \\
    &\leq \sigma_{AB} \cdot \dist{a_0}{\Proj{B}{x}} + \norm{a_0 - x} \\
    &\leq \sigma_{AB} (\norm{a_0 - x} + \dist{x}{B}) + \norm{a_0 - x} \\
    &\leq \sigma_{AB} \cdot \dist{x}{B} + \sigma_A (\sigma_{AB} + 1) \cdot \dist{x}{A}.
\end{align*}
Then $\hat{x} \in \Proj{A \cap B}{x; \sigma_{AB} + \sigma_A (\sigma_{AB} + 1)}$, or $\hat{x} \in \Proj{B}{a_0; \sigma_{AB}} \cap \Proj{A \cap B}{x; \sigma_{AB}}$ if $x \in A$.

When $B$ is super-regular, it becomes possible to measure the distance between the limit $\hat{x}$ of the QAP and the metric projection $\Proj{B}{a_0}$ of the initial condition.

\begin{theorem}
\label{theorem:ap_as_qmp}
Let $A, B \subseteq H$ be non-empty and intersect at $\bar{x} \in A \cap B$, where
\begin{itemize}
    \item $A$ and $B$ are locally closed and intersect transversally with $\cos(A,B,\bar{x}) = \bar{c} < 1$, and
    \item $A$ and $B$ are super-regular.
\end{itemize}
For every $c \in (\bar{c}, 1)$ and $\sigma_A, \sigma_B \geq 1$, define $\kappa_A$ and $\kappa_B$ as
\begin{equation*}
    \kappa_A = \sigma_A \cdot c, \quad \kappa_B = c
\end{equation*}
when $\sigma_A \geq 1$ and $\sigma_B = 1$, or as
\begin{equation*}
    \kappa_{A} = \frac{\sigma_A}{\sigma_B} \left[ c + \sqrt{1 - c^2} \sqrt{\sigma_B^2 - 1} \right], \quad \kappa_{B} = \frac{\sigma_B}{\sigma_A} \left[ c + \sqrt{1 - c^2} \sqrt{\sigma_A^2 - 1} \right]
\end{equation*}
when $\sigma_A, \sigma_B \in (1, \tfrac{1}{c})$. Assume that $\kappa_{A} \cdot \kappa_{B} < 1$ and let $\sigma_{AB} > \frac{1 + \kappa_{A}}{1 - \kappa_{A} \kappa_{B}} \sigma_B$. Then for each $a_0 \in A$ that is sufficiently close to $\bar{x}$, the QAP \eqref{eq:qap} converge to $\hat{x} \in A \cap B$ such that
\begin{equation*}
    \dist{\hat{x}}{\Proj{B}{a_0}} \leq \sqrt{\sigma_{AB}^2 - 1} \cdot \dist{a_0}{B}.
\end{equation*}
\end{theorem}
\begin{proof}
Depending on the values of $\sigma_A$ and $\sigma_B$, we can use Corollary~\ref{corollary:1qoptimal_convergence_2sr} or Theorem~\ref{theorem:2qoptimal_convergence} to show that there exists $\delta > 0$ such that the QAP started from $a_0 \in A$ with $\norm{a_0 - \bar{x}} \lesssim \delta$ converge to $\hat{x} \in A \cap B$, which satisfies
\begin{equation*}
    \norm{\hat{x} - a_0} \leq \tfrac{1 + \kappa_{A}}{1 - \kappa_{A} \kappa_{B}} \sigma_B \cdot \dist{a_0}{B} \quad \text{and} \quad \norm{\hat{x} - \bar{x}} < \tfrac{3}{8} \delta.
\end{equation*}
Next, consider a strictly increasing function
\begin{equation*}
    g : \varepsilon \mapsto \varepsilon + \sqrt{\left( \tfrac{1 + \kappa_{A}}{1 - \kappa_{A} \kappa_{B}} \sigma_B \right)^2 - 1 + \varepsilon^2}.
\end{equation*}
Since $\sqrt{\sigma^2_{AB} - 1} > g(0)$, there is a unique $\tilde{\varepsilon} > 0$ such that $g(\tilde{\varepsilon}) = \sqrt{\sigma^2_{AB} - 1}$. As $B$ is super-regular at $\bar{x}$, we can pick $\tilde{\delta} > 0$ corresponding to $\tilde{\varepsilon}$ based on Corollary~\ref{corollary:qmp_close_to_omp}. Then
\begin{equation*}
    \dist{\hat{x}}{\Proj{B}{a_0}} \leq \sqrt{\sigma_{AB}^2 - 1} \cdot \dist{a_0}{B}
\end{equation*}
provided that $a_0 \in \ballrx{\tilde{\delta}/2}{\bar{x}}$ and $\hat{x} \in \ballrx{\tilde{\delta}}{\bar{x}}$. This can be achieved by reducing $\delta$.
\end{proof}

The same argument can be used to estimate the distance to $\Proj{A \cap B}{a_0}$ if the intersection $A \cap B$ itself is super-regular at $\bar{x}$. The situation is common when the sets are convex, smooth manifolds \cite{lewis2008alternating}, or prox-regular \cite{adly2016preservation}. Similar estimates were proved in \cite{andersson2013alternating} for the AP on manifolds.

\begin{corollary}
\label{corollary:ap_as_qmp}
In the setting of Theorem~\ref{theorem:ap_as_qmp}, assume that $A \cap B$ is also super-regular at $\bar{x}$. Then the QAP \eqref{eq:qap} converge to $\hat{x} \in A \cap B$ such that
\begin{equation*}
    \dist{\hat{x}}{\Proj{A \cap B}{a_0}} \leq \sqrt{\sigma_{AB}^2 - 1} \cdot \dist{a_0}{A \cap B}.
\end{equation*}
\end{corollary}

With $\sigma_B = 1$, our Corollary~\ref{corollary:ap_as_qmp} states that \begin{equation*}
    \dist{\hat{x}}{\Proj{A \cap B}{a_0}} \lesssim \frac{\sqrt{2 c \sigma_A} \sqrt{2 + c \sigma_A}}{1 - c^2 \sigma_A} \cdot \dist{a_0}{A \cap B}.
\end{equation*}
When $c \sigma_A \ll 1$, the QAP give a good approximation of the optimal metric projection onto $A \cap B$ in the sense that
\begin{equation*}
    \dist{\hat{x}}{\Proj{A \cap B}{a_0}} \lesssim 2 \sqrt{c \sigma_A} \cdot \dist{a_0}{A \cap B}.
\end{equation*}
\section{Low-rank approximation with constraints}
\label{sec:lowrank}
The QAP arise in the context of constrained low-rank approximation of matrices and tensors. We focus on two specific problems and treat them numerically\footnote{The code is available at \url{https://github.com/sbudzinskiy/LRAP}.} in Subsection~\ref{subsec:nonnegative} and Subsection~\ref{subsec:maximum}: imposing nonnegativity onto a low-rank approximation and computing low-rank approximations that are good in the maximum norm \eqref{eq:max_norm}. Before diving into the results of the experiments, we begin with an introduction to low-rank approximation (see also the supplementary material).

\subsection{Low-rank decompositions}
\label{subsec:lrm}
For a matrix $\Matrix{X} \in \Real^{m \times n}$, its rank is the number of linearly independent columns (or, equivalently, rows). If $\rank{\Matrix{X}} = r$, the matrix can be represented as a product
\begin{equation*}
    \Matrix{X} = \Matrix{U} \Matrix{V}^\intercal, \quad \Matrix{U} \in \Real^{m \times r}, \quad \Matrix{V} \in \Real^{n \times r}.
\end{equation*}
When such decomposition is known, it takes $(m+n)r$ real numbers to store $\Matrix{X}$ instead of the original $mn$. Similarly, the cost of the matrix-vector product $\Matrix{Xy}$ with $\Matrix{y} \in \Real^{n}$ changes from $\mathcal{O}(mn)$ arithmetic operations to $\mathcal{O}(mr+nr)$. The gains can be significant when the rank $r$ is low compared to the matrix sizes $m$ and $n$.

The effect is more profound for higher-order tensors: it is often impossible even to store all $n_1 \times \dots \times n_d$ entries of a tensor $\Tensor{X}$. There are several widely used tensor decompositions; we focus on the \emph{tensor-train} (TT) format \cite{oseledets2009breaking, oseledets2011tensor}:
\begin{equation*}
    \Tensor{X}(i_1, \ldots, i_d) = \sum_{\alpha_1 = 1}^{r_1} \dots \sum_{\alpha_{d-1} = 1}^{r_{d-1}} \Matrix{U}_1(i_1, \alpha_1) \Tensor{U}_2(\alpha_1, i_2, \alpha_2) \ldots \Matrix{U}_d(\alpha_{d-1}, i_d).
\end{equation*}
The matrices $\Matrix{U}_1 \in \Real^{n_1 \times r_1}$, $\Matrix{U}_d \in \Real^{r_{d-1} \times n_d}$ and tensors $\Tensor{U}_k \in \Real^{r_{k-1} \times n_k \times r_k}$, $2 \leq k \leq d-1$, are known as TT cores, and the numbers $\bm{r} = (r_1, \ldots, r_{d-1})$ are called the TT ranks of the decomposition. When a TT decomposition of $\Tensor{X}$ with small TT ranks is known, the storage requirements shrink from $\mathcal{O}(n^d)$ down to $\mathcal{O}(d n r^2)$, which makes it possible to work with extremely large tensors. The smallest possible TT ranks among the decompositions of $\Tensor{X}$ are called the TT ranks of the tensor and are denoted by $\ttrank{\Tensor{X}} \in \N^{d-1}$. In fact, the $k$th TT rank of $\Tensor{X}$ is equal to the rank of its $k$th unfolding matrix of size $n_1 \ldots n_{k} \times n_{k+1} \ldots n_d$.

\subsection{Matrices and tensors of fixed rank}
The matrices we get to encounter in applications are often full-rank, but they can be approximated with low-rank matrices. Consider the Frobenius inner product in $\Real^{m \times n}$,
\begin{equation*}
    \Dotp{\Matrix{X}}{\Matrix{Y}}{F} = \sum\nolimits_{i = 1}^{m} \sum\nolimits_{j = 1}^{n} \Matrix{X}(i, j) \Matrix{Y}(i, j),
\end{equation*}
and the set of rank-$r$ matrices
\begin{equation*}
    \mathcal{M}_{r} = \set{\Matrix{X} \in \Real^{m \times n}}{\rank{\Matrix{X}} = r}.
\end{equation*}
The problem of minimising the approximation error $\Norm{\Matrix{X} - \Matrix{X}_r}{F}$ over $\Matrix{X}_r \in \mathcal{M}_r$ is equivalent to computing the optimal metric projection $\Proj{\mathcal{M}_r}{\Matrix{X}}$. 

The set $\mathcal{M}_r$ is a smooth submanifold\footnote{As a side note, the set of matrices whose rank is smaller than or equal to $r$ is a closed algebraic variety, which is prox-regular at rank-$r$ matrices \cite{luke2013prox}.} of $\Real^{m \times n}$ \cite[Example~8.14]{lee2003introduction}. This result follows from an interpolation identity: if $\rank{\Matrix{X}} = r$ then there exist $r$ row-indices $I$ and $r$ column-indices $J$ such that
\begin{equation}
\label{eq:matrix_cross}
    \Matrix{X} = \Matrix{X}(:, J) \cdot \Matrix{X}(I, J)^{-1} \cdot \Matrix{X}(I, :).
\end{equation}
Being a smooth submanifold, $\mathcal{M}_r$ is locally closed at each point, so $\Proj{\mathcal{M}_r}{\Matrix{X}}$ is non-empty for every matrix $\Matrix{X}$ that is sufficiently close to $\mathcal{M}_r$ (recall Lemma~\ref{lemma:locally_closed_local_proj}). 

Matrix analysis fully describes the metric projections: $\Proj{\mathcal{M}_r}{\Matrix{X}}$ is non-empty if and only if $\rank{\Matrix{X}} \geq r$, and every $\Matrix{X}_r \in \Proj{\mathcal{M}_r}{\Matrix{X}}$ can be obtained as a truncated SVD of $\Matrix{X}$ \cite{trefethen1997numerical}; the SVD can be numerically computed in $\mathcal{O}(mn \cdot \min\{m,n\})$ operations \cite{demmel1997applied}.

A similar story holds for tensors with fixed TT ranks:
\begin{equation*}
    \mathcal{M}_{\bm{r}}^{TT} = \set{\Tensor{X} \in \Real^{n_1 \times \dots \times n_d}}{\ttrank{\Tensor{X}} = \bm{r} = (r_1, \ldots, r_{d-1})}.
\end{equation*}
The set is a smooth submanifold \cite{holtz2012manifolds}, and so the problem of finding the best approximation of $\Tensor{X}$ in the TT format is well-posed for all $\Tensor{X}$ sufficiently close to $\mathcal{M}_{\bm{r}}^{TT}$. However, when $d > 2$, we cannot \textit{compute} the optimal metric projection $\Proj{\mathcal{M}_{\bm{r}}^{TT}}{\Tensor{X}}$. This issue naturally leads us to quasioptimal metric projections.

\subsection{Quasioptimal low-rank approximation}
\label{subsec:qoptimal_lowrank}
It is possible to extend the idea of truncating the SVD from matrices to tensors, albeit with weaker outcomes. The \textsc{ttsvd} algorithm \cite{oseledets2011tensor} approximates a tensor $\Tensor{X}$ of order $d$ in the TT format by recursively truncating the SVDs of the unfoldings. Even though the resulting tensor $\Tensor{X}_{\bm{r}}$ with $\ttrank{\Tensor{X}_{\bm{r}}} = \bm{r}$ is not guaranteed to be the best approximation, it belongs to $\Proj{\mathcal{M}_{\bm{r}}^{TT}}{\Tensor{X}; \sqrt{d-1}}$. The computational complexity of \textsc{ttsvd} is $\mathcal{O}(n^{d+1})$ operations, dominated by the SVD of the first unfolding matrix.

Quasioptimal low-rank projections cannot be avoided in the higher-order setting, but they also appear in the context of matrices when we try to come up with \textit{faster} algorithms. The downside of the truncated SVD approach (\textsc{svd}) is that it computes all singular values and vectors of a matrix, while we actually require only a part of them. 

To reduce the complexity, we can try to approximate only the dominant singular subspaces or use a conceptually different method, unrelated to the SVD. We shall focus on three algorithms for fast low-rank matrix approximation:
\begin{itemize}
    \item the randomised SVD (\textsc{rsvd}, \cite{halko2011finding}) forms a small random \emph{sketch} of the matrix to approximate its dominant singular subspace; this requires $\mathcal{O}(mnr)$ operations, which is the same amount of work needed to form a rank-$r$ matrix from its factors;
    \item the cross approximation based on the maximum-volume principle (\textsc{vol}, \cite{goreinov2001maximal}) turns the interpolation identity \eqref{eq:matrix_cross} into an approximation tool by choosing the submatrix $\Matrix{X}(I,J)$ with the locally largest modulus of the determinant; good index sets $I$ and $J$ can be selected based on $\mathcal{O}(mr+nr)$ adaptively sampled entries of the matrix \cite{goreinov2010find} with the total complexity of $\mathcal{O}( (m+n)(r \cdot \mathrm{elem}(X) + r^2))$ operations, where $\mathrm{elem}(X)$ is the cost of computing a single element of $X$;
    \item the cross approximation based on the maximum-projective-volume principle (\textsc{pvol}, \cite{osinsky2018pseudo, osinsky2018rectangular}) generalises \eqref{eq:matrix_cross} even further and allows $I$ and $J$ to contain more indices than the desired rank (typically two-three times as many); it has the same asymptotic complexity as \textsc{vol}, but is able to construct better low-rank approximations at the expense of the larger hidden constant.
\end{itemize}
Find a more detailed discussion of these algorithms in the supplementary material. 

\subsection{Low-rank nonnegative approximation}
\label{subsec:nonnegative}

Matrices and tensors with nonnegative entries arise in applications related to images, video, recommender systems, probability, kinetic equations. Preserving nonnegativity in the low-rank approximation serves two main purposes: it keeps the object physically meaningful and helps to avoid potential numerical instabilities in the future processing.

The most popular technique is \emph{nonnegative matrix factorisation} \cite{gillis2020nonnegative}, it consists in searching for a low-rank approximation with nonnegative factors. However, the \emph{nonnegative rank} can be significantly larger than the usual rank \cite{beasley2009real}, which leads to slower postprocessing. In scientific computing, low-rank approximations are meant to make algorithms faster, so it is desirable to keep the rank as low as possible.

\subsubsection{Matrix approximation} 
An alternative point of view was suggested in \cite{vanluyten2008nonnegative}: instead of enforcing nonnegativity on the factors, we can try to find an optimal metric projection of the given matrix onto $\mathcal{M}_r \cap \Real_{+}^{m \times n}$. Some properties of the metric projection were investigated in \cite{grussler2015optimal}, where we can also find a numerical comparison of several algorithms, among which the AP (under the name \emph{lift-and-project}); the AP were later rediscovered in \cite{song2020nonnegative}.

\paragraph{Previous alternating-projection algorithms}
The truncated SVD of a matrix yields its optimal metric projection onto $\mathcal{M}_r$. To project onto $\Real_{+}^{m \times n}$, we simply need to set every negative entry of a matrix to zero. This means that the (idealistically exact) SVD-based AP
\begin{equation*}
    \Matrix{Z}_{k+1} = \Proj{\Real_{+}^{m \times n}}{\Matrix{Y}_k}, \quad \Matrix{Y}_{k+1} \in \Proj{\mathcal{M}_r}{\Matrix{Z}_{k+1}}, \quad k \in \N_0,
\end{equation*}
converge to $\mathcal{M}_r \cap \Real_{+}^{m \times n}$ locally around points of transversal intersection \cite{lewis2008alternating, lewis2009local, andersson2013alternating}. Such points do exist (e.g., rank-$r$ matrices with \emph{positive} entries), but we are not aware of concise ways to describe them in the entirety.

The computational complexities of these projections are not balanced: \textsc{svd} requires $\mathcal{O}(mn \cdot \min\{m,n\})$ operations and the nonnegative projection of a factorised matrix takes $\mathcal{O}(mnr)$ to form the full matrix and update its entries. So to significantly reduce the overall cost of one iteration, we should focus on the low-rank projection. In \cite{song2022tangent}, it was proposed to project the nonnegative matrix onto the \emph{tangent space} to $\mathcal{M}_r$ before projecting onto $\mathcal{M}_r$ itself. Such inexact alternating projections, which are similar in flavour to the subject of \cite{drusvyatskiy2019local}, have the complexity of $\mathcal{O}(mnr)$ operations per iteration. A different approach was taken in \cite{matveev2023sketching}, where the use of \textsc{rsvd} reduced the cost to $\mathcal{O}(mnr)$ operations, but with a potentially smaller constant than in \cite{song2022tangent}.

\paragraph{New alternating-projection algorithms}
We suggest a way to reduce the complexity of the two successive projections \emph{below} the cost of the nonnegative projection.\footnote{The same ideas are applicable to other entrywise constraints.} The \textsc{vol} and \textsc{pvol} methods operate by adaptively sampling $\mathcal{O}(mr + nr)$ elements of a matrix, which means that $\mathcal{O}(mnr)$ operations need not be spent to form the full matrix for the nonnegative projection. Instead, we can compute each required element with $\mathcal{O}(r)$ operations and make it nonnegative on the fly. This gives the asymptotic complexity of $\mathcal{O}(m r^2 + n r^2)$ operations per AP iteration.

The hidden constant here depends on the number of times the index sets $I$ and $J$ need to be updated. We can expect to reduce the number of these updates if we use the previously computed $I$ and $J$ as the starting index sets for the next step --- we shall call this approach \emph{warm-start} as opposed to the \emph{cold-start} approach, where we always start with new random $I$ and $J$.

\paragraph{Example: Solution to a two-component coagulation equation}
We test the performance of the proposed algorithms on an example from \cite{matveev2023sketching}. The following nonnegative function
\begin{equation}
\label{eq:smolukh}
    \nu(v_1, v_2, t) = \frac{e^{-v_1 - v_2}} {(1 +  t /2)^2} \cdot I_0 \left(2 \sqrt{\frac{v_1 v_2 t}{ t + 2}}\right),
\end{equation}
where $I_0$ is the modified Bessel function of order zero, is the solution to a Smoluchowski coagulation equation. Let $\Matrix{X}$ be the $1024 \times 1024$ matrix that we get by discretising the function at $t = 6$ on an equidistant tensor grid with a step of $0.1$ over $[0, \infty) \times [0, \infty)$.

We show $\Matrix{X}$ and its singular values in Fig.~\ref{fig:smolukh} and plot the absolute value of the smallest negative entry of each $\Matrix{X}_r \in \Proj{\mathcal{M}_r}{\Matrix{X}}$. The approximations for $r = 0$ or $r = 1$ are nonnegative (by the Perron--Frobenius theorem for $r = 1$). But $\Matrix{X}_r$ contains negative entries for every $r > 1$, and they decay at the same rate as the singular values.

\begin{figure}[ht]
\centering
	\includegraphics[width=\textwidth]{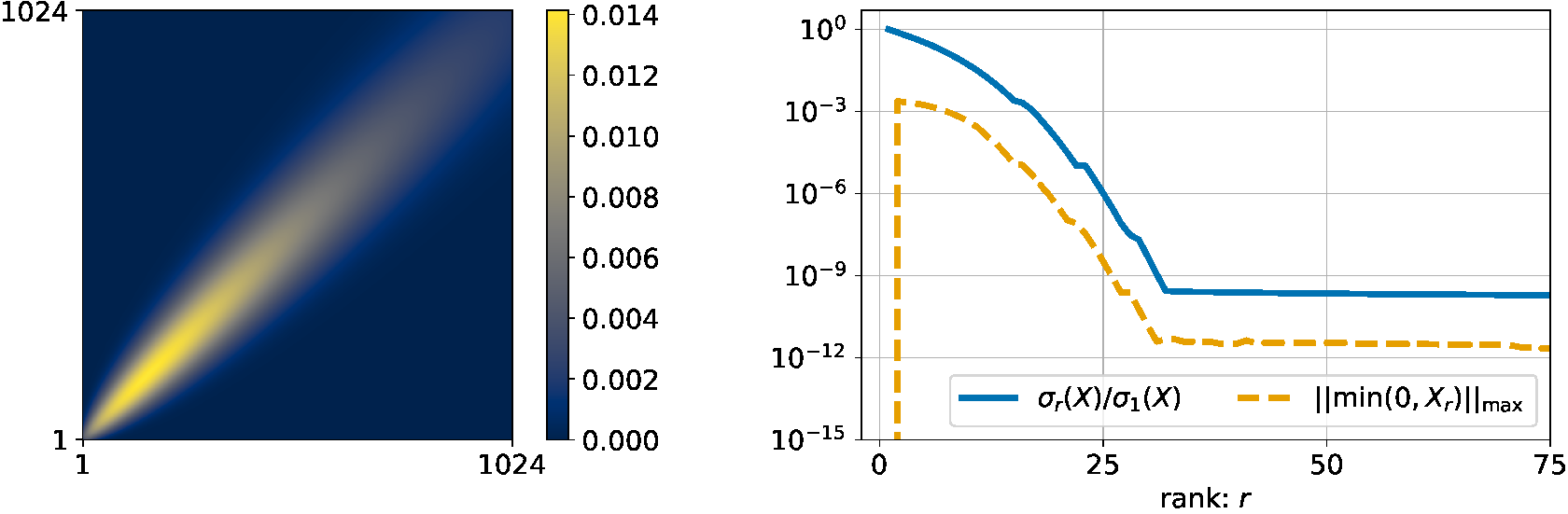}
\caption{Properties of the discretised solution \eqref{eq:smolukh} to a two-component Smoluchowski equation: (left)~the solution itself; (right)~its normalised singular values and the modulus of the smallest negative entry of its best rank-$r$ approximation for varying $r$.}
\label{fig:smolukh}
\end{figure}

\paragraph{Alternating projections with different low-rank projections} Let us try to refine $\Matrix{X}_{10}$ to make it nonnegative. We compare six AP algorithms that use \textsc{svd}, \textsc{rsvd}, warm-started $\textsc{vol}_w$, cold-started $\textsc{vol}_c$, warm-started $\textsc{pvol}_w$ with $|I| = |J| = 30$ and cold-started $\textsc{pvol}_c$ with $|I| = |J| = 30$.

We care about two performance metrics: the rate at which the low-rank iterates $\Matrix{Y}_k$ approach the nonnegative orthant and the approximation error $\Norm{\Matrix{X} - \Matrix{Y}_k}{F}$ and how it compares with the initial error $\Norm{\Matrix{X} - \Matrix{Y}_0}{F} = \dist{\Matrix{X}}{\mathcal{M}_{10}}$. We ran 10 random experiments with different random seeds for every type of the low-rank projection (except for \textsc{svd}) and performed 1000 iterations of the AP. In Fig.~\ref{fig:smolukh_ap}, we show how the Frobenius norm of the negative elements $\Norm{\min(0, \Matrix{Y}_k)}{F}$ changes in the course of iterations by plotting its median together with the $25\%$ and $10\%$ percentiles. The median approximation errors after 10 and 1000 iterations are collected in Tab.~\ref{tab:smolukh_ap}.

\begin{figure}[ht]
\centering
    \includegraphics[width=0.6\linewidth]{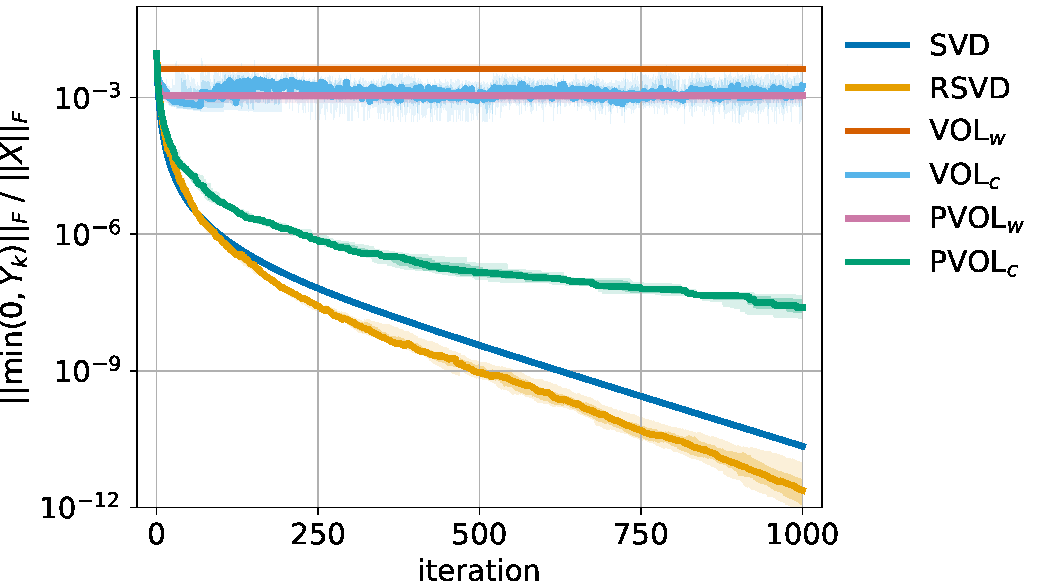}
\caption{Convergence of the low-rank iterates $\Matrix{Y}_k$ to $\Real_{+}^{m \times n}$ for the alternating projections with inexact low-rank projections: the median, the $25\%$ and the $10\%$ percentiles.}
\label{fig:smolukh_ap}
\end{figure}

\begin{table}[ht]
    \centering
    \caption{Median approximation errors of the low-rank iterates $\Matrix{Y}_k$ after $10$ and $1000$ steps of the alternating projections with inexact low-rank projections.}
    \label{tab:smolukh_ap}
    \begin{tabular*}{\textwidth}{@{\extracolsep\fill}ccccccc@{}}
        \toprule
        Error growth & \textsc{svd} & \textsc{rsvd} & $\textsc{vol}_w$ & $\textsc{vol}_c$ & $\textsc{pvol}_w$ & $\textsc{pvol}_c$ \\
        \midrule
        $\tfrac{\Norm{\Matrix{X} - \Matrix{Y}_{10}}{F}}{\Norm{\Matrix{X} - \Matrix{Y}_{0}}{F}}$ & $1.125$ & $1.221$ & $1.097$ & $1.150$ & $1.137$ & $1.134$ \\
        $\tfrac{\Norm{\Matrix{X} - \Matrix{Y}_{1000}}{F}}{\Norm{\Matrix{X} - \Matrix{Y}_{0}}{F}}$ & $1.132$ & $1.226$ & $1.097$ & $4.834$ & $1.139$ & $1.144$ \\
        \bottomrule\\
    \end{tabular*}
\end{table}

The linear convergence of \textsc{ap-svd} and \textsc{ap-rsvd} has already been observed in \cite{matveev2023sketching}. The warm-started $\textsc{ap-vol}_w$ and $\textsc{ap-pvol}_w$, however, fail to converge: it takes 1 iteration for the former and about 40 for the latter to reach the plateau. The explanation for such behaviour lies in the cross-approximation approach itself. While \textsc{vol} does not satisfy the interpolation property \eqref{eq:matrix_cross} per se when applied to a matrix of higher rank, it still interpolates the selected rows and columns. This means that $\textsc{ap-vol}_w$ at iteration $k > 1$ deals with the rows $\Matrix{Y}_k(I,:)$ and columns $\Matrix{Y}_k(:,J)$ that are unaffected by the nonnegative projection and with the submatrix $\Matrix{Y}_k(I,J)$ that has locally largest volume by construction. It follows that $\Matrix{Y}_{k+1} = \Matrix{Y}_k$, because \textsc{vol} does not know what happens outside of its index sets. For the \textsc{pvol} algorithm with $|I| > r$ and $|J| > r$, it is still likely that the index sets $I$ and $J$ will cease to change after a couple of iterations of $\textsc{ap-pvol}_w$ since the projective volume will be largely determined by the positive entries of the submatrix (provided that the negative entries are sufficiently small in the absolute value); with fixed $I$ and $J$, the AP will only affect $\Matrix{Y}_k(I,:)$ and $\Matrix{Y}_k(:,J)$, making them nonnegative and staying ignorant to all the other entries of the matrix.

The main issue of $\textsc{ap-vol}_w$ and $\textsc{ap-pvol}_w$ is that only the \emph drives the low-rank projection step; with the cold-started AP, we attempt to inject some \emph{global information} into cross approximation. As Fig.~\ref{fig:smolukh_ap} demonstrates, the freshly generated index sets are not sufficient to make $\textsc{ap-vol}_c$ converge to a nonnegative matrix: they only lead to irregular behaviour and growth of the approximation error (see Tab.~\ref{tab:smolukh_ap}). At the same time, $\textsc{ap-pvol}_c$ begins to converge, even if slower than \textsc{ap-svd} and \textsc{ap-rsvd}. 

The approximation errors achieved by the convergent methods (\textsc{ap-svd}, \textsc{ap-rsvd}, and $\textsc{ap-pvol}_c$) are only $20\%$ larger than the initial, smallest error. Another observation from Tab.~\ref{tab:smolukh_ap} is that most of the growth happens during the first 10 iterations.

\paragraph{Accelerated alternating projections}
Thousand iterations were not enough for \textsc{ap-svd} and \textsc{ap-rsvd} to produce a nonnegative low-rank matrix. We can try to improve their convergence rates by modifying the AP framework itself. Since the nonnegative orthant $\Real_{+}^{m \times n}$ is an obtuse convex cone, we can use the \textit{reflection-projection} method \cite{bauschke2004reflection},
\begin{equation*}
    \Matrix{Z}_{k+1} = 2 \cdot \Proj{\Real_{+}^{m \times n}}{\Matrix{Y}_k} - \Matrix{Y}_k, \quad \Matrix{Y}_{k+1} \in \Proj{\mathcal{M}_r}{\Matrix{Z}_{k+1}}, \quad k \in \N_0,
\end{equation*}
where the reflection is nothing but the entrywise absolute value $|\Matrix{Y}_k|$. The main idea here is that the reflection pushes the current iterate towards the interior of $\Real_{+}^{m \times n}$. 

Another possibility is to replace the nonnegative projection with a scalar shift, which can be chosen based on how far $\Matrix{Y}_k$ is from being nonnegative. We propose the \textit{shift-projection} method with a parameter $\beta > 0$:
\begin{equation*}
    \alpha_k = \beta \cdot \Norm{\min(0, \Matrix{Y}_k)}{max}, \quad \Matrix{Z}_{k+1} = \Matrix{Y}_k + \alpha_k, \quad \Matrix{Y}_{k+1} \in \Proj{\mathcal{M}_r}{\Matrix{Z}_{k+1}}, \quad k \in \N_0.
\end{equation*}
The additive shift is similar to the projection when it acts on the smallest negative entry and is similar to the reflection when it acts on the negative entries that are of the same order as the smallest one. However, unlike what happens with the projection and reflection, the positive entries change as well. In addition, the scalar shift increases the rank of $\Matrix{Y}_k$ at most by one, so the SVD of $\Matrix{Z}_{k+1}$ can be computed efficiently.

Computing $\alpha_k$ is easy when we allow ourselves to explicitly form $\Matrix{Y}_k$ by multiplying the low-rank factors --- this is exactly what we try to avoid with \textsc{vol} and \textsc{pvol}, though. We can come up with a surrogate for $\Norm{\min(0, \Matrix{Y}_k)}{max}$ by applying rank-one \textsc{vol} to the entrywise $\mathrm{arccot}(\Matrix{Y}_k)$: the entry of largest absolute value in the transformed matrix corresponds to the smallest negative entry of the original matrix. To achieve better coverage, we start with a random initial position each time we calculate $\alpha_k$.

We test the modified schemes in the same settings as before. There are 18 algorithms to compare, which use 6 different low-rank projections and 3 versions\footnote{In the supplementary material, we present results for one more version of the AP.} of the AP: the original AP, the reflection-projections, and the shift-projections with $\beta = 1/2$. We present the results in Fig.~\ref{fig:smolukh_modified_ap} and Tab.~\ref{tab:smolukh_modified_ap}. 

\begin{figure}[h!]
\centering
	\includegraphics[width=0.9\textwidth]{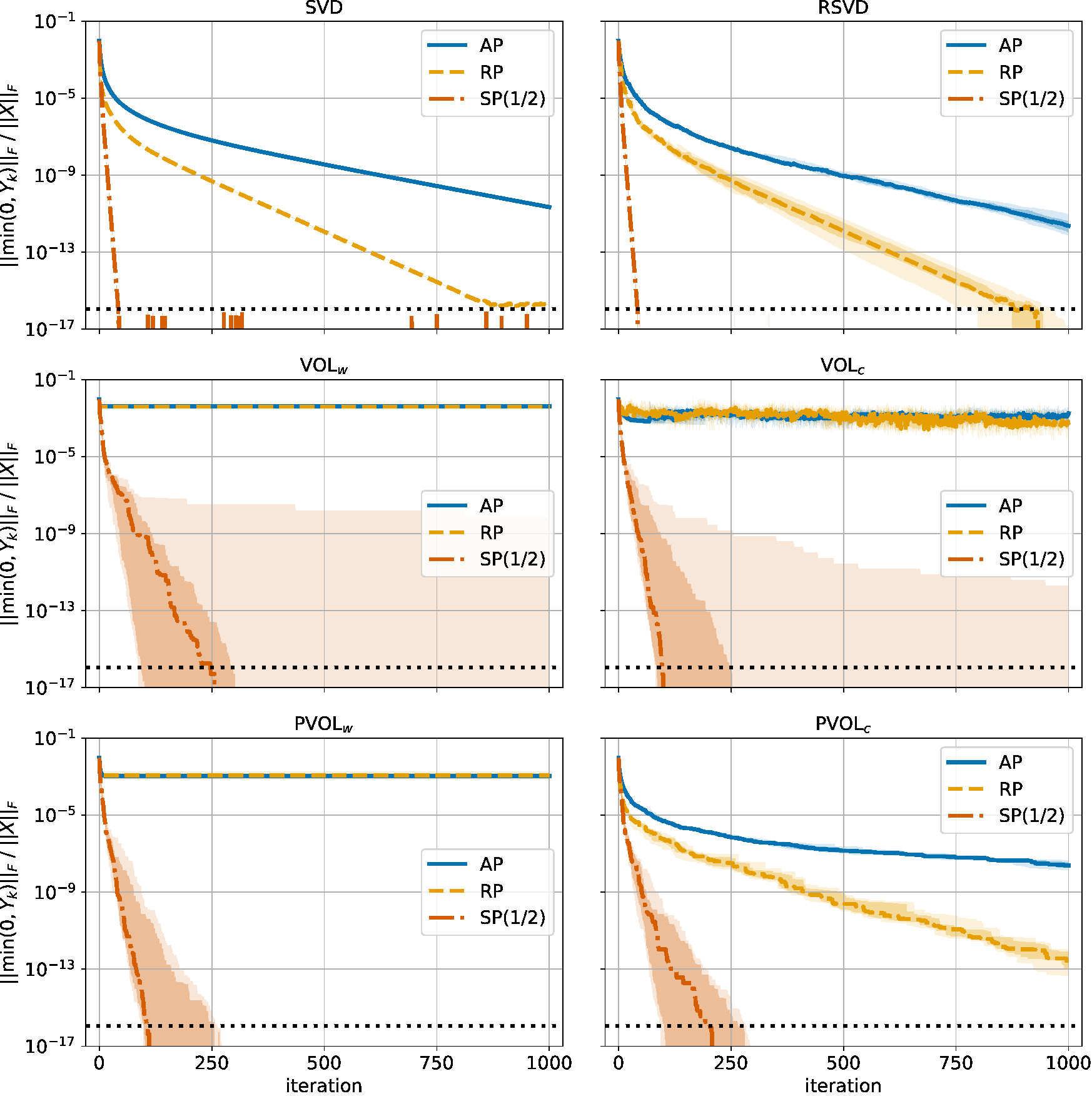}
\caption{Convergence of the low-rank iterates $\Matrix{Y}_k$ to the nonnegative orthant for the modified alternating-projection schemes with inexact low-rank projections: the median, the $25\%$ and the $10\%$ percentiles.}
\label{fig:smolukh_modified_ap}
\end{figure}

\begin{table}[ht]
    \centering
    \caption{Median approximation errors of the low-rank iterates $\Matrix{Y}_k$ after $10 / 1000$ steps of modified alternating-projection schemes with inexact low-rank projections.}
    \label{tab:smolukh_modified_ap}
    \begin{tabular*}{\textwidth}{@{\extracolsep\fill}lcccccc@{}}
        \toprule
        Modification  & \textsc{svd} & \textsc{rsvd} & $\textsc{vol}_w$ & $\textsc{vol}_c$ & $\textsc{pvol}_w$ & $\textsc{pvol}_c$ \\
        \midrule
        \textsc{ap} & $1.13 / 1.13$ & $1.22 / 1.23$ & $1.10 / 1.10$ & $1.15 / 4.83$ & $1.14 / 1.14$ & $1.13 / 1.14$ \\
        \textsc{rp} & $1.19 / 1.19$ & $1.48 / 1.48$ & $1.33 / 1.33$ & $2.35 / 9.41$ & $1.24 / 1.24$ & $1.23 / 1.23$ \\
        $\textsc{sp}(\tfrac{1}{2})$ & $5.78 / 5.78$ & $5.78 / 5.78$ & $5.55 / 5.76$ & $5.62 / 5.79$ & $5.60 / 5.79$ & $5.65 / 5.79$ \\
        \bottomrule\\
    \end{tabular*} 
\end{table}

The modifications \textsc{rp-svd}, \textsc{rp-rsvd} and $\textsc{rp-pvol}_c$ converge faster than their original counterparts and produce slightly higher approximation errors. The $\textsc{sp}(\tfrac{1}{2})$ algorithm converges after a \emph{finite} number of steps for \emph{all} low-rank projections with the resulting approximation error about $5.5$ times larger than the initial error. The shift-projections appear to be `aggressive' enough with injecting the global information to make the warm-started $\textsc{sp}(\tfrac{1}{2})$-$\textsc{vol}_w$ and $\textsc{sp}(\tfrac{1}{2})$-$\textsc{pvol}_w$ converge. The performance of the latter is more consistent, though: it converged in each random experiment.

\subsubsection{Tensor-train approximation}
Next, we evaluate the performance of the accelerated QAP by computing low-rank nonnegative TT approximations with \textsc{ttsvd}. Linear convergence of the standard QAP for this problem was numerically established in \cite{sultonov2023low}. A different approach was proposed in \cite{jiang2023nonnegative}: instead of computing quasioptimal metric projections onto the set of low-rank tensors, the authors used optimal metric projections onto several sets of low-rank matrices. Low-rank nonnegative TT approximations were also studied in \cite{shcherbakova2022fast}.

We present an experiment that illustrates what happens to the negative elements during the AP. Let $f \in \Real_{+}^{1024}$ be a discretised mixture of one-dimensional Gaussian probability densities in Fig.~\ref{fig:mixture}. We can consider $f$ as a tenth-order $2 \times 2 \times \cdots \times 2$ tensor and approximate it in the TT format --- this gives the so-called \emph{quantised} TT (QTT) approximation of the original vector. We compressed $f$ using \textsc{ttsvd} with relative accuracy $10^{-2}$ to get a QTT approximation of rank $\bm{r} = (2,2,2,3,3,4,5,4,2)$ and the actual relative approximation error of $6.97 \cdot 10^{-3}$.

\begin{figure}[t!]
\centering
    \includegraphics[width=0.45\linewidth]{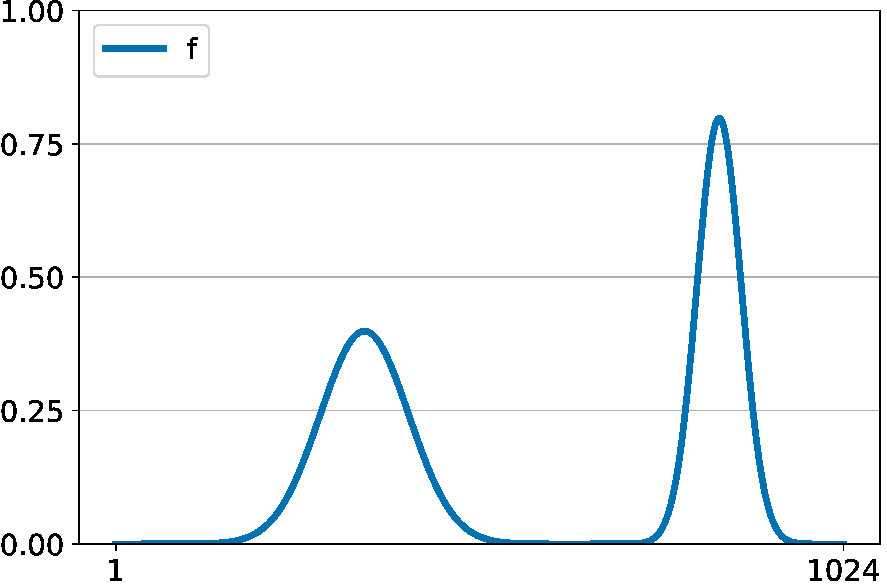}
\caption{Mixture of one-dimensional Gaussian probability densities.}
\label{fig:mixture}
\end{figure}

The QTT approximation produced by \textsc{ttsvd} contains negative entries, and Fig.~\ref{fig:mixture_negative} shows that they appear as intervals rather than separate points. In Fig.~\ref{fig:mixture_negative}, we compare the initial QTT approximation with the 10th iterations of \textsc{ap-ttsvd}, \textsc{rp-ttsvd}, and $\textsc{sp}(\tfrac{1}{2})$-\textsc{ttsvd}: even after a small number of iterations, the number of negative elements is significantly reduced until only a handful of individual points is left. With further iterations, the remaining entries decay at a linear rate, and the first 50 iterations are responsible for the most rapid decrease of their Frobenius norm (see Fig.~\ref{fig:mixture_convergence}). The approximation errors obtained with \textsc{ap-ttsvd}, \textsc{rp-ttsvd}, and $\textsc{sp}(\tfrac{1}{2})$-\textsc{ttsvd} increase by a factor of $1.08$, $1.09$, $2.68$ after 10 iterations and by a factor of $1.09$, $1.09$, $2.68$ after 1000 iterations, respectively. See also the supplementary material.

\begin{figure}[ht!]
\centering
	\includegraphics[width=0.9\textwidth]{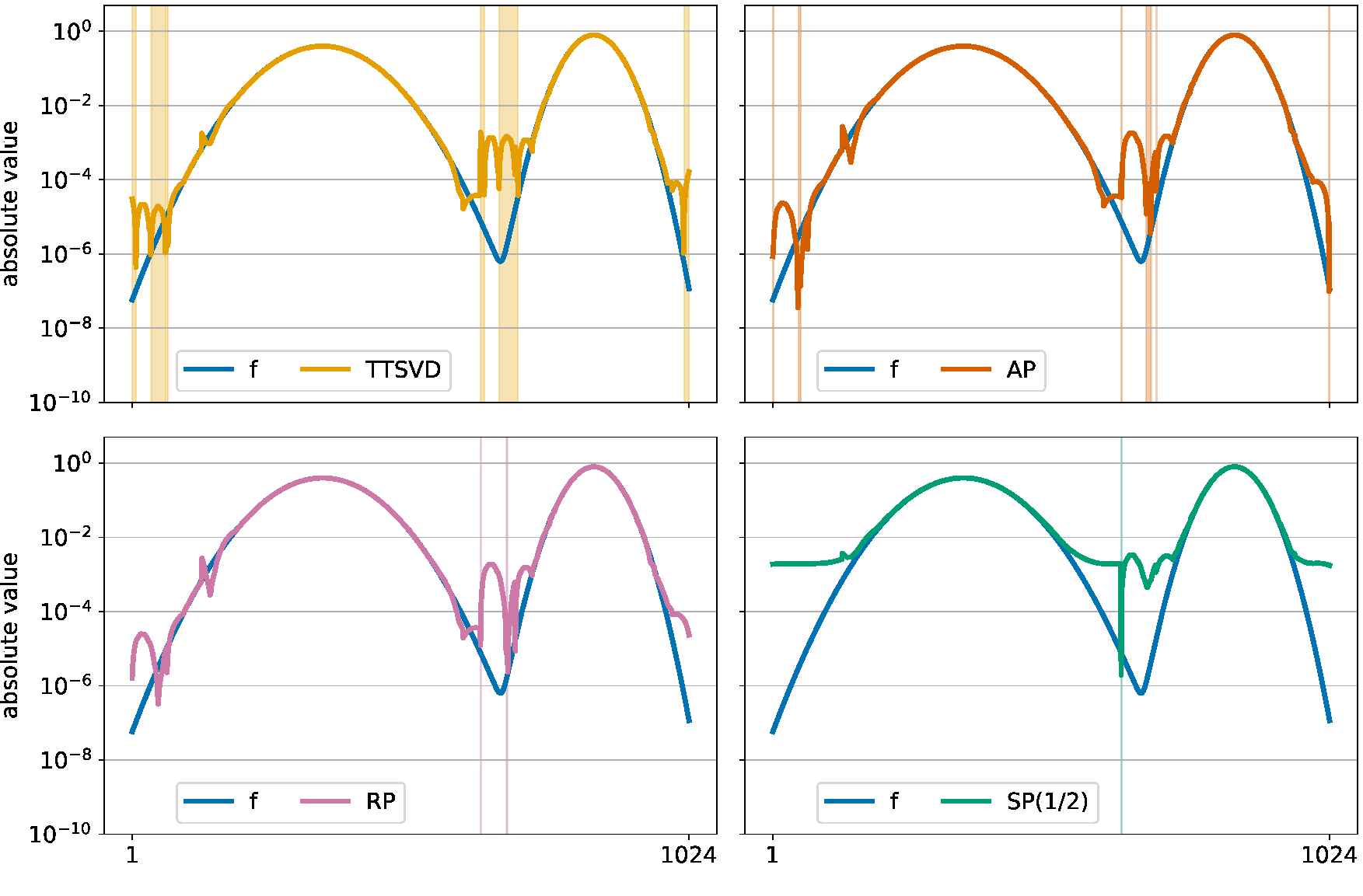}  
\caption{Mixture of one-dimensional Gaussian probability densities and the absolute value of its QTT approximations with highlighted negative entries: the \textsc{ttsvd} and its refinements after 10 iterations of \textsc{ap-ttsvd}, \textsc{rp-ttsvd} and $\textsc{sp}(\tfrac{1}{2})$-\textsc{ttsvd}.}
\label{fig:mixture_negative}
\end{figure}

\begin{figure}[h!]
\centering
	\includegraphics[width=0.9\textwidth]{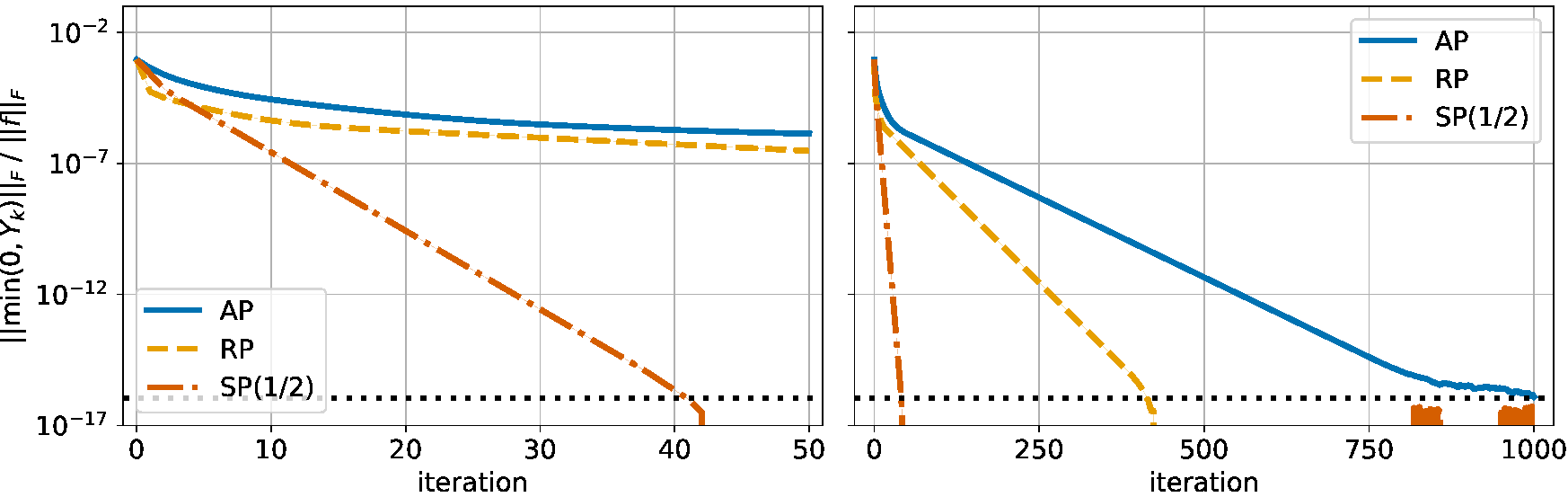}
\caption{Convergence of the QTT approximations to the nonnegative orthant for the modified alternating-projection schemes based on \textsc{ttsvd}.}
\label{fig:mixture_convergence}
\end{figure}

\subsubsection{Regularised matrix completion}
Out previous experiments show that the first iterations of the AP play a significant role in computing low-rank nonnegative approximations: most of the negative elements disappear during this stage, and the approximation error almost stabilises. This motivates us to incorporate a small number of iterations of the AP into computational low-rank routines as a way to regularise their solutions towards nonnegativity.

Consider the matrix completion problem. Let $\Matrix{X}_* \in \mathcal{M}_r$, and let $\Omega \subseteq \{1, \ldots, m\} \times \{1, \ldots, n\}$ be a set of index pairs. Assuming that we know only those entries of $\Matrix{X}_*$ that are indexed by $\Omega$, we aim to recover the whole matrix. One of the possible ways to approach the problem is via Riemannian optimisation \cite{vandereycken2013low}. Without going into the algorithmic details, let us note that a single step of the Riemannian gradient descent takes  $\mathcal{O}(\max\{m,n\} r^2 + |\Omega|r)$ operations.

An important question is how many entries $|\Omega|$ are needed to recover $\Matrix{X}_*$. The manifold $\mathcal{M}_r$ is of dimension $r(m + n - r)$, so at least this many samples are required, but in general we need $|\Omega| = \gamma \cdot r(m + n - r)$ with some $\gamma > 1$. If there is any additional information about the matrix, it can be used to recover $\Matrix{X}_*$ from a smaller sample.

Assume that $\Matrix{X}_*$ is nonnegative. We propose to use a small number of warm-started $\textsc{ap-vol}_w$ iterations after each gradient step as regularisation. Crucially, the asymptotic complexity of the regularised Riemannian gradient descent remains the same. To test the idea, we fix $m = n = 1000$, $r = 5$, and generate 20 random nonnegative matrices $\Matrix{X}_*$ as products of $n \times r$ low-rank factors with independent identically distributed entries from the uniform distribution on the interval $(0, n^{-1/2})$. Next, for each value of the oversampling parameter $\gamma$ and each test matrix, we generate $\gamma \cdot r(m + n - r)$ index pairs uniformly at random with replacement and a rank-$r$ initial condition as a product of two random $n \times r$ matrices with orthonormal columns. We then run 1000 iterations of the Riemannian gradient descent and 1000 iterations of the Riemannian gradient descent regularised with \emph{one} iteration of $\textsc{ap-vol}_w$ after each gradient step --- from the same initial condition. For each value of $\gamma$, we calculate how many of the 20 problem instances achieve the relative reconstruction error of $10^{-6}$ in the Frobenius norm and plot the corresponding frequency of `success'. The results in Fig.~\ref{fig:completion} show that such simple regularisation makes it possible to recover low-rank nonnegative matrices from \emph{far fewer} entries than required by the original Riemannian gradient descent.

\begin{figure}[ht!]
\centering
    \includegraphics[width=0.45\linewidth]{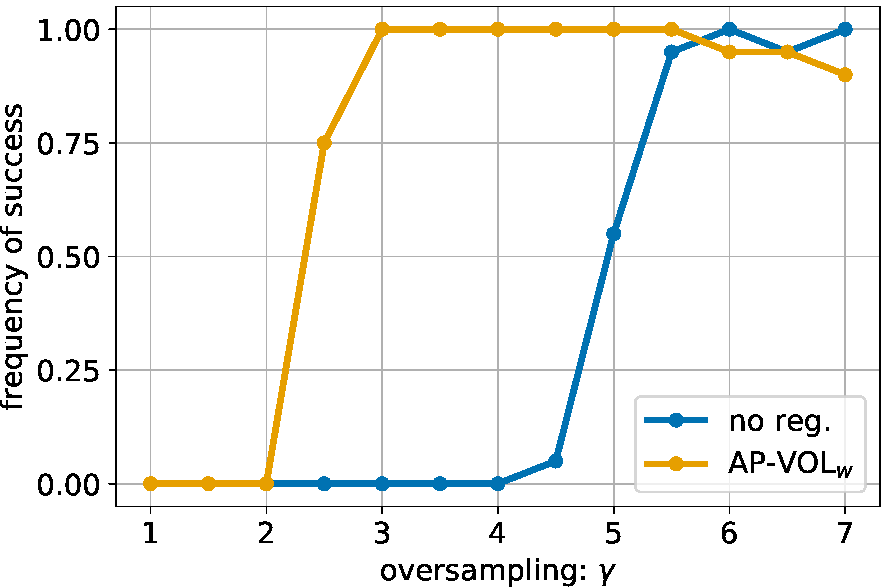}
\caption{Riemannian matrix completion with alternating-projections regularisation can recover nonnegative matrices from fewer entries.}
\label{fig:completion}
\end{figure}

The performance of $\textsc{ap-vol}_w$ for matrix completion suggests that fast AP algorithms can be used as regularisers for other problems in scientific computing. The most obvious possible extension is TT completion: Riemannian optimisation \cite{budzinskiy2023tensor} can be combined with the AP based on TT cross approximation \cite{oseledets2010tt}. This approach could also be applied to other low-rank reconstruction problems, such as non-parametric estimation of probability densities \cite{rohrbach2022rank}.

Computational PDEs is another potential\footnote{Since the first publication of our article, the AP have been applied to the Smoluchowski equation \cite{matveev2025nonnegative}.} area of application: the solutions to certain PDEs are naturally nonnegative, and low-rank numerical schemes \cite{kazeev2014direct,manzini2021nonnegative} should preserve this property to keep the numerical solutions physically meaningful. In addition, the appearance of negative entries could lead to numerical instabilities.

\subsection{Low-rank approximation in the maximum norm}
\label{subsec:maximum}

How well can a matrix be approximated with a matrix of low rank? The usual answer is: look at the singular values. Their decay determines the smallest error of low-rank approximation in the spectral norm, the Frobenius norm, and other unitarily invariant norms. An approximation guarantee for the maximum norm \eqref{eq:max_norm}, which is not unitarily invariant, was obtained in \cite{srebro2005rank}: any $n \times n$ matrix $\Matrix{X}$ can be approximated as $\Norm{\Matrix{X - Y}}{max} \leq \varepsilon \Norm{\Matrix{X}}{2}$ with $\rank{\Matrix{Y}} \leq \lceil 9 \log(3n^2) / \varepsilon^2\rceil$.

When $\varepsilon$ is fixed, one can find a rank-one $\Matrix{Y}$ such that $\Norm{\Matrix{X} - \Matrix{Y}}{\max} \leq \varepsilon$, or show that it does not exist, in a finite number of operations, but the problem is NP-hard \cite{gillis2019low}. The minimisation of $\Norm{\Matrix{X} - \Matrix{Y}}{\max}$ was analysed in \cite{zamarashkin2022best,morozov2023optimal}, but the convergence to the global minimum remains unproved for $r > 1$. The numerical experiments of \cite{zamarashkin2022best} with random matrices and $r = \sqrt{n}$ suggest a bound $\inf_{\Matrix{Y} \in \mathcal{M}_r} \Norm{\Matrix{X} - \Matrix{Y}}{\max} \leq \log^{0.6}(n) / \sqrt{n}$ that is asymptotically tighter than the bound of \cite{srebro2005rank} by a factor of $n^{1/4}$.

We propose to numerically estimate $ \inf_{\Matrix{Y} \in \mathcal{M}_r} \Norm{\Matrix{X} - \Matrix{Y}}{\max}$ with the AP.\footnote{This approach can be used with other vectorized $\ell_p$ norms; the case we consider is $p = \infty$.} For the given $\Matrix{X} \in \Real^{m \times n}$ and $r < \rank{X}$, we choose an initial approximation $\Matrix{Y}_0 \in \mathcal{M}_r$ and set $\varepsilon_{+} = \Norm{\Matrix{X} - \Matrix{Y}_0}{\max}$ and $\varepsilon_{-} = 0$. Now, we can perform a variant of \textit{binary search} on the interval $(\varepsilon_{-}, \varepsilon_{+}]$. We pick the middle of the interval $\varepsilon = \frac{1}{2}(\varepsilon_{-} + \varepsilon_{+})$ and run the (Q)AP for $\mathcal{M}_r$ and the closed ball $B_\varepsilon(\Matrix{X}) = \set{\Matrix{Z} \in \Real^{m \times n}}{\Norm{\Matrix{X} - \Matrix{Z}}{\max} \leq \varepsilon}$. The metric projection onto $B_\varepsilon(\Matrix{X})$ can be computed as $\Proj{B_\varepsilon(\Matrix{X})}{\Matrix{Y}} = \Matrix{X} + \Proj{B_\varepsilon(\Matrix{0})}{\Matrix{Y} - \Matrix{X}}$ and amounts to clipping the large entries of $\Matrix{Y} - \Matrix{X}$. We stop the iterations when
\begin{equation*}
    \Norm{\Matrix{X} - \Matrix{Y}_{k}}{\max} < (1 + \delta) \cdot \Norm{\Matrix{X} - \Matrix{Y}_{k-1}}{\max}, \quad 0 < \delta < 1.
\end{equation*}
If the achieved approximation error $\Norm{\Matrix{X} - \Matrix{Y}_{k}}{\max}$ is not sufficiently close to $\varepsilon$, i.e.,
\begin{equation*}
    \Norm{\Matrix{X} - \Matrix{Y}_{k}}{\max} > \varepsilon_{-} + \tfrac{2}{3} (\varepsilon_{+} - \varepsilon_{-}) > \varepsilon_{-} + \tfrac{1}{2} (\varepsilon_{+} - \varepsilon_{-}) = \varepsilon,
\end{equation*}
we increase $\varepsilon_{-}$ by a small amount such as $(\varepsilon_{+} - \varepsilon_{-}) / 50$. Then, we update $\varepsilon_{+}$ with $\min\{ \varepsilon_{+}, \Norm{\Matrix{X} - \Matrix{Y}_{k}}{\max} \}$ and repeat the process for the new interval $(\varepsilon_{-}, \varepsilon_{+}]$ until it becomes small enough, at which point we conclude that $\inf_{\Matrix{Y} \in \mathcal{M}_r} \Norm{\Matrix{X} - \Matrix{Y}}{\max} \leq \varepsilon_{+}$.

\subsubsection{Orthogonal matrices}
We test these AP on two types of orthogonal matrices: random orthogonal matrices and identity matrices. First, we fix $n = 1600$ and vary the rank. The initial conditions $\Matrix{Y}_0$ are chosen as the product of two $n \times r$ random matrices with orthonormal columns in the random case and as the product of two $n \times r$ random Gaussian matrices with variance $1 / r$ for the identity matrix. Note that we cannot initialise the iterations for the identity matrix with its truncated SVD, because this leads to a loop. For the random orthogonal matrices, we perform 10 random experiments for every set of parameters and plot the median, the 25\% and the 10\% percentiles. For the identity matrices, we perform 5 random experiments and plot the minimum achieved error. The results in Fig.~\ref{fig:maximum_varying_r} suggest that the approximation error decays as $1 / \sqrt{r}$ for random orthogonal matrices and as $1 /r$ for the identity matrices.

\begin{figure}[h!]
\centering
	\includegraphics[width=0.9\textwidth]{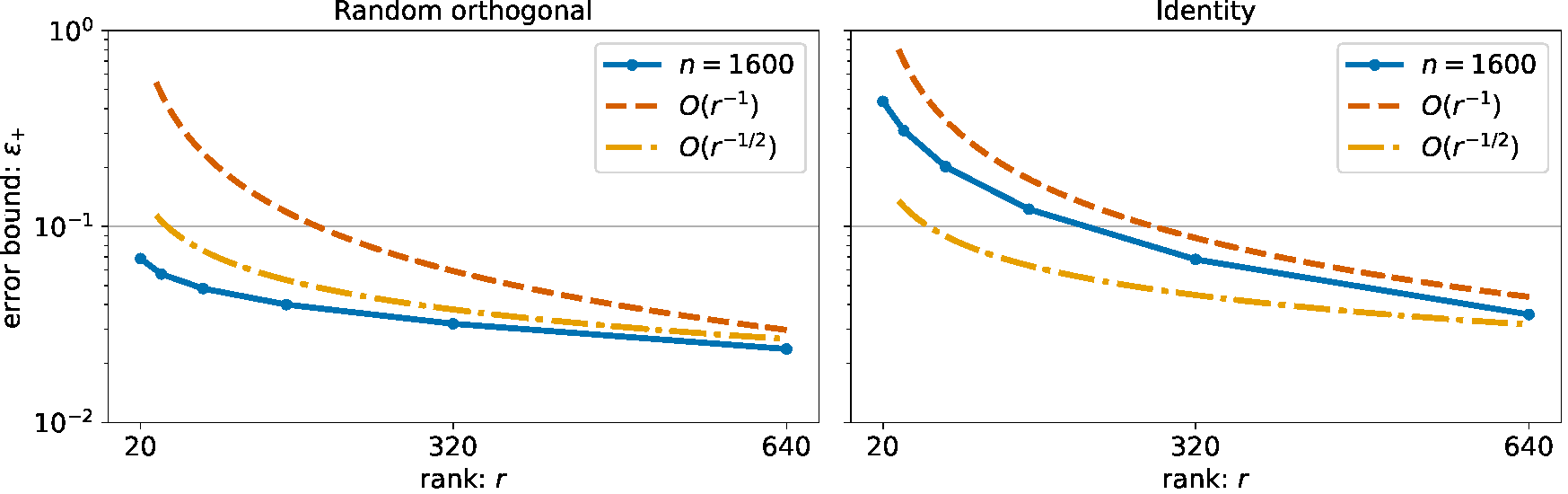}
\caption{Errors of rank-$r$ approximation in the maximum norm achieved with \textsc{ap-svd} for $1600 \times 1600$ matrices.}
\label{fig:maximum_varying_r}
\end{figure}

In the second series of experiments, we fix the rank $r = 40$ and consider matrices of different sizes. In Fig.~\ref{fig:maximum_varying_n}, we show how the approximation error achieved by \textsc{ap-rsvd} depends on the matrix size $n$. We observe that the error \emph{decays} as $\log(n) / \sqrt{n}$ for random orthogonal matrices. Meanwhile, the approximation errors for the identity matrices grow as $\log^2(n)$. 

\begin{figure}[h!]
\centering
	\includegraphics[width=0.9\textwidth]{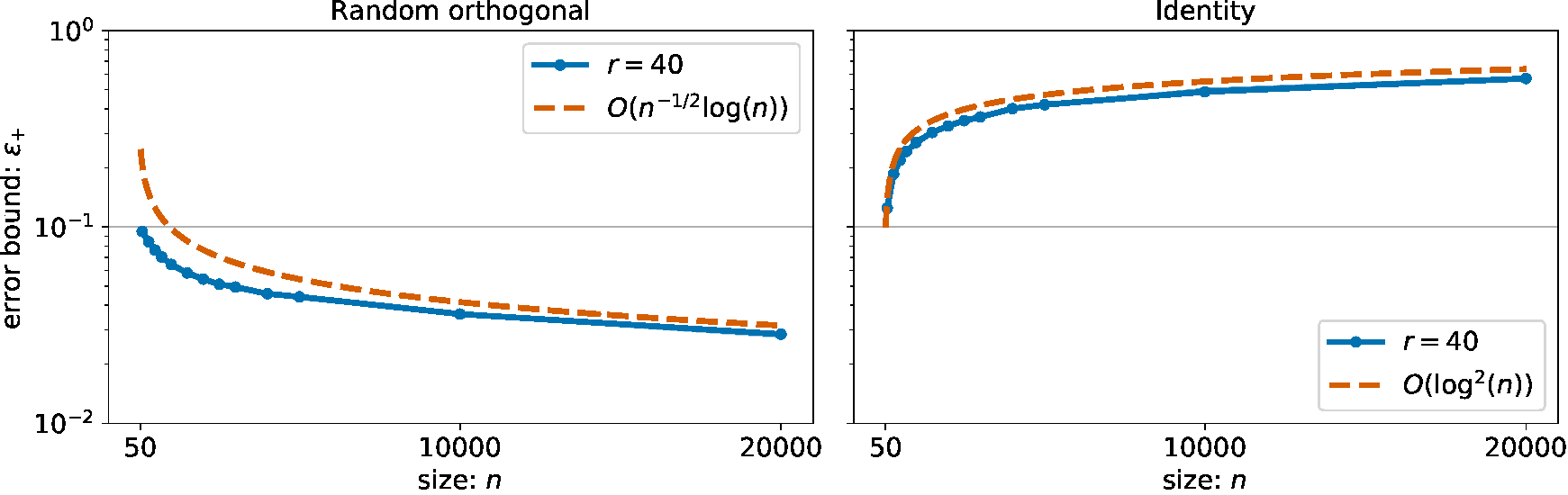}
\caption{Errors of rank-$40$ approximation in the maximum norm achieved with \textsc{ap-rsvd} for $n \times n$ matrices.}
\label{fig:maximum_varying_n}
\end{figure}

The two classes of orthogonal matrices exhibit qualitatively different behaviour. Both have spectral norms equal to one for every $n$, but their maximum norms are truly distinct. For the identity matrices, the maximum norm is also always equal to one. The maximum norm of an $n \times n$ random orthogonal matrix $\Matrix{X}$, upscaled by $\sqrt{n / \log(n)}$, converges in probability to 2 as $n \to \infty$ \cite{jiang2005maxima}. Thus, the smallest maximum-norm approximation error decays to zero since $\inf_{\Matrix{Y} \in \mathcal{M}_r} \Norm{\Matrix{X} - \Matrix{Y}}{\max} \leq \Norm{\Matrix{X}}{\max}$.

More numerical experiments are required to let us better understand how the above ultimate upper bound and the analytical bound of \cite{srebro2005rank} are related.\footnote{Since the first publication of our article, such experiments have been carried out in \cite{budzinskiy2024distance, budzinskiy2025entrywise, budzinskiy2025big}.}

\subsubsection{Quantised entries}
The maximum norm is also useful in approximating matrices and tensors, whose entries are stored in the fixed-point format. Every tensor with quantised entries can be seen as a rounded representation of a real-valued tensor, and our goal is to find one that is low-rank. Low-rank approximation of integer-valued matrices was considered in \cite{gillis2019low}, where a block-coordinate descent method was used to approximate $n \times n$ matrices that are formed as products of $n \times r$ standard Gaussian matrices and then rounded entrywise to the nearest integer. When $n = 200$ and the truncated SVD is chosen as the initial approximation, this method works in more than $90\%$ of random instances for $r \in \{ 1, 20 \}$, but never succeeds for $r \in \{ 2, 5, 10 \}$.

We repeat these experiments using the AP and set the initial value of $\varepsilon_{+}$ to $1/2$. The statistics over 20 random experiments are collected in Tab.~\ref{tab:tensor_integer}. Our approach succeeds every time for $r \in \{ 5, 10, 20 \}$ and achieves better approximation for $r = 20$ than \cite{gillis2019low}.

Next, we consider the case of $100 \times 100 \times 100$ tensors that are formed as TT decompositions of rank $\bm{r} = (r,r)$ with standard Gaussian random TT cores and rounded entrywise to the nearest integer. We apply \textsc{ap-ttsvd} to 20 random tensors, initialised each time with a \textsc{ttsvd}-approximation of rank $\bm{r}$ (see Tab.~\ref{tab:tensor_integer}). The AP succeed for $r \in \{10, 20\}$, but fail for $r = 5$. Similar to \cite{gillis2019low}, we observe that the higher the rank, the `easier' the problem is to solve and the lower the achievable errors are.

\begin{table}[ht!]
    \centering
    \caption{Approximation errors in the maximum norm for low-rank approximation of integer-valued matrices and tensors, achieved with alternating projections.}
    \label{tab:tensor_integer}
    \begin{tabular*}{\textwidth}{@{\extracolsep\fill}lcccccc@{}}
        \toprule
        & \multicolumn{3}{@{}c@{}}{Matrix} & \multicolumn{3}{@{}c@{}}{Tensor} \\\cmidrule{2-4}\cmidrule{5-7}%
        Rank & Minimum & Mean & Maximum & Minimum & Mean & Maximum \\
        \midrule
        5 & $0.491$ & $0.494$ & $0.498$ & $0.5007$ & $0.5011$ & $0.5017$ \\
        10 & $0.468$ & $0.473$ & $0.482$ & $0.4979$ & $0.4988$ & $0.4998$ \\
        20 & $0.433$ & $0.438$ & $0.444$ & $0.4851$ & $0.4859$ & $0.4876$ \\
        \bottomrule
    \end{tabular*}  
\end{table}
\section{Conclusion}
\label{sec:conclusion}
The main goal of the article was to shed more light on why the simple and versatile alternating projections work so well in practice, where the numerically computed metric projections are never exact. With low-rank approximation of matrices and higher-order tensors in mind as the model problem, we developed a convergence theory that explains how the benign local geometry at the intersection of two sets can compensate for the wildly inexact nature of certain approximate projections. Our theory can, in turn, pave the way for the use of alternating projections in a wider range of applications. This point is illustrated with non-trivial numerical results that we obtained for regularised matrix completion and low-rank approximation in the maximum norm.

An important question is whether the quasioptimal metric projection and the normal-cone condition \eqref{eq:inexact_angle} impose equivalent constraints; this is true for hyperplanes, but how regular does the set need to be for this equivalence to hold? Among the possible extensions of this work are (i)~the analysis of  convergence under weaker regularity assumptions as in \cite{bauschke2013restricted, noll2016local} and (ii)~the use of quasioptimal metric projections in the Douglas--Rachford algorithm \cite{phan2016linear, aragon2020douglas}.

\section*{Acknowledgements}
I am grateful to Sergey Matveev for our pleasant discussions about the topic: it is together that we came up with the idea of using cross approximation in low-rank alternating projections and saw its potential for regularising various iterative low-rank solvers. I thank Vladimir Kazeev and Dmitry Kharitonov for reading the early version of the text and suggesting improvements.

\begin{appendices}
    \section{Pythagorean property}
\label{appendix:basp}

Here, we refine Lemma~\ref{lemma:basp_superregular} by considering sets that are more regular than super-regular. A non-empty set $A \subseteq H$ is called \emph{prox-regular} at $\hat{a} \in A$ if there exist $t > 0$ and $\delta > 0$ such that $A \cap \ballrx{\delta}{\hat{a}}$ is closed in $\ballrx{\delta}{\hat{a}}$ and 
\begin{equation*}
    \dotp{v}{a' - a} \leq \tfrac{t}{2} \norm{a' - a}^2 \quad \text{for all} \quad a,a' \in A \cap \ballrx{\delta}{\hat{a}} \quad \text{and} \quad v \in \NormalLim{A}{a} \cap \ballrx{\delta}{0}.
\end{equation*}
What we take as the definition of prox-regularity is proved in \cite[Proposition 1.2]{poliquin2000local} to be equivalent to the original definition proposed in \cite{poliquin1996prox}. The class of prox-regular sets is large and includes smooth manifolds, convex sets, and super-regular sets. The following lemma is an analogue of Lemma~\ref{lemma:superregular_proj_nonexpansove} on the non-expansiveness of the projection.

\begin{lemma}
\label{lemma:proxregular_proj_nonexpansive}
Let $A \subseteq H$ be non-empty and prox-regular at $\hat{a} \in A$. There exist $t > 0$ and $\delta > 0$ such that for each $s \in (0, \min\{2,t\})$ and $x \in \ballrx{s\delta/2t}{\hat{a}}$ it holds that $\Proj{A}{x} = \{ \pi \}$ and
\begin{equation*}
    \norm{\pi - a} \leq \tfrac{2}{2 - s} \norm{x - a} \quad \text{for every} \quad a \in A \cap \ballrx{s\delta / 2t}{\hat{a}}.
\end{equation*}
\end{lemma}
\begin{proof}
See \cite[Proposition~3.1]{poliquin2000local}.
\end{proof}

\begin{lemma}
\label{lemma:basp_proxregular}
Let $A \subseteq H$ be non-empty and prox-regular at $\hat{a} \in A$. There exist $t > 0$ and $\delta > 0$ such that for each $x \in \ballrx{s\delta/2t}{\hat{a}} \setminus A$ with $s \in (0, \min\{2,t\})$ it holds that
\begin{equation*}
        \dotp[\bigg]{\frac{x - \pi}{\norm{x - \pi}}}{\frac{x - a}{\norm{x - a}}} \geq \frac{\norm{x - \pi}}{\norm{x - a}} - \frac{2s}{(2 - s)^2} 
\end{equation*}
for $\Proj{A}{x} = \{ \pi \}$ and every $a \in A \cap \ballrx{s\delta/2t}{\hat{a}}$.
\end{lemma}
\begin{proof}
    Repeat the proof of Lemma~\ref{lemma:basp_superregular} and use Lemma~\ref{lemma:proxregular_proj_nonexpansive} instead of Lemma~\ref{lemma:superregular_proj_nonexpansove}.
\end{proof}

Adding more regularity, we can obtain even tighter estimates. A non-empty set $A \subseteq H$ is called \emph{locally convex} at $\hat{a} \in A$ if there exists a neighbourhood $\mathcal{V}$ of $\hat{a}$ such that $A \cap \mathcal{V}$ is convex. It is known \cite[Proposition~1.5]{mordukhovich2006variational} that in this case 
\begin{equation*}
    \NormalLim{A}{\hat{a}} = \set{v \in H}{\dotp{v}{a' - \hat{a}} \leq 0 \quad \text{for all} \quad a' \in A \cap \mathcal{V}}.
\end{equation*}

\begin{lemma}
\label{lemma:basp_convex}
Let $A \subseteq H$ be non-empty, locally convex and locally closed at $\hat{a} \in A$. There exists $\delta > 0$ such that for each $x \in \ballrx{\delta/2}{\hat{a}} \setminus A$ it holds that
\begin{equation*}
        \dotp[\bigg]{\frac{x - \pi}{\norm{x - \pi}}}{\frac{x - a}{\norm{x - a}}} \geq \frac{\norm{x - \pi}}{\norm{x - a}}
\end{equation*}
for $\Proj{A}{x} = \{ \pi \}$ and every $a \in A \cap \ballrx{\delta}{\hat{a}}$.
\end{lemma}
\begin{proof}
Pick $\delta > 0$ such that $A \cap \ballrx{\delta}{\hat{a}}$ is convex and closed in $\ballrx{\delta}{\hat{a}}$. Then use the above formula for $\NormalLim{A}{\hat{a}}$ and note that $\Proj{A}{x} = \Proj{A \cap \ballrx{\delta}{\hat{a}}}{x}$ for $x \in \ballrx{\delta/2}{\hat{a}}$. Since $A \cap \ballrx{\delta}{\hat{a}}$ is convex, there can be at most one optimal metric projection onto it.
\end{proof}

\section{Proof details}
\label{appendix:implicit_function_theorem}

\begin{proof}[Proof of Lemma~\ref{lemma:locally_closed_local_proj}]
$\bm{1.}$ Let $x \in \mathcal{V}$ be a limit point of $A \cap \mathcal{V}$. Then $\dist{x}{A} = 0$ and there exists a sequence $(x_n) \subset A \cap \mathcal{V}$ that converges to $x$. By assumption, there is a point $a \in A$ such that $\dist{x}{a} = \dist{x}{A} = 0$. It follows that $x = a \in A$, but $x$ itself lies in $\mathcal{V}$ so $A \cap \mathcal{V}$ is closed in $\mathcal{V}$.

$\bm{2.}$ Pick $x \in \ballrx{r/2}{\hat{a}}$ and note that $\dist{x}{A} \leq \norm{x - \hat{a}} < r/2$. By definition, there is a minimising sequence $(a_n) \subset A$ such that $\norm{x - a_n} \to \dist{x}{A}$ as $n \to \infty$. Then
\begin{equation*}
    \limsup_{n \to \infty} \norm{\hat{a} - a_n} \leq \norm{x - \hat{a}} + \lim_{n \to \infty} \norm{x - a_n} < \tfrac{r}{2} + \dist{x}{A} < r.
\end{equation*}
By the Bolzano--Weierstrass theorem, there is a convergent subsequence $(a_{n_k}) \subset \ballrx{r}{\hat{a}}$. Since $A \cap \ballrx{r}{\hat{a}}$ is closed in $\ballrx{r}{\hat{a}}$, we have $\lim_{k \to \infty} a_{n_k} = a \in A \cap \ballrx{r}{\hat{a}}$. By continuity,
\begin{equation*}
    \norm{x - a} = \lim_{k \to \infty} \norm{x - a_{n_k}} = \dist{x}{A}
\end{equation*}
and $a \in \Proj{A}{x}$. Finally, for every projection $\pi \in \Proj{A}{x}$, we have
\begin{equation*}
    \norm{\pi - \hat{a}} \leq \norm{x - \pi} + \norm{x - \hat{a}} \leq 2 \norm{x - \hat{a}} < r. \qedhere
\end{equation*}
\end{proof}

\begin{proof}[Proof of Theorem~\ref{theorem:2qoptimal_convergence}: Implicit function theorem]
Consider a function
\begin{equation*}
    f(\varepsilon, \alpha, \beta) = 2 \varepsilon + \alpha \left(\tfrac{1}{\beta} - \varepsilon(1 + \varepsilon)\right) + \sqrt{1 - \alpha^2} \sqrt{1 - \left(\tfrac{1}{\beta} - \varepsilon(1 + \varepsilon)\right)^2}
\end{equation*}
and its partial derivatives
\begin{gather*}
    \frac{\partial f}{\partial \varepsilon} = 2 - \alpha(1 + 2 \varepsilon) + \sqrt{1 - \alpha^2} \frac{(1 + 2 \varepsilon)\left(\tfrac{1}{\beta} - \varepsilon(1 + \varepsilon)\right)}{\sqrt{1 - \left(\tfrac{1}{\beta} - \varepsilon(1 + \varepsilon)\right)^2}}, \\
    \frac{\partial f}{\partial \alpha} = \left(\tfrac{1}{\beta} - \varepsilon(1 + \varepsilon)\right) - \frac{\alpha}{\sqrt{1 - \alpha^2}} \sqrt{1 - \left(\tfrac{1}{\beta} - \varepsilon(1 + \varepsilon)\right)^2}.
\end{gather*}
Let $\hat{\alpha} \in (0,1)$, $\hat{\beta} > 1$, and put $\varepsilon = 0$. Then 
\begin{gather*}
    \frac{\partial f}{\partial \varepsilon}(0, \hat{\alpha}, \hat{\beta}) = 2 - \hat{\alpha} + \frac{\sqrt{1 - \hat{\alpha}^2}}{\sqrt{\hat{\beta}^2 - 1}} > 0, \quad
    \frac{\partial f}{\partial \alpha}(0, \hat{\alpha}, \hat{\beta}) = \frac{1}{\hat{\beta}} \left( 1 - \frac{\hat{\alpha}}{\sqrt{1 - \hat{\alpha}^2}} \sqrt{\hat{\beta}^2 - 1} \right).
\end{gather*}
We can apply the implicit function theorem to show that there is a $C^\infty$ function $\varepsilon(\alpha)$, defined in the neighbourhood of $\hat{\alpha}$, such that $f(\varepsilon(\alpha), \alpha, \hat{\beta}) = f(0, \hat{\alpha}, \hat{\beta})$ and 
\begin{equation*}
    \varepsilon(\hat{\alpha}) = 0, \quad \frac{d \varepsilon}{d \alpha}(\hat{\alpha}) = - \frac{\frac{\partial f}{\partial \alpha}(0, \hat{\alpha}, \hat{\beta})}{\frac{\partial f}{\partial \varepsilon}(0, \hat{\alpha}, \hat{\beta})}.
\end{equation*}
In addition, since $\alpha \mapsto \tfrac{\alpha}{\sqrt{1 - \alpha^2}}$ is strictly increasing, the derivative $\tfrac{\partial f}{\partial \alpha}(0, \hat{\alpha}, \hat{\beta})$ is positive if and only if $\hat{\alpha} \hat{\beta} < 1$. Thus, $\varepsilon(\alpha)$ is strictly decreasing if and only if $\hat{\alpha} \hat{\beta} < 1$.
\end{proof}    
\end{appendices}

\printbibliography
\end{refsection}

\begin{refsection}
\newpage
\renewcommand{\appendixtocname}{Supplementary materials}
\renewcommand{\appendixpagename}{Supplementary materials}
\begin{appendices}
    \section{Quasioptimal low-rank approximation: Details}
\label{appendix:lowrank}
\subsection{Randomised SVD}
The standard (and optimal in every unitarily invariant norm) way of computing a rank-$r$ approximation of a matrix is via the truncated SVD, which requires $\mathcal{O}(mn \cdot \min\{m,n\})$ operations. More efficient low-rank approximation algorithms of complexity $\mathcal{O}(\mathrm{nnz}(\Matrix{X}) \cdot r)$ can be achieved with randomisation, where $\mathrm{nnz}(\Matrix{X})$ is the number of non-zero entries of $\Matrix{X}$. One of the classics is the \emph{randomised range finder} \cite[Algorithm~4.1]{halko2011finding}: fix an oversampling parameter $p \geq 2$, generate an $n \times (r + p)$ random Gaussian matrix $\Omega$ and compute the QR decomposition of the product $\Matrix{QR} = \Matrix{X} \Omega$. It is then proved that \cite[Theorem~10.5]{halko2011finding}
\begin{equation*}
    \Expectation \Norm{\Matrix{X} - \Matrix{Q} \Matrix{Q}^\intercal \Matrix{X}}{F} \leq \sqrt{1 + \frac{r}{p-1}} \cdot \dist{\Matrix{X}}{\mathcal{M}_r}.
\end{equation*}
For $p = r + 1$ the estimate turns into $\Expectation \Norm{\Matrix{X} - \Matrix{Q} \Matrix{Q}^\intercal \Matrix{X}}{F} \leq \sqrt{2} \cdot \dist{\Matrix{X}}{\mathcal{M}_r}$. Note, however, that the matrix $\Matrix{Q} \Matrix{Q}^\intercal \Matrix{X}$ is of rank $r + p$ rather than $r$. To get a rank-$r$ approximation, we can compute the truncated SVD of $\Matrix{Q}^\intercal \Matrix{X}$. The error of this randomised SVD was analysed in \cite[Theorem~5.7]{gu2015subspace} for a modification of the randomised range finder that uses $q$ steps of power iterations: that is, the QR decomposition is computed for $(\Matrix{X} \Matrix{X}^\intercal)^q \Matrix{X} \Omega$. We present an expected approximation error that is less tight than the one in \cite{gu2015subspace}, but more suitable for our needs:
\begin{equation*}
    \Expectation \Norm{\Matrix{X} - \Matrix{Q} \cdot \Proj{\mathcal{M}_r}{\Matrix{Q}^\intercal \Matrix{X}}}{F} \leq \sqrt{1 + 144 e^2 (n-r) \frac{r(r+p)}{(p+1)^2} \left( \frac{\sigma_{r+1}}{\sigma_r} \right)^{4q}} \cdot \dist{\Matrix{X}}{\mathcal{M}_r}.
\end{equation*}
Here, $\sigma_{r+1} \leq \sigma_{r}$ are the singular values of $\Matrix{X}$. For $p = r - 1$ the estimate simplifies to
\begin{equation*}
    \Expectation \Norm{\Matrix{X} - \Matrix{Q} \cdot \Proj{\mathcal{M}_r}{\Matrix{Q}^\intercal \Matrix{X}}}{F} \leq \sqrt{1 + 288 e^2 (n-r) \left( \frac{\sigma_{r+1}}{\sigma_r} \right)^{4q}} \cdot \dist{\Matrix{X}}{\mathcal{M}_r}.
\end{equation*}
Depending on the spectral gap and the number of power iterations, the randomised SVD can be expected to compute arbitrarily good quasioptimal projections onto $\mathcal{M}_r$. Note that the complexity of the algorithm grows linearly with $q$.

\subsection{CUR and cross approximations}
An alternative way to construct low-rank approximations is to use the columns and rows of the matrix itself. Recall the interpolation property of rank-$r$ matrices:
\begin{equation*}
        \Matrix{X} = \Matrix{X}(:, J) \cdot \Matrix{X}(I, J)^{-1} \cdot \Matrix{X}(I, :).
\end{equation*}
The same idea can be used to compute a rank-$r$ approximation. Choose $k_c \geq r$ columns $\Matrix{C} = \Matrix{X}(:, J)$, $k_r \geq r$ rows $\Matrix{R} = \Matrix{X}(I,:)$, and a generator matrix $\Matrix{U} \in \Real^{k_c \times k_r}$. The following is called a CUR approximation:
\begin{equation*}
    \Matrix{X} \approx \Matrix{CUR}.
\end{equation*}
When the columns and rows are fixed, the optimal choice of the generator is \cite{stewart1999four}
\begin{equation*}
    \mathrm{argmin}_{\Matrix{U}} \Norm{\Matrix{X} - \Matrix{C U R}}{F} = \Matrix{C}^\dagger \Matrix{X} \Matrix{R}^\dagger,
\end{equation*}
where $\dagger$ stands for the Moore--Penrose pseudoinverse. If $\rank{\Matrix{C}^\dagger \Matrix{X} \Matrix{R}^\dagger} = \rank{\Matrix{X}}$, we recover the interpolation identity, since $\Matrix{C}^\dagger \Matrix{X} \Matrix{R}^\dagger = \Matrix{X}(I,J)^{-1}$ (see \cite[Theorem~A]{hamm2020perspectives}); otherwise, the two generators are different. The CUR approximation with a non-optimal generator $\Matrix{X}(I,J)^{\dagger}$ is known as the cross approximation:
\begin{equation*}
    \Matrix{X} \approx \Matrix{C} \cdot \Matrix{X}(I,J)^{\dagger} \cdot \Matrix{R}.
\end{equation*}

The approximation error achieved by the CUR and cross approximations depends on the choice of the index sets $I$ and $J$. It is shown in \cite[Theorem~1.3]{deshpande2006matrix} that every matrix $\Matrix{X}$ with $\rank{\Matrix{X}} \geq r$ contains $r$ columns $\Matrix{C} = \Matrix{X}(:, J)$ such that
\begin{equation*}
    \Norm{\Matrix{X} - \Matrix{C} \Matrix{C}^{\dagger} \Matrix{X}}{F} \leq \sqrt{1 + r} \cdot \dist{\Matrix{X}}{\mathcal{M}_r}.
\end{equation*}
Similarly, there are $r$ rows $\Matrix{R} = \Matrix{X}(I, :)$ that achieve
\begin{equation*}
    \Norm{\Matrix{X} - \Matrix{X} \Matrix{R}^{\dagger} \Matrix{R}}{F} \leq \sqrt{1 + r} \cdot \dist{\Matrix{X}}{\mathcal{M}_r}.
\end{equation*}
This proves the existence of a $\sqrt{2 + 2r}$-quasioptimal CUR approximation:
\begin{equation*}
    \Norm{\Matrix{X} - \Matrix{C} \Matrix{C}^{\dagger} \Matrix{X} \Matrix{R}^{\dagger} \Matrix{R}}{F}^2 = \Norm{\Matrix{X} - \Matrix{C} \Matrix{C}^{\dagger} \Matrix{X}}{F}^2 + \Norm{\Matrix{C} \Matrix{C}^{\dagger} \Matrix{X} - \Matrix{C} \Matrix{C}^{\dagger} \Matrix{X} \Matrix{R}^{\dagger} \Matrix{R}}{F}^2 \leq 2(1+r) \cdot \dist{\Matrix{X}}{\mathcal{M}_r}^2.
\end{equation*}
An analogous statement holds for the cross approximation \cite{zamarashkin2018existence}: there are index sets $I$ and $J$ of cardinality $r$ such that
\begin{equation}
\label{eq:cross_qoptimal}
    \Norm{\Matrix{X} - \Matrix{C} \cdot \Matrix{X}(I,J)^{-1} \cdot \Matrix{R}}{F} \leq (1 + r) \cdot \dist{\Matrix{X}}{\mathcal{M}_r}. 
\end{equation}
In both cases, the corresponding rows and columns can be computed in polynomial time \cite{cortinovis2020low}, though the complexities of $\mathcal{O}(mnr \cdot \min\{m,n\})$ and $\mathcal{O}(mnr \cdot \min\{m^2,n^2\})$ operations, respectively, are higher than for the truncated SVD.

Better approximation can be achieved when we oversample the rows and columns. It is possible to construct a rank-$r$ CUR approximation based on $k_c = k_r = k > 4r$ rows and columns that satisfies \cite[Theorem~7.1]{boutsidis2014optimal}
\begin{equation*}
    \Norm{\Matrix{X} - \Matrix{C U R}}{F} \leq \sqrt{1 + 80 \frac{r}{k - 4r}} \cdot \dist{\Matrix{X}}{\mathcal{M}_r}.
\end{equation*}
Therefore, we can reduce the approximation errors by taking larger $k$. The deterministic algorithm proposed in \cite{boutsidis2014optimal} has the complexity $\mathcal{O}(mnk \cdot \min\{m^2,n^2\})$, which is similar to the cost of the cross approximation algorithm from \cite{cortinovis2020low}. It is also shown in \cite[Theorem~8.1]{boutsidis2014optimal} that this quasioptimality estimate captures the correct relationship between the error and the oversampling: if $\Norm{\Matrix{X} - \Matrix{C U R}}{F} \leq \sqrt{1 + \varepsilon} \cdot \dist{\Matrix{X}}{\mathcal{M}_r}$ then $k_r$ and $k_c$ must scale as $r / \varepsilon$.

\subsection{Maximum-volume principle}
While the randomised SVD does indeed have a smaller computational complexity, the CUR- and cross-approximation algorithms that we have mentioned so far are slower than the optimal projection via the SVD.

The sampling strategy is the bottleneck. To decide which columns and rows to pick, each of them is assigned a specific score. And even though only the $\mathcal{O}(r)$ top scorers are chosen, the scores themselves are typically computed based on the left and right singular vectors. Hence, the SVD is required.

A different way of choosing the index sets $I$ and $J$ is based on maximising the \emph{volume} of $\Matrix{X}(I, J)$. The idea originates in the interpolation theory and leads to error estimates in the maximum norm \cite[Corollary~2.3]{goreinov2001maximal}: if the $r \times r$ matrix $\Matrix{X}(I, J)$ satisfies
\begin{equation*}
    |\det \Matrix{X}(I, J)| \geq \nu \cdot \max_{\hat{I}, \hat{J}} |\det \Matrix{X}(\hat{I}, \hat{J})|
\end{equation*}
for some $\nu \leq 1$ then the error of cross approximation is bounded by
\begin{equation}
\label{eq:quasimaxvol}
    \Norm{\Matrix{X} - \Matrix{C} \cdot \Matrix{X}(I,J)^{-1} \cdot \Matrix{R}}{\max} \leq \frac{1 + r}{\nu} \cdot \sigma_{r+1}(\Matrix{X}).
\end{equation}
Oversampling leads to better estimates \cite[Theorem~7]{osinsky2018pseudo}. If $\Matrix{X}(I,J) \in \Real^{k_r \times k_c}$ maximises the $r$-projective volume (the product of the $r$ largest singular values) among all the $k_r \times k_c$ submatrices of $\Matrix{X}$ then
\begin{equation*}
    \Norm{\Matrix{X} - \Matrix{C} \cdot \Proj{\mathcal{M}_r}{\Matrix{X}(I,J)}^{\dagger} \cdot \Matrix{R}}{\max} \leq t(k_r, r) t(k_c,r) \cdot \sigma_{r+1}(\Matrix{X}),
\end{equation*}
where
\begin{equation*}
    t(k, r) = \sqrt{1 + \frac{r}{k - r + 1}}.
\end{equation*}
Similar bounds hold in the spectral norm too \cite{goreinov1997theory, osinsky2018pseudo, mikhalev2018rectangular}.

In brief, cross approximation driven by the maximisation of volume is a powerful tool. Does it produce quasioptimal low-rank projections, though? It is known that the bound \eqref{eq:cross_qoptimal} does not always hold for the submatrices of maximum volume \cite{zamarashkin2018existence}. Nevertheless, for a class of random matrices $\Matrix{X}$ it was proved in \cite{zamarashkin2021accuracy} that the cross approximation based on the submatrix with the largest (generalised) $r$-projective volume delivers, in expectation, a $t(k_r, r) t(k_c,r)$-quasioptimal projection of $\Matrix{X}$ onto $\mathcal{M}_r$.

Another quasioptimality result can be found in \cite[Corollary~5.7]{hamm2021perturbations}. When the row-indices $I$ and column-indices $J$ are selected to maximise the volume of the $k_r \times r$ submatrix of the $r$ dominant left singular vectors and of the $k_c \times r$ submatrix of the $r$ dominant right singular vectors, respectively, the error of the cross approximation can be estimated as
\begin{equation*}
    \Norm{\Matrix{X} - \Matrix{C} \cdot \Proj{\mathcal{M}_r}{\Matrix{X}(I,J)}^{\dagger} \cdot \Matrix{R}}{F} \leq \frac{3}{2}\Big[ 1 + t(k_r, r) + t(k_c,r) + 3 t(k_r, r) t(k_c,r) \Big] \cdot \dist{\Matrix{X}}{\mathcal{M}_r},
\end{equation*}
provided that the spectral gap is large $\sigma_{r}(\Matrix{X}) \geq 4 t(k_r, r) t(k_c,r) \cdot \dist{\Matrix{X}}{\mathcal{M}_r}$.

Thus, there is non-negligible support to treat cross approximation based on the maximum-volume principle as a quasioptimal low-rank projection. It remains to compute the corresponding index sets. The problem is NP-hard \cite{civril2009selecting}, but 
\begin{itemize}
    \item the bound \eqref{eq:quasimaxvol} from \cite{goreinov2001maximal} suggests that a submatrix, whose volume is within a factor from the largest, suffices for a good approximation, and
    \item many of the results, including the estimates in \cite{osinsky2018pseudo} and the proof of quasioptimality for a class of random matrices from \cite{zamarashkin2021accuracy}, hold for a submatrix of locally maximum volume, where `locally' means that it is allowed to change one column and one row.
\end{itemize}
The \textsc{maxvol} algorithm of \cite{goreinov2010find} and its generalisations to rectangular submatrices \cite{osinsky2018rectangular, mikhalev2018rectangular} leverage both points and find $k_r$ rows in an $m \times r$ matrix such that the volume of the corresponding submatrix is within a prescribed factor $\nu$ from the locally maximum volume. The algorithm is iterative: the initialisation step takes $\mathcal{O}(m (k_r + r^2))$ operations, the cost of a single iteration is $\mathcal{O}(m k_r)$ operations, and the number of iterations depends on $\nu$ and the initial guess of the rows.

This is the base of fast cross approximation. In the case of a square submatrix, $k_r = k_c = r$, we can update the row-indices $I$ and column-indices $J$ in an alternating fashion:
\begin{equation}
\label{eq:maxvol_iterations}
    I \gets \textsc{maxvol}(\Matrix{X}(:, J), I), \quad J \gets \textsc{maxvol}(\Matrix{X}(I, :)^T, J).
\end{equation}
This procedure has the complexity of 
\begin{equation*}
    \mathcal{O}\Big((m + n) r^2 \cdot N_{\mathrm{iterations}} + (m + n)r \cdot N_{\mathrm{interchanges}}\Big),    
\end{equation*}
where $N_{\mathrm{iterations}}$ is the number of times the updates \eqref{eq:maxvol_iterations} are carried out and $N_{\mathrm{interchanges}}$ is the overall number of row-interchanges performed by \textsc{maxvol}.

The analogue for maximising the $r$-projective volume using $k_r \geq r$ rows and $k_c \geq r$ columns was proposed in \cite{osinsky2018rectangular}. There are two main differences. First, it is suggested to optimise $I$ and $J$ separately. To this end, two auxiliary index sets $I_r$ and $J_r$ of cardinality $r$ are needed. For the given $I$ and $J_r$, the first update is done as in \eqref{eq:maxvol_iterations}:
\begin{equation*}
    I \gets \textsc{maxvol}(\Matrix{X}(:, J_r), I).
\end{equation*}
The second novelty is that the auxiliary indices $J_r$ are not updated with \textsc{maxvol}; instead, we use the \textsc{dominant-r} algorithm \cite{osinsky2018rectangular}, which is a modification of the strong rank-revealing QR decomposition \cite{gu1996efficient}:
\begin{equation*}
    J_r \gets \textsc{dominant-r}(\Matrix{X}(I,:), J_r).
\end{equation*}
The asymptotic complexity of \textsc{dominant-r} is identical to \textsc{maxvol}: $\mathcal{O}(n(k_r + r^2))$ operations for the initialisation, $\mathcal{O}(n k_r)$ operations per iteration, and the number of interchanges, once again, depends on the quality of the initial guess of $J_r$ and on how close we want the volume to be to the locally maximum volume. Applying the same steps to $\Matrix{X}^\intercal$ in order to update $J$ and $I_r$, we get a submatrix of large $r$-projective volume; the complexity is
\begin{equation*}
    \mathcal{O}\Big((m + n) (k_r + k_c + r^2) \cdot N_{\mathrm{iterations}} + (m + n)(k_r + k_c) \cdot N_{\mathrm{interchanges}}\Big).  
\end{equation*}
The bounds of the approximation errors tell us that $k_r$ and $k_c$ of order $r$ are sufficient.

Both cross-approximation algorithms scale linearly with $m$ and $n$. They do not require access to the whole $\Matrix{X}$ by operating on submatrices of size $m \times r$ and $r \times n$. This means that we can keep $\Matrix{X}$ out of memory and evaluate the submatrices on the fly. If $\mathrm{elem}(\Matrix{X})$ is the cost of computing a single element of $\Matrix{X}$, the complexities---with $k_r = \mathcal{O}(r)$ and $k_c = \mathcal{O}(r)$---become
\begin{equation*}
    \mathcal{O}\Big((m + n) (r^2 + r \cdot \mathrm{elem}(X)) \cdot N_{\mathrm{iterations}} + (m + n)r \cdot N_{\mathrm{interchanges}}\Big).
\end{equation*}

The fact that only a small portion of adaptively chosen entries of a matrix make it possible to construct a good low-rank approximation is impressive, but it would be even more so for tensors, since it becomes prohibitive to store all $\prod_{k=1}^d n_k$ elements for large $d$. In \cite{oseledets2010tt}, the maximum-volume principle is extended to the tensor-train (TT) format. This leads to a fast cross-approximation algorithm that is based on the successive application of \textsc{maxvol} to the unfoldings of the tensor. The complexity of the proposed method is 
\begin{equation*}
    \mathcal{O}\Big( dn(r^3 + r^2 \cdot \mathrm{elem}(\Tensor{X})) \cdot N_{\mathrm{iterations}} + dnr^2 \cdot N_{\mathrm{interchanges}} \Big),
\end{equation*}
which scales linearly with the number of dimensions $d$ --- merely $\mathcal{O}(dnr^2)$ entries of the tensor are sampled. In addition, the dependence on the rank has increased by one order compared to the matrix case. The explanation is in the TT decomposition itself: $d-2$ of its TT cores are tensors of size $r \times n \times r$, while the remaining two are matrices. The approximation errors of the TT-cross approximation in the maximum norm were studied in \cite{savostyanov2014quasioptimality, osinsky2019tensor}.

\section{Low-rank nonnegative approximation: More tests}
\label{appendix:nn_tests}

One more option to modify the method of alternating projections
\begin{equation*}
    \Matrix{Z}_{k+1} = \Proj{\Real_{+}^{m \times n}}{\Matrix{Y}_k}, \quad \Matrix{Y}_{k+1} \in \Proj{\mathcal{M}_r}{\Matrix{Z}_{k+1}}, \quad k \in \N_0,
\end{equation*}
is to compute the projection onto the translate of $\Real_{+}^{m \times n}$ so that $\Matrix{Z}_{k+1}$ ends up in the interior. We write the method of \textit{indentation-projection} as
\begin{equation*}
    \alpha_k = \beta \cdot \Norm{\min(0, \Matrix{Y}_k)}{max}, \quad \Matrix{Z}_{k+1} = \Proj{\alpha_k + \Real_{+}^{m \times n}}{\Matrix{Y}_k}, \quad \Matrix{Y}_{k+1} \in \Proj{\mathcal{M}_r}{\Matrix{Z}_{k+1}}, \quad k \in \N_0.
\end{equation*}
Basically, $\Matrix{Z}_{k+1} = \max(\alpha_k, \Matrix{Y}_k)$ so that the negative entries of $\Matrix{Y}_k$ are pushed farther than by the shift-projection method, but the positive entries stay intact. Such iterations are known to converge in a finite number of steps for two convex sets \cite{rami2007finite}.

We repeat the experiments from the main text about the two-component Smoluchowski equation using indentation-projections with $\beta = 3/4$. The results in Fig.~\ref{fig:smolukh_modified_ap_extended} and Tab.~\ref{tab:smolukh_modified_ap_extended} show that indentation-projections improve the convergence rates for \textsc{svd}, \textsc{rsvd}, and $\textsc{pvol}_c$ while increasing the approximation error three times. The warm-started $\textsc{ip}(3/4)$-$\textsc{vol}_w$ and $\textsc{ip}(3/4)$-$\textsc{pvol}_w$ keep stagnating, and the cold-started $\textsc{ip}(3/4)$-$\textsc{vol}_c$ exhibits irregular behaviour.

\begin{figure}[h!]
\centering
	\includegraphics[width=0.9\textwidth]{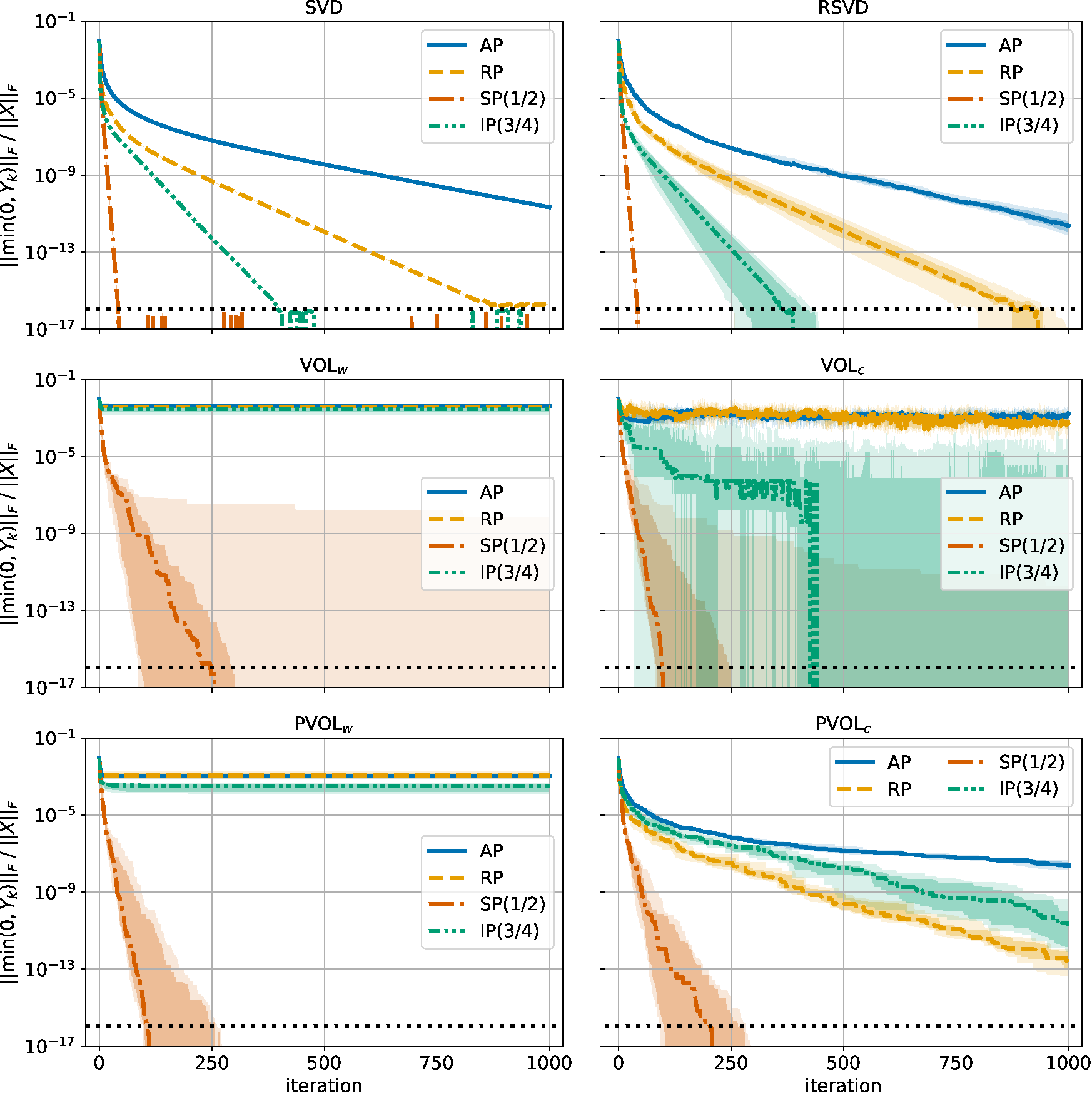}
\caption{Convergence of the low-rank iterates $\Matrix{Y}_k$ to the nonnegative orthant for the modified alternating-projection schemes with inexact low-rank projections: the median, the $25\%$ and the $10\%$ percentiles.}
\label{fig:smolukh_modified_ap_extended}
\end{figure}

\begin{table}[h!]
    \centering
    \caption{Median approximation errors of the low-rank iterates $Y_k$ after $10 / 1000$ steps of modified alternating-projection schemes with inexact low-rank projections.}
    \label{tab:smolukh_modified_ap_extended}
    \begin{tabular*}{\textwidth}{@{\extracolsep\fill}lcccccc@{}}
        \toprule
        Modification  & \textsc{svd} & \textsc{rsvd} & $\textsc{vol}_w$ & $\textsc{vol}_c$ & $\textsc{pvol}_w$ & $\textsc{pvol}_c$ \\
        \midrule
        \textsc{ap} & $1.13 / 1.13$ & $1.22 / 1.23$ & $1.10 / 1.10$ & $1.15 / 4.83$ & $1.14 / 1.14$ & $1.13 / 1.14$ \\
        \textsc{rp} & $1.19 / 1.19$ & $1.48 / 1.48$ & $1.33 / 1.33$ & $2.35 / 9.41$ & $1.24 / 1.24$ & $1.23 / 1.23$ \\
        $\textsc{sp}(1/2)$ & $5.78 / 5.78$ & $5.78 / 5.78$ & $5.55 / 5.76$ & $5.62 / 5.79$ & $5.60 / 5.79$ & $5.65 / 5.79$ \\
        $\textsc{ip}(3/4)$ & $3.11 / 3.11$ & $3.21 / 3.21$ & $4.41 / 4.86$ & $6.85 / 7.08$ & $3.11 / 3.11$ & $3.04 / 3.04$ \\
        \bottomrule\\
    \end{tabular*}
\end{table}

We also apply indentation-projections to compute a nonnegative QTT approximation of the mixture of Gaussian probability densities: the method converges in \emph{one} iteration and the approximation error grows $1.44$ times (see Fig.~\ref{fig:mixture_ip}).

\begin{figure}[h!]
\centering
	\includegraphics[width=0.9\textwidth]{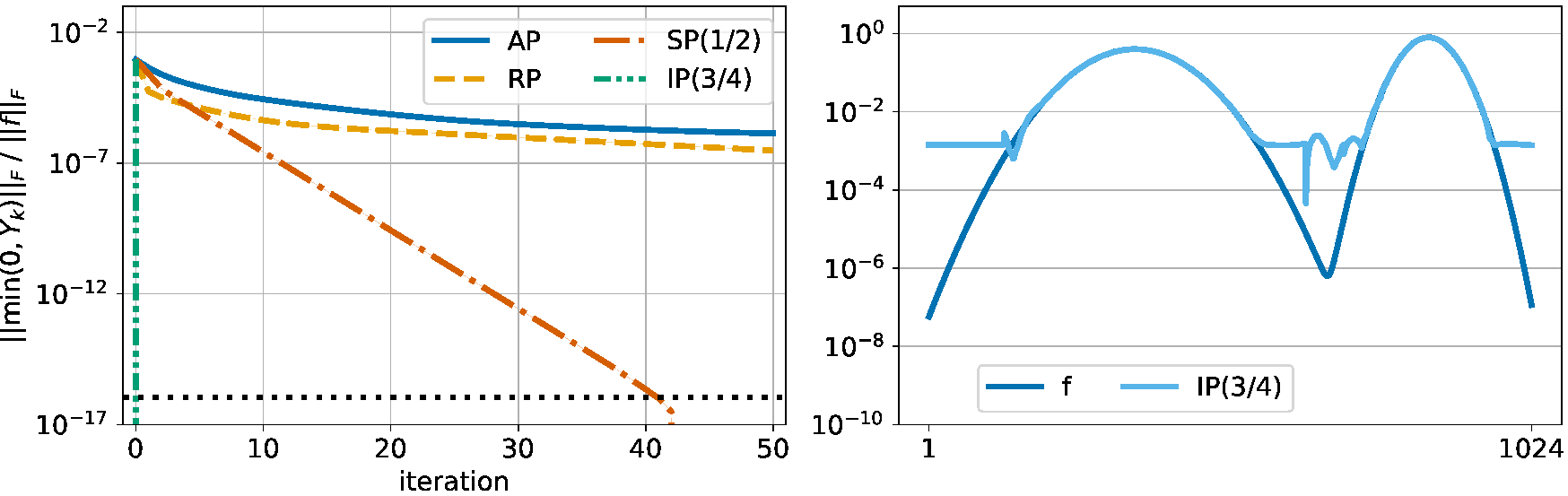}
\caption{The nonnegative QTT approximation obtained with the method of indentation-projections based on \textsc{ttsvd}.}
\label{fig:mixture_ip}
\end{figure}
\end{appendices}
\printbibliography[title={References for supplementary materials}]
\end{refsection}

\end{document}